\documentclass[12pt]{article}
\usepackage{amssymb,amsmath,amsthm,secdot,bbm,mathrsfs}

\usepackage[normalem]{ulem}
\usepackage[
bookmarks=true,
bookmarksnumbered=true,
colorlinks=true, pdfstartview=FitV, linkcolor=blue, citecolor=blue,
urlcolor=blue]{hyperref}

\usepackage[shortlabels]{enumitem}
\usepackage{color}

\usepackage[nameinlink,capitalize]{cleveref}

\usepackage{microtype}

\newcommand*{\Gref}[1]{\hyperref[SDEfunc]{Eq($#1$)}}
\newcommand*{\Gtag}[1]{\tag{Eq($#1$)}}

\DeclareMathOperator{\1}{\mathbbm{1}}%%
\DeclareMathOperator{\I}{\mathbbm{1}}%%

\DeclareMathOperator{\Law}{Law}%
\DeclareMathOperator{\sign}{sign}

\def\E{\hskip.15ex\mathsf{E}\hskip.10ex}
\def\P{\mathsf{P}}
\DeclareMathOperator{\Var}{Var}

\def\eps{\varepsilon}
\def\phi{\varphi}

\newcommand{\tand}{\quad\textrm{and}\quad}

\newtheorem{theorem}{Theorem}[section]
\newtheorem{lemma}[theorem]{Lemma}
\newtheorem{proposition}[theorem]{Proposition}
\newtheorem{corollary}[theorem]{Corollary}
\newtheorem{conjecture}[theorem]{Conjecture}

\theoremstyle{definition}
\theoremstyle{definition}\newtheorem{remark}[theorem]{Remark}
\theoremstyle{definition}\newtheorem{definition}[theorem]{Definition}

\crefname{Lemma}{Lemma}{Lemmas}
\crefname{Theorem}{Theorem}{Theorems}

\crefrangeformat{proposition}{Propositions~#3#1#4--#5#2#6}

\numberwithin{equation}{section}

\renewcommand{\ge}{\geqslant}
\renewcommand{\le}{\leqslant}

\newcommand{\nn}{\nonumber}

\newcommand{\wt}{\widetilde}
\newcommand{\wh}{\widehat}

\newcommand{\qh}{\widehat{q}}

\newcommand{\0}{{D}}

\newcommand{\A}{\mathcal{A}}
\newcommand{\B}{\mathcal{B}}
\newcommand{\Ba}{\mathbf{B}}
\newcommand{\C}{\mathcal{C}}

\newcommand{\F}{\mathcal{F}}

\newcommand{\G}{\mathcal{G}}

\newcommand{\M}{\mathcal{M}}
\newcommand{\N}{\mathbb{N}}
\newcommand{\R}{\mathbb{R}}

\newcommand{\Z}{\mathbb{Z}}

\def\({\lt(}
\def\){\rt)}

\newcommand{\lt}{\left}
\newcommand{\rt}{\right}

\newcommand{\Bes}{\mathcal{B}} %Besov
\newcommand{\bes}{\Bes} %Besov

\DeclareMathOperator{\Leb}{Leb}
\newcommand{\var}{\textrm{var}}

\newcommand{\Di}[1]{|#1|}
\definecolor{Brown}{rgb}{.75,.5,.25}
\definecolor{DGreen}{rgb}{0,0.55,0}
\definecolor{Olive}{rgb}{0.41,0.55,0.13}

\definecolor{Red}{rgb}{1,0,0}
\definecolor{Blue}{rgb}{0,0,1}
\definecolor{Olive}{rgb}{0.41,0.55,0.13}
\definecolor{Yarok}{rgb}{0,0.5,0}
\definecolor{Green}{rgb}{0,1,0}
\definecolor{MGreen}{rgb}{0,0.8,0}
\definecolor{DGreen}{rgb}{0,0.55,0}
\definecolor{Yellow}{rgb}{1,1,0}
\definecolor{Cyan}{rgb}{0,1,1}
\definecolor{Magenta}{rgb}{1,0,1}
\definecolor{Orange}{rgb}{1,.5,0}
\definecolor{Violet}{rgb}{.5,0,.5}
\definecolor{Purple}{rgb}{.75,0,.25}
\definecolor{Brown}{rgb}{.75,.5,.25}
\definecolor{Grey}{rgb}{.7,.7,.7}
\definecolor{Black}{rgb}{0,0,0}

\begin{document}

\makeatletter
\let\@fnsymbol\@arabic
\makeatother
	
\title{Stochastic equations with singular drift driven by  fractional Brownian motion}
	
	\author{
		Oleg Butkovsky%
		\thanks{Weierstrass Institute, Mohrenstrasse 39, 10117 Berlin, FRG. Email: \texttt{oleg.butkovskiy@gmail.com}}$^{\,\,\,,\!\!}$
		\thanks{Institut für Mathematik, Humboldt-Universität zu Berlin, Rudower Chaussee 25,
			12489, Berlin,	FRG.}
		\and				
		Khoa  L\^e
		\thanks{School of Mathematics, University of Leeds, U.K. Email: \texttt{k.le@leeds.ac.uk}}
		\and
		Leonid Mytnik%
		\thanks{Technion --- Israel Institute of Technology,
			Faculty of Data and Decision Sciences,
			Haifa, 3200003, Israel.  Email: \texttt{leonid@ie.technion.ac.il}
		}
	}
	
	\maketitle

\begin{abstract}
We consider stochastic differential equation
$$
d X_t=b(X_t) dt +d W_t^H,
$$
where the drift $b$ is either a measure or an integrable function, and $W^H$ is a $d$-dimensional fractional Brownian motion with Hurst parameter $H\in(0,1)$, $d\in\N$. For the case where $b\in L_p(\R^d)$, $p\in[1,\infty]$ we show weak existence of solutions to this equation under the   condition 
$$
\frac{d}p<\frac1H-1,
$$
which is an extension of the Krylov--R\"ockner condition (2005) to the fractional  case. We construct a counter-example showing optimality of this condition. If $b$ is a Radon measure, particularly the delta measure, we prove weak existence of solutions to this equation  under the optimal  condition $H<\frac1{d+1}$. We also show strong well-posedness of solutions to this equation under certain conditions. 
To establish these results, we utilize the stochastic sewing technique and develop a new version of the stochastic sewing lemma.
\end{abstract}

\section{Introduction}
Stochastic differential equations (SDEs) driven by fractional Brownian motion (fBM) are commonly used in  mathematical biophysics, where fBM is considered a canonical model for anomalous diffusion, see, e.g., \cite{phy0,phy3,phy_JM,phy2}. In this article we study  multidimensional SDE 
\begin{align}\label{mainSDE}
&dX_t=b(X_t)dt + dW_t^H,\quad t\in[0,T]\\
&X_0=x\in\R^d\nn,
\end{align}
where $d\in\N$, $T>0$, $b$ is a finite signed Radon measure or $b\in L_p(\R^d,\R^d)$, $W^H$ is an fBM with the Hurst parameter $H\in(0,1)$, and the initial condition $x\in\R^d$. We show weak existence of solutions to this equation in the optimal regime (\cref{T:func,T:ident}) and construct a counter-example to prove optimality of our assumptions (\cref{T:genopt}). For $d=1$ we obtain strong well-posedness under certain conditions (\cref{T:uniq}). Additionally, we provide a connection between equation \eqref{mainSDE} and a stochastic equation involving local time (\cref{T:measure}). As a byproduct of our results, we deduce a short direct proof of classical results of Geman--Horowitz \cite{MR556414} and Xiao \cite{Xiao} that if $Hd<1$, then $W^H$ has a local time of certain space-time regularity. We also show that a sum of a fractional Brownian motion and an adapted process of bounded variation has a jointly continuous local time under the optimal condition $H(d+1)<1$ (\cref{T:loc}). To prove these results, we further develop the stochastic sewing technique initiated by \cite{LeSSL} and obtain a  Rosenthal-type stochastic sewing lemma \cref{T:RoSSL} which relaxes the assumptions on high moments of increments of the process. 

It has been known for a long time that an ill-posed deterministic system can become well-posed when randomly perturbed. Indeed, the differential equation  $dX_t=\sign(X_t)|X_t|^\alpha dt$, $\alpha\in(0,1)$, $X_0=0$ has infinitely many solutions. On the other hand, the perturbed  equation $dX_t=\sign(X_t)|X_t|^\alpha dt+dW_t$, where $W$ is a Brownian motion, has a unique solution. This phenomenon is called regularization by noise; we refer the reader to monograph \cite{F11} for many interesting examples and further discussion.

Regularization by noise is particularly well-studied in the case of Brownian noise $(H=\frac12)$, and many sharp results have been obtained. The pioneering works of Zvonkin and Veretennikov \cite{zvonkin74,ver80} showed that if the drift $b$ is a measurable bounded function, then the SDE
\begin{equation}\label{mainSDEBM}
dX_t=b(X_t)dt + dW_t,\quad t\in[0,T],
\end{equation}
where $W$ is a Brownian motion, has a unique strong solution. This was later extended  by Krylov and R\"ockner \cite{kr_rock05} to unbounded integrable drifts $b\in L_q([0,1],L_p(\R^d))$, $p\ge2$, under the condition $\frac2q+\frac{d}p<1$. The critical case has been treated in \cite{BC2005} ($d=1$) and \cite{krylov2020strong} ($d\ge3)$. It is also known that the Krylov--R\"ockner condition is optimal: if $\frac{d}p>1$, then one can construct a drift $b\in L_p(\R^d)$ such that \eqref{mainSDEBM}  has no solutions \cite[Example~7.4]{BFGM}, \cite[Example~1.1]{krylov21}. Further generalizations can be found in \cite{Zhang11}.

The case where the drift $b$ is a Radon measure has also been thoroughly studied. In this case, the term $b(X_t)$ in \eqref{mainSDEBM} is not defined, and one needs to explain what is meant by a solution to this equation. The main idea is that, instead of defining $b(X_t)$ at each fixed $t$, one directly defines the integral $\int_0^t b(X_r)dr$ either via the local time of $X$ or via approximations of $b$, see \cite{LG84, BC} and the discussion in \cref{S:MR}. It is known from \cite{LG84} and \cite[Theorem~4.5]{BC2005} that if $d=1$ and $b$ is a finite signed Radon measure such that $|b(\{x\})|\le1$ for each $x\in\R$, then \eqref{mainSDEBM} has a unique strong solution. This condition is also optimal: if the measure $b$ has atoms of size larger than $1$, then \eqref{mainSDEBM} has no solutions, see \cite[Theorem~3.2]{BC2005} and \cite[p.~312]{HS81}.

Regularization by noise phenomena outside from martingale contexts such as SDEs driven by fractional Brownian motion and stochastic partial differential equations are much less understood. This is not because regularization by noise  only occurs in the Brownian case, but rather because there are very few tools available to study this problem in other setups. Indeed, all the proofs in the aforementioned articles heavily rely on It\^o calculus and PDE techniques;  these instruments are largely inapplicable beyond the semimartingale setting. 
Recall that the rougher the driving noise $W$ is in \eqref{mainSDEBM}, the more irregular drift $b$ one can take so that \eqref{mainSDEBM} remains well-posed, see \cite{CG16} for a collection of results highlighting this phenomenon. The idea behind this is that if the fluctuations of the noise are bigger than the fluctuations of the drift $\int b(X_r) dr$, then the noise averages out singularities of the drift $b$ 
so that the map $\eta\mapsto\int b(W+\eta)dr$ is Lipschitz in suitable metrics. 
 % helps the solution to leave singularities, 
% and equation \eqref{mainSDEBM} should have a unique solution. 
If the noise is not irregular enough, then we give an explicit example showing that \eqref{mainSDEBM} may have no solutions, see \cref{T:genopt}.
% Otherwise,  equation \eqref{mainSDEBM} might have no or multiple solutions. We provide a justification of this informal principle by showing that if the noise $W$ is not irregular enough, then \eqref{mainSDEBM} might have no solutions, see \cref{T:genopt}. 
On the other spectrum, \cite[Theorem~1.2]{HP20} constructs a random continuous forcing $W$, which is so irregular that equation \eqref{mainSDEBM} has a unique solution for \textit{any} Schwartz distribution $b$.  We refer the reader to \cite[Section~1]{GG22} and \cite[end of Section~1]{CdR18} for further discussion of this  principle.

Therefore, an important objective is to generalize sharp results on regularization by Brownian noise to non-semimartingale noises while maintaining optimality.

To understand what the optimal condition for well-posedness of \eqref{mainSDE} should be, we apply the following heuristics. Clearly, $W^H\in\C^{H-\eps}$. This implies that $X\in \C^{H-\eps}$. Further, 
if $b\in L_p(\R^d)$, $p\in[1,\infty]$, then $b\in\C^{-\frac{d}p}$ and very informally we would get $b(X_\cdot)\in \C^{-\frac{(H-\eps)d}p}$. We gain $1$ in regularity by integration, and hence we see that the drift of \eqref{mainSDE} $\int b(X_r)\,dr\in
\C^{1-\frac{(H-\eps)d}p}$. By the above principle, we expect to have  well-posedness if 
\begin{equation}\label{optimalineq}
1-\frac{Hd}p>H.
\end{equation}
As a sanity check, for the Brownian case $H=\frac12$, \eqref{optimalineq} becomes $\frac{d}p<1$ and coincides with the Krylov--R\"ockner condition. Another version of the same heuristic which uses the scaling argument and leads to \eqref{optimalineq} can be found in \mbox{\cite{GG22}}.

Regularization by fractional Brownian noise was initiated by Nualart and Ouknine in \cite{OuNu,NO03}.
This problem sparks new interests recently, with new different approaches \cite{CG16,ART21,GG22}. Unfortunately, these works do not allow to deduce well-posedness for $b\in L_p(\R^d)$ or for $b$ being a Radon measure under condition \eqref{optimalineq}.

Our main result is the weak existence of solutions to \eqref{mainSDE} under the above condition.  In the case where $b$ is a measure, we also obtain weak existence under the corresponding analogue of \eqref{optimalineq}. We construct a counterexample that shows the optimality of \eqref{optimalineq}.
Further, we show that, under certain assumptions, a solution to  \eqref{mainSDE} has a jointly continuous local time. This extends to the case of irregular drifts the corresponding results obtained in \cite{LouCheng17,SSV22} where existence of local times of solutions to SDEs driven by fBM with smooth drifts was established.

In the particular case where $b = \delta_0$, with $\delta_0$ being the delta measure at $0$,  equation \eqref{mainSDE} becomes
\begin{equation}\label{sFBM}
	dX_t=\delta_0(X_t)dt + dW_t^H,\quad t\in[0,T]
\end{equation}
and its solution is called \textit{skew fractional Brownian motion} (sfBM). Strong well-posedness of sFBM is known for $H=\frac12$ (\cite{HS81}) and $H\le\frac14$ (\cite{CG16,ART21}). This leaves the interval $H\in(\frac14,\frac12)$ where well-posedness of \eqref{sFBM} is expected but has not yet been proved. In this article we reduce this gap. We show weak existence of solutions to \eqref{sFBM} in the whole range $H<\frac12$ and strong unqiueness for $H<(\sqrt13-3)/2$. The remaining part --- strong uniqueness for $H\in[(\sqrt13-3)/2,1/2)$ --- poses a significant challenge that would likely require the development of vastly different techniques. 

Now, let us say a few words about our proof methods. As we explained above, well-posedness for \eqref{mainSDE} with $H=\frac12$ has been analyzed mostly using PDE methods and It\^o's formula. A breakthrough in studying \eqref{mainSDE} for $H\neq\frac12$ has been achieved in \cite{CG16} where deterministic sewing was applied. \cite{LeSSL} introduced a new technique called \textit{stochastic sewing} and applied it to analyze \eqref{mainSDE}. Further developments of stochastic sewing and its application to studying equation \eqref{mainSDE} were made in \cite{ART21,GG22,MP22,Gerreg22}. We also apply stochastic sewing, however compared with the above works, our proof strategy has the following new ingredients which allowed us to get optimal results on the existence of weak solutions.

First, as discussed above, the drift $\int b(X_r)\, dr$ has H\"older regularity $1-\frac{Hd}p-\eps<1$. On the other hand, if we measure its regularity in the variation scale, then it is easy to see that this drift is always of finite $1$-variation. 
By combining this insight with the idea of random control, which was introduced in \cite{ABLM}, we improve the corresponding powers in stochastic sewing from  $1-\frac{Hd}p-\eps$ to $1$, see the proof of \cref{L:driftb2}.

Next, we replace the Burkholder--Davis--Gundy inequality in stochastic sewing lemma with the Rosenthal--Burkholder inequality. This yields a new version of stochastic sewing lemma, \cref{T:RoSSL}. Morally, it says that instead of bounding $m$th moment (where $m$ is large) of the difference $\delta A_{s,u,t}$, it is enough to have a good bound only on the second conditional moment $\E [\delta A_{s,u,t}^2|\F_u]$. \cref{R:highmom} explains why this improvement is crucial for getting weak existence in the optimal regime. 

Finally, we also utilize the quantitative John--Nirenberg inequality \cite{le2022} and taming singularities technique \cite{BFG,le2021taming}. \cref{R:newthings} elaborates why this leads to better results than just the application of stochastic sewing alone.

We have made significant effort to avoid using Girsanov's transform. This is because we do not expect Girsanov's theorem to be applicable over the entire range of $p\in[1,\infty]$, $H\in(0,1)$, $d\in\N$, where \eqref{optimalineq} holds. Therefore, to extend the corresponding bounds from fBM to perturbed fBM, we rely solely on stochastic sewing.
A parallel project \cite{BLM23} focuses on the range of $p$, $H$,  $d$ for which Girsanov's theorem is applicable. It establishes the strong existence and uniqueness of solutions to \eqref{mainSDE} for these values of the parameters (a subset of the range of parameters for which \eqref{optimalineq} holds). Weak uniqueness for SDE \eqref{mainSDE} for $b$ being a Schwartz distribution was considered in recent work \cite{BM24}.

The fact that our approach does not rely on Girsanov's formula or It\^o's lemma makes it quite robust. We believe that it can be transferred to other types of noise, see \cref{R:levy}. In particular, it can be used to study regularization by noise and local times for stochastic partial differential equations, potentially complementing  the results from \cite{ABLM,Zambottibspde}.

The rest of the paper is organized as follows. Our main results are presented in \cref{S:MR}. We collect our toolkit including Rosenthal's type stochastic sewing lemma in \cref{sec:tools}. \cref{S:key} contains key bounds needed for the proofs, while proofs of the main results are postponed till \cref{sec:proof}. Auxiliary technical results on Besov spaces as well as necessary heat kernel bounds are put in Appendix.
  
\textbf{Convention on constants}. Throughout the paper $C$ denotes a positive constant whose value may change from line to line; its dependence is always specified in the corresponding statement.

\textbf{Convention on integrals}. In this paper, all integrals with respect to the deterministic measure are understood in the Lebesgue sense.

\textbf{Acknowledgements}. The authors are very grateful to Peter Friz, Lucio Galeati, M\'at\'e Gerencs\'er and Konstantinos Dareiotis   for numerous interesting discussions about regularization by noise for SDEs driven by fractional Brownian motion and to Willem van Zuijlen for a very helpful discussion about different aspects of theory of Besov spaces. We would also like to thank the referees for thoroughly reading the paper and offering very valuable suggestions. OB has received funding from the DFG Research Unit FOR~2402, Deutsche Forschungsgemeinschaft (DFG, German Research Foundation) under Germany's Excellence Strategy --- The Berlin Mathematics Research Center MATH+ (EXC-2046/1, project ID: 390685689, sub-project EF1-22), DFG CRC/TRR 388 “Rough Analysis, Stochastic Dynamics and Related Fields”, Project B08. LM is supported in part by ISF grant No. ISF 1985/22. Significant progress on the project has been achieved during the visits of the authors to Technion, ICMS Edinburgh (research in group stay in May 2022), and Mathematisches Forschungsinstitut Oberwolfach (mini-workshop 2207c, February 2022). We would like to thank all these institutions 
for providing excellent accommodation and very good working conditions.

\section{Main results}\label{S:MR}

We begin by introducing the basic notation. Let $d \in \mathbb{N}$. Let $\C_b^\infty=\C^\infty_b(\R^d,\R^d)$ be the space of infinitely
differentiable real functions $\R^d\to\R^d$ which are bounded and have bounded derivatives of all orders. Let $\M(\mathbb{R}^d, \mathbb{R}^d)$ denote the set of $d$-dimensional signed finite Radon measures on $\mathbb{R}^d$. Let $\M_+(\mathbb{R}^d, \mathbb{R}^d) \subset \M(\mathbb{R}^d, \mathbb{R}^d)$ be the set of non-negative $d$-dimensional finite Radon measures on $\mathbb{R}^d$. Let $\B^{\beta}_p = \B^{\beta}_{p,\infty}(\R^d)$ be the Besov space of regularity $\beta \in \mathbb{R}$ and integrability $p \in [1, \infty]$. If $\beta<0$, then we write $\C^\beta:=\B^\beta_{\infty,\infty}$. 
For $\beta\in(0,1]$, $d,k\in\N$, $Q\subset \R^d$, let   $\C^\beta (Q,\R^k)$ be the space of all functions $f\colon Q\to \R^k$,  such that \begin{equation*}
\|f\|_{\C^\beta(Q,\R^k)}:=\sup_{x\in Q}|f(x)|+[f]_{\C^\beta(Q,\R^k)}<\infty, \text{\,\, where\,\,} [f]_{\C^\beta(Q,\R^k)}:=\sup_{x,y\in Q, x\neq y}\frac{|f(x)-f(y)|}{|x-y|^\beta}.
\end{equation*}
We will also use the space $\C^\beta_0(Q,\R^k)=\C^\beta(Q,\R^k)\cap \{f\colon Q\to\R^k, f(0)=0\}$ of H\"older functions which are zero at zero. Often, for brevity, we will write $\C^\beta(Q):=
\C^\beta(Q,\R^k)$ if there is no ambiguity.

We will frequently use in the paper the following elementary embeddings   without explicitly mentioning them (\cite[Proposition~2.39]{bahouri}, \cite[Proposition~2.1, Theorem~2.5]{Besovbook}):
\begin{equation*}
\M(\R^d,\R^d)\subset\B^0_1,\quad L_p(\R^d,\R^d)\subset \B^0_p,\quad \B^\alpha_p\subset \B^{\alpha-\frac{d}{p}+\frac{d}{q}}_q,
\end{equation*}
where $p\in[1,\infty]$, $\alpha\in\R$, $q\in[p,\infty]$. The Lebesgue measure on $\R^d$ will be denoted by $\Leb$. An open ball in $\R^d$ centered at $x\in\R^d$ of radius $r$ will be denoted by $\Ba(x,r)$. The volume of the unit ball in $\R^d$ will be denoted by 
\begin{equation}\label{nud}
v_d:=\Leb(\Ba(0,1)).
\end{equation}
 For $x\in\R^d$ we will denote the delta measure at $x$ by $\delta_x$.

For a function $f\colon [s,t]\to\R^d$, where $0\le s\le t$, define its $\ell$-variation, $\ell\in[1,\infty)$ by
$$
\|f\|_{\ell-\var;[s,t]}:=\bigl(\sup_\Pi\sum_{i=0}^{n-1} |f(t_{i+1}-f(t_i)|^\ell\bigr)^{\frac1\ell},
$$
where the supremum is taken over all partitions $\Pi=\{t_0=s,t_1,\hdots, t_n=t\}$ of interval $[s,t]$. The space of functions $f\colon [s,t]\to\R^d$ with finite $\ell\in[1,\infty)$ variation will be denoted by $\C^{\ell-\var}([s,t],\R^d)$.

Denote by $p_t$, $t>0$, the density of a $d$-dimensional vector with independent Gaussian components each of mean zero and variance $t$:
$$
p_t(x) =(2\pi t)^{-d/2}e^{-\frac{|x|^2}{2t}},\quad  x\in\R^d,
$$
and let $P_t$ be the corresponding Gaussian semigroup.

% Fix the time interval $[0,T]$, $T>0$, $d\in\N$, and $H\in(0,1)$. 
We recall from \cite[Section~5.1.3, formula~(5.8) and Proposition~5.1.3]{Nu} that given a $d$-dimensional fractional Brownian motion $W^H$ one can construct on the same probability space a standard Brownian motion $B$ such that
\begin{equation}\label{WB}
W^H_t=\int_0^t K_H(t,s)\,dB_s=:\Psi_t(B),\quad t\in[0,T],
\end{equation}
where for $0\le s \le t \le T$
\begin{align}
&K_H(t,s):=C(H,d)s^{\frac12-H}\int_s^t (r-s)^{H-\frac32}r^{H-\frac12}\,dr\quad\text{when}\quad H>1/2,\label{Hl12}\\
&K_H(t,s):=C(H,d)\Bigl(t^{H-\frac12}s^{\frac12-H}(t-s)^{H-\frac12}
\nonumber\\&\quad\quad \quad \quad \quad+(\frac12-H)s^{\frac12-H}\int_s^t (r-s)^{H-\frac12}r^{H-\frac32}\,dr\Bigr)\quad\text{when}\quad H<1/2,\label{Hg12}
\end{align}
and $C(H,d)$ is a certain positive constant. 

Let $(\Omega, \F,\P)$ be a probability space and let $(\F_t)_{t\in[0,T]}$ be a filtration on this space such that $\F_0$ contains all null sets.
As in \cite[Definition~1]{OuNu}, we say that $W^H$ is a \textit{$(\F_t)$-fractional Brownian motion}, if there exists an $(\F_t)$-Brownian motion $B$ such that \eqref{WB} holds. Since the natural filtrations generated by $W$ and $B$ in \eqref{WB} coincide (\cite[Section 5.1.3]{Nu}), we see that $W^H$ is an $(\F^{W^H}_t)$-fractional Brownian motion.

As mentioned earlier, we consider SDE \eqref{mainSDE} in two different settings. First, if $b$ is a measurable function $\R^d\to\R^d$, then, as usual, we say that a continuous process $X$ is a solution to \eqref{mainSDE} with the initial condition $x\in\R^d$ if a.s.
\begin{equation*}\label{SDEfunc}
	\Gtag{x;b}
	X_t=x +\int_0^t b(X_r)\,dr +W_t^H, \qquad t\in[0,T],
\end{equation*}
where the integration is understood in the Lebesgue sense.
Second, we consider the case when $b\in\M(\R^d,\R^d)$, in which $b$ is not a function and thus
the composition $b(X_t)$ in \eqref{mainSDE} is not-well defined. We will define precisely the meaning of a solution to \eqref{mainSDE} in two different ways.

 One approach follows the ideas of Bass and Chen \cite[Definition~2.1]{BC}, \cite[Definition~2.5]{BC03} (see also \cite[Definition~3.9]{bib:zz17}, \cite[Definition~2.1]{ABM2020}), who suggested to understand an ill-defined drift term $\int b(X_t) dt$ as the limit of the corresponding approximations.  We recall the following concepts.

\begin{definition}
We say that a sequence of functions $f^n\colon\R^d\to\R^d$, $n\in\Z_+$, converges to a function $f$ in $\B^{\beta-}_{p}$, $\beta\in\R$, $p\ge1$, if $\sup_n \|f_n\|_{\B^{\beta}_{p}}<\infty$ and  for any $\beta'<\beta$ we have $\|f_n-f\|_{\B^{\beta'}_{p}}\to0$ as $n\to\infty$.
\end{definition}

\begin{definition}\label{D:sol}
Let $b\in\B^{\beta}_{p}$, where $\beta\in\R$, $p\in[1,\infty]$, $T>0$.  
A continuous   process $(X_t)_{t\in[0,T]}$ taking values in $\R^d$ is called a \textbf{regularized solution} to SDE \eqref{mainSDE} with the initial condition $x\in\R^d$, if there exists a process $(\psi_t)_{t\in[0,T]}$ taking values in $\R^d$ such that:
\begin{enumerate}[(i)]
\item $X_t=x+\psi_t+W_t^H$, $t\in[0,T]$;
\item for \textit{any} sequence $(b^n)_{n\in\Z_+}$ of $\C^\infty_b(\R^d,\R^d)$ functions   converging to $b$ in $\B^{\beta-}_{p}$ we have
    \begin{equation*}
    	\lim_{n\to\infty}\sup_{t\in[0,T]}\Big|\int_0^t b^n(X_r)\,dr-\psi_t\Big|= 0\,\,\text{in probability}.
    \end{equation*}
% where $\psi^n(t):=\int_0^t b^n(X_r)\,dr$.
\end{enumerate}
\end{definition}

\begin{remark} \begin{enumerate}[(i)]
\item
It is easy to see that for any $f\in \B^\beta_p$, there exists a sequence of $\C_b^\infty$ functions $(f_n)_{n\in\Z_+}$ which converges to $f$ in $\B^{\beta-}_p$ as $n\to\infty$. For example, one can take $f_n := P_{\frac1n}f$, see, e.g., \cite[Lemma~A.3]{ABLM}. Therefore, the sequence $(b^n)_{n\in\Z_+}$  required in the second part of \cref{D:sol} always exists.
\item Since, as mentioned before, $\M(\R^d,\R^d)\subset \B^0_1$, the above definition provides a notion of a solution to SDE \eqref{mainSDE} for the case where $b$ is a signed finite Radon measure.
		
\item It is immediate that for $b\in\C^\beta(\R^d)=\B^\beta_\infty(\R^d)$, $\beta>0$, a regularized solution to \eqref{mainSDE} is a standard notion of a solution to SDE \Gref{x;b}. Indeed, in this case we can choose a sequence  $(b^n)_{n\in\Z_+}$ which converges to $b$ uniformly and thus $\psi_t=\int_0^t b(X_r)\, dr$. We will show that this is also the case if $b\in L_p(\R^d)$ for the whole range of $p$ in which we are able to construct a solution, see \cref{T:func} for the precise formulation.
\end{enumerate}
\end{remark}

We will consider the following class of solutions.
\begin{definition}
We say that a regularized solution $(X,W^H)$ to \eqref{mainSDE} is in the class \textbf{BV}, if $X-W^H$ has finite variation almost surely, i.e.
$\|X-W^H\|_{1-\var;[0,T]}<\infty$ a.s.
\end{definition}
We see that if $b$ is a nonnegative function or a non-negative measure, then the process $X-W^H$ is nondecreasing, and thus automatically of finite variation (one can see this by approximating the non-negative measure $b$ by non-negative functions  $b^n$ in \cref{D:sol})

Following \cite[Section~3.2]{OuNu}, we say that a couple $(X,W^H)$ on a complete filtered probability space $(\Omega, \F, \P, (\F_t)_{t\in[0,T]})$ is a \textit{weak solution} to \eqref{mainSDE}, if $W^H$ is an $(\F_t)$-fractional Brownian motion, $X$ is adapted to $(\F_t)$, and $X$ is a solution to \eqref{mainSDE} (in the standard sense or in the sense of \cref{D:sol}). A weak solution $(X,W^H)$ is called a \textit{strong solution}  if $X$ is adapted to $(\F^{W^H}_t)$, that is the filtration generated by $W^H$. We say that \textit{pathwise uniqueness} holds for \eqref{mainSDE} if for any two weak solutions of  \eqref{mainSDE}  $(X,W^H)$ and $(Y,W^H)$  with common noise $W^H$ on a common probability space (w.r.t. possibly different filtrations) and with the same initial conditions one has $\P(X_t=Y_t \text{ for all $t\in[0,T]$})=1$. 

Now let us present our main results for the  case when $b$ is a measurable function.
\begin{theorem}\label{T:func}
Let $b$ be a measurable function in $L_p(\R^d,\R^d)$, $p\in[1,\infty]$, $H\in(0,1)$, $x\in\R^d$ and suppose that \begin{equation}\label{maincond}
\frac{d}{p}<\frac1H-1.
\end{equation}
Then the following holds:
\begin{enumerate}[\rm{(}i\rm{)}]
\item\label{part1tfunc}  equation \Gref{x;b} has a weak solution $(X,W^H)$;
\item\label{part2tfunc} let $(X,W^H)$ be any weak solution to  \Gref{x;b}.  Then   for any $m\ge1$ there exists a constant $C=C(H,d,m,p,T,\|b\|_{L_p(\R^d)})$ such that for any $0\le s \le t\le T$ one has 
\begin{equation*}%\label{varcondsol}
\|\,\|X-W^H\|_{1-\var;[s,t]}\,\|_{L_m(\Omega)}\le C (t-s)^{1-\frac{Hd}p};
\end{equation*}
\item\label{partiiitfunc} let $(X,W^H)$ be any weak solution to  \Gref{x;b}. Then it is also a weak regularized solution to  \eqref{mainSDE} (in the sense of \cref{D:sol}). 
\item\label{part3tfunc} let $(X,W^H)$ be any  weak regularized solution to equation \eqref{mainSDE} in the class \textbf{BV}. Then $(X,W^H)$ is a weak solution to  \Gref{x;b}.\end{enumerate}
\end{theorem}

As mentioned before, we are able to show that condition \eqref{maincond} is essentially optimal. We provide the following general counter-example, which justifies the heuristics discussed in the introduction: regularization by noise requires the noise to be rough enough. Otherwise, SDE \Gref{x;b} might fail to have a solution.
\begin{theorem}\label{T:genopt}
Let $d\in\N$, $\gamma\in(0,1)$, $\alpha>\frac{1-\gamma}\gamma$, $f=(f^1,f^2,\hdots,f^d)\in\C_0^\gamma([0,1],\R^d)$. Suppose that $f$ is not identically $0$. Then the deterministic equation 
\begin{equation}\label{SDEce}
	X_t^i=-\int_0^t\sign(X_s^i) |X_s|^{-\alpha}\I(|X_s|<1) ds+f_t^i,\quad i=1,\hdots,d,
\end{equation}
where the integration is understood in the Lebesgue sense, has no continuous solutions.
\end{theorem}

\begin{corollary}\label{T:opt}
Let $d\in \N$, $p\in[1,\infty)$, $H\in(0,1)$ and suppose that
\begin{equation}\label{maincondnot}
\frac{d}{p}>\frac1H-1.
\end{equation}
Then there exists a function $b\in L_p(\R^d,\R^d)$, $x\in\R^d$ such that 
 equation \Gref{x;b} has no weak solutions.
\end{corollary}

Surprisingly enough, \cref{T:func} combined with \cref{T:genopt} allows to show that  trajectories of $W^H$, $H\in(0,1)$ cannot be $H'$-H\"older continuous for any $H'>H$. Indeed, \Gref{x;b} has a solution, but equation \eqref{SDEce} has no solutions if the forcing $f=W^H$ is smooth enough. Thus, $W^H$ cannot be too smooth. This statement is well-known, see, e.g., \cite[p.~220]{MRLifshits}, but the standard proof is, of course, very different. Here we provide a regularization by noise perspective on this statement.
\begin{corollary}\label{C:nr}
Let $H\in(0,1)$, $H'>H$. Then $\P(W^H\in\C^{H'}([0,1],\R))=0$.
\end{corollary}

In case $H=1/2$, condition \eqref{maincond} just becomes  the Krylov--R\"ockner condition $d/p<1$. Previously, weak existence of solutions to \Gref{x;b} for $L_p$ drifts was established in the literature under more restrictive conditions than \eqref{maincond}. Namely,  \cite[Theorem~3.7]{NO03} requires $d=1$, $p>2$, $H\in(0,\frac12)$; \cite[Theorem~1.13]{CG16} and embedding $L_p(\R^d)\subset \C^{-\frac{d}p}$ requires $d\in\N$, $\frac{d}p<\frac1{2H}-1$, $H\in(0,\frac12)$; \cite[Theorem~6.1]{LeSSL} assumes $\frac{d}{p}<\frac1{2H}$, $p>2$, $H\in(0,\frac12)$; \cite[Theorem~8.2]{GG22} and \cite[Theorem~2.5]{ART21}  impose condition $\frac{d}p<\frac1{2H}-\frac12$.

\begin{remark}
The case of time-dependent drift $b \in L_q([0,1],L_p(\R^d))$, $p,q \in [1,\infty]$,  
was considered in a subsequent work \cite{bgallay}. Note that in this case  
a natural candidate for the weak existence condition is a version of  
the Ladyzhenskaya–Prodi–Serrin (LPS) condition:  
\begin{equation*}  
\frac{d}{p} + \frac{1}{Hq} < \frac{1}{H} - 1,  
\end{equation*}  
see \cite[Example~1.2]{GG22}. However, the optimal condition for weak existence is even better. It was shown in \cite{bgallay} that the SDE \eqref{mainSDE} has a weak solution if  
\begin{equation}\label{timespace}
\frac{d}{p} + \frac{1-H}{H} \frac{1}{q} < \frac{1}{H} - 1,  
\end{equation}  
and that if this condition is not satisfied, weak existence might fail. 
For further discussions regarding condition \eqref{timespace}, we refer to \cite{Galeati24}.
\end{remark}

\begin{remark}\label{R:levy}
Recall that the fractional Brownian motion has a scaling property $\Law(W_t^H)=\Law(t^H W^H_1)$. If $L^\alpha$ is an $\alpha$-stable process, $\alpha\in(1,2)$, then one has a similar property $\Law(L^\alpha_t)=\Law(t^{1/\alpha} L^\alpha_1)$. Thus,  parameter $H\in(1/2,1)$ ``morally'' corresponds to $1/\alpha$. Therefore, if one repeats the strategy of the  proof of \cref{T:func}\ref{part1tfunc} for an $\alpha$-stable  process instead of fractional Brownian motion (with appropriate modifications due to the discontinuity of $L^\alpha$, see subsequent work \cite{bgallay} where a more general result is obtained) one gets the following condition for weak existence of a solution to \Gref{x;b} with the driving noise $L^\alpha$ in place of $W^H$: $b\in L_p(\R^d,\R^d)$ and 
$$
\frac{d}{p}<\alpha-1
$$
(this is  \eqref{maincond} with $H$ replaced by $1/\alpha$). This coincides with the Podolynny and Portenko result \cite[p. 123]{PP95}.
\end{remark}

Now we turn to study the case where $b$ is a measure. Here one can again define a solution to 
\eqref{mainSDE} in the regularised sense. However there is also another way to make sense of the drift $\int b(X_t) dt$ which was pioneered by Stroock and Yor \cite[Theorem~1.9]{SY81} and Le~Gall \cite{LG84}, see also \cite{ES85}, \cite{BC2005}, and the survey \cite{lejay2006constructions}. The idea is to utilize the local time (occupation density) of the process $X$. 

\begin{definition}\label{d:meas}

Let $b\in\M(\R^d,\R^d)$, $T>0$, $x\in\R^d$. We say that a continuous   process $(X_t)_{t\in[0,T]}$   taking values in $\R^d$ solves
\begin{equation}\label{measureeq}
X_t=x+\int_{\R^d} L^X_t(y)\, b(dy) +W_t^H,\quad  t\in[0,T],
\end{equation}
where%
%\KL{can we pick a continuous version of $L^X$? }
\begin{equation}\label{lebloc}
L^X_t(y):=\limsup_{\eps\to0}\frac1{v_d\eps^{d}}\int_0^t \I(|X_s-y|<\eps)\,ds,\quad t\in[0,T],\,\, y\in\R^d,
\end{equation}
 if  the  occupation measure of $X$,  $\mu^X_t(A):=\int_0^t \I(X_s\in A)\,ds$, $A\in\mathscr{B}(\R^d)$ is absolutely continuous with respect to the Lebesgue measure and \eqref{measureeq} holds $\P$-almost surely. The integration with respect to the measure is understood in the Lebesgue sense.
\end{definition}

In what follows by solution to equation \eqref{measureeq} we always mean a solution in the sense of   \cref{d:meas}.

It is clear, that if $b$ has a  density $\rho_b$ with respect to the Lebesgue measure, then 
\begin{equation}\label{occupationtime}
\int_{\R^d} L^X_t(y)\, b(dy)=\int_0^t \rho_b(X_s)\,ds,
\end{equation}
 and \eqref{measureeq} becomes \Gref{x;\rho_b}.
\begin{remark}
Clearly, $L^X$ in \eqref{measureeq}--\eqref{lebloc} is just a local time of $X$, that is  $L_t^X=\frac{d \mu_t^X}{d \Leb}$. 
However, since measure $b$ can be singular with respect to the Lebesgue measure, it is crucial to carefully select an appropriate version of the Radon–Nikodym derivative (as it was done in \eqref{lebloc}) in order to avoid absurd situations. Indeed, let $d=1$ and consider an equation $d X_t = \delta_0(X_t) \,dt +d W_t^H$. 
Let  $L^W$ be the local time of $W^H$. Put 
$$
\wt L^W_t(x):=L^W_t(x)\I(x\neq0).
$$ 
Then $\wt L^W$ is also a version of the Radon–Nikodym derivative $\frac{d \mu_t^W}{d \Leb}$. Further, by definition, $\int_{\R} \wt L^W_t(y)\, \delta_0(dy)=0$, and thus
\begin{equation*}
W_t^H=\int_{\R} \wt L^W_t(y)\, \delta_0(dy) +W_t^H.
\end{equation*}
Hence, if we had allowed $L^X$ in \eqref{measureeq} to be \textit{any} version of the local time of $X$, we would have concluded that $W^H$ solves the equation $dX_t = \delta_0(X_t)  dt + dW_t^H$,  which makes no sense. Le Gall \cite{LG84} avoided this issue by requiring $L^X$ to be the semimartingale local time of $X$, defined via the Tanaka formula. Since It\^o's calculus for a generic $X$ is unavailable when $H \neq 1/2$, we had to manually select the correct version of $L^X$ in \eqref{lebloc}.
\end{remark}

The next theorem shows weak existence and  stability of solutions to SDE \eqref{mainSDE} for the case when $b$ is a measure.

\begin{theorem}\label{T:ident}
Let $b\in \M(\R^d,\R^d)$, $H\in(0,\frac1{d+1})$, $x\in\R^d$.
Then the following holds:
\begin{enumerate}[\rm{(}i\rm{)}]
\item\label{part1tm}  equation \eqref{mainSDE} has a weak regularized  solution $(X,W^H)$ (in the sense of \cref{D:sol}), which is in the class \textbf{BV};

\item let $(b_n)_{n\in\Z_+}$  be a sequence of $\C^1(\R^d,\R^d)$ functions converging to $b$ in $\B^{0-}_1$. Suppose that 
\begin{equation}\label{condL1}
\sup_{n\in\Z_+}\|b_n\|_{L_1(\R^d)}<\infty.
\end{equation}
Let $(x_n)_{n\in\Z_+}$ be a sequence  of vectors in  $\R^d$ converging to $x\in\R^d$. Let $X_n$  be a strong solution to \Gref{x_n;b_n}. Then  the sequence $(\Law(X_n,W^H))_{n\in\Z_+}$ is tight in  $\C([0,T],\R^{2d})$ and any of its partial limits  is a weak regularized solution to equation \eqref{mainSDE} with the initial condition $x$.
\end{enumerate}
\end{theorem} 

\cref{T:ident} advances the current state of the art. Indeed,  \cite[Theorem~2.5 and Corollary~2.6]{ART21} show weak existence of solutions to \eqref{mainSDE} for the case where $d=1$, $b\in \M_+(\R,\R)$, $H<\sqrt2 -1$; \cite[Theorem~1.13]{CG16} and embedding $\M(\R^d)\subset \C^{-d}$ requires $d\in\N$, $b\in \M(\R^d,\R^d)$, $H<\frac1{2(d+1)}$, and the last bound was improved recently in \cite[Theorem~8.2]{GG22} to $H<\frac1{2d+1}$.

\cref{T:ident}(ii) shows that  when $H<\frac1{d+1}$, approximations of equation \eqref{mainSDE} can converge only to a regularized solution to this equation. On the other hand, \cref{T:ident}(ii) might not hold in the regime $H\ge\frac1{d+1}$. Indeed, take $d=1$, $H=1/2$, $b_n=\beta p_{1/n}$, $\beta\in\R$. Obviously, the sequence $(b_n)_{n\in\N}$ satisfies \eqref{condL1} and converges to $\beta\delta_0$ in $\bes^{0-}_1$. On the other hand, Le Gall \cite[page~65]{LG84} showed that a sequence of solutions to \Gref{0;b_n} converges to a solution of \eqref{mainSDE} with the drift $\frac{1-\exp(-2\beta)}{1+\exp(-2\beta)}\delta_0$ which is different from the expected drift $\beta\delta_0$.

The next result  shows that in the whole considered  range $H\in(0,\frac1{d+1})$,  a regularized solution to \eqref{mainSDE} is essentially equivalent to a solution of equation \eqref{measureeq}. In a special case, when the measure $b$ has a density $\rho_b$, these two notions are the same as the standard notion of a solution to \Gref{x;\rho_b} thanks to the occupation times formula \eqref{occupationtime}.

\begin{theorem}\label{T:measure}
Let $b\in \M(\R^d,\R^d)$, $H\in(0,\frac1{d+1})$, $x\in\R^d$. Then the following holds:
\begin{enumerate}[\rm{(}i\rm{)}]
\item\label{part3tm} let $(X,W^H)$ be any weak regularized solution to equation \eqref{mainSDE} in the class \textbf{BV}. Then $(X,W^H)$ is a weak solution to equation \eqref{measureeq};
\item let $(X,W^H)$ be any weak solution to \eqref{measureeq}. Then $(X,W^H)$ is a weak regularized solution to equation \eqref{mainSDE} and belongs to \textbf{BV}.
\end{enumerate}
If (i) or (ii) holds, then  for any $m\ge1$ there exists a constant $C=C(H,d,m,T,|b|(\R^d))>0$ such that for any $0\le s\le t\le T$ we have 
	\begin{equation}\label{reztmeas}
	\|\,\|X-W^H\|_{1-\var;[s,t]}\,\|_{L_m(\Omega)}\le C (t-s)^{1-Hd}.
\end{equation} 
\end{theorem}

\begin{remark}
	Condition $H<1/(d+1)$ in \cref{T:ident,T:measure} is also essentially optimal. Indeed, if $d=1$, $H=1/2$, then Harrison and Shepp \cite{HS81} showed that equation  \eqref{measureeq} has no solution if $b=\beta \delta_0$ for $|\beta|>1$.
\end{remark}

Now we move to the strong well-posedness of \eqref{mainSDE}.

\begin{theorem}\label{T:uniq}
Suppose that $d=1$, $H\in(0,1)$, $x\in\R$.
\begin{enumerate}[\rm{(}i\rm{)}]
\item\label{tuniqp1} Let $b\in \M_+(\R,\R)$, $H<(\sqrt{13}-3)/2\approx 0.303$. Then equation \eqref{mainSDE} has a unique strong regularized solution. 
\item\label{partuniq} Let $b\in \M(\R,\R)$, $H<(\sqrt{13}-3)/2$. Then in class \textbf{BV} equation \eqref{mainSDE} has a unique strong regularized solution. 
\item Let $b$ be a measurable function in $L_p(\R,\R)$, $p\in[1,\infty]$ and suppose that $H<\frac{p}{2p+1}$. For $p\in[1,2]$ suppose additionally that
\begin{equation}\label{uniqcond}
H^2+H(1+\frac2p)-1<0.
\end{equation}
Then equation \Gref{x;b} has a unique strong solution.
\end{enumerate}
\end{theorem}

By taking in \cref{T:uniq}\ref{tuniqp1} $b=\delta_0$, we get that skew fractional Brownian motion, that is solution to \eqref{sFBM}, is well-defined for $H<(\sqrt{13}-3)/2$. This improves the previous best bound $H\le\frac14$ \cite[Theorem~2.9]{ART21}. We believe that our methods can be very useful for studying the flow of skew Brownian motions, in particular, its H\"older continuity and differentiability, thus extending the results of \cite{ABF20} (see also \cite[Theorem~1.4(iii)]{GG22}). However this will the subject of further research.

We do not claim optimality of \cref{T:uniq}. However, this result  improves the current state of the art. Indeed, if $b\in \M(\R,\R)$, then strong uniqueness of \eqref{mainSDE} is known for $H\le \frac14$ (\cite[Theorem~1.13]{CG16} and embedding $\M(\R)\subset \C^{-1}$; \cite[Theorem~2.9]{ART21}) and for $H=\frac12$ as long as $b$ does not have atoms with weight more than $1$ (\cite[Theorem~2.3]{LG84}, \cite[Theorem~4.5]{BC2005}). 
We make the following conjecture, but proving it seems to be a very hard challenge.

\begin{conjecture}\label{conj}
Let $b\in \M(\R,\R)$, $H\in(0,1/2)$. Then equation \eqref{mainSDE} has a unique strong solution.
\end{conjecture}
Parts (i) and (ii) of \cref{T:uniq} reduce the gap where strong well--posedness of \eqref{mainSDE} is still not known but expected from $(1/4,1/2)$ to  $[(\sqrt{13}-3)/2,1/2)$. 

If $b\in L_p(\R,\R)$, then \cite[Theorem~1.13]{CG16} and embedding $L_p(\R)\subset \C^{-\frac1p}$ yields strong well-posedness of  \Gref{x;b} for $H<\frac{p}{2(p+1)}$ (see also \cite[Theorem~6.2]{LeSSL} for the same result). This is improved by \cref{T:uniq}(iii) for all $p\in[1,\infty)$, as $H<p/(2p+2)$  implies $H<p/(2p+1)$ and \eqref{uniqcond}. 

Let us mention that the condition on $H$ in \cref{T:uniq}(iii) for strong existence and uniqueness is more restrictive than the bound $H<p/(p+1)$ from \cref{T:func} (take $d=1$ in \eqref{maincond}) which guarantees weak existence. One might expect that strong existence and uniqueness would also hold under the conditon $H<p/(p+1)$. Note that condition $H<(\sqrt{13}-3)/2$ from parts (i), (ii) of \cref{T:uniq} is just the condition from part (iii) of the theorem with $p=1$. Therefore the conjecture that $H<p/(p+1)$ might lead to strong well-posedness for an $L_p(\R,\R)$ drift is in line with \cref{conj} for measure-valued drifts. 

\smallskip
\cref{T:func}\ref{part3tfunc},  \cref{T:measure}\ref{part3tm},  \cref{T:uniq}\ref{partuniq} require that the drift $X-W^H$ is of finite variation a.s. This  is a weaker version of the corresponding condition which appears in the analysis of SDEs with irregular drift driven by Brownian motion or an $\alpha$-stable process, see \cite[Definition 2.1(iii)]{BC}, \cite[Definition 2.5(b)]{BC03}, \cite[Definition 3.1 and Corollary 5.3]{bib:zz17}, \cite[Definition~2.2]{ABM2020} and so on. As mentioned above, if $b$ is a nonnegative function or a nonnegative measure, this condition is automatically satisfied. Note that in all the cases no regularity of the drift is imposed a priori. 

To show equivalence of equations \eqref{mainSDE} and \eqref{measureeq} for measure valued drifts, we analyze local times of fractional Brownian motion and related processes. Local time of fractional Brownian motion has been studied in  \cite[Theorem~8.1]{BG73}, \cite[Theorem~4]{Pitt}, \cite[Theorem~30.4]{MR556414}. It follows from \cite[Corollary~1.1 and Lemma~2.5]{Xiao}, \cite[Proposition~3.3]{SSV22} and a straightforward application of the Kolmogorov continuity theorem, that under the condition $Hd<1$ fractional Brownian motion has a local time which is jointly continuous in time and space; moreover, it belongs to  $\C^{1-Hd-\eps}([0,T],\R)$ in time and $\C^{(\frac{1-Hd}{2H}\wedge1)-\eps}(\R^d,\R)$ in space. We were able to obtain a new short proof of this result which avoids tedious moment computations and Fourier analysis. %, and remove the unnecessary minimum with $1$ in the space regularity of local time. 
This is the subject of part (i) of the next theorem. The existence of the local time obtained in the second part of the theorem extends the well-known fact that occupation measure of a sum of a $1$-dimensional Brownian motion and an adapted process of bounded variation is absolutely continuous with respect to the Lebesgue measure.

\begin{theorem}\label{T:loc}
Let $d\in\N$, $H\in(0,1)$. Let $W^H$ be an $(\F_t)$--fractional Brownian motion. 
\begin{enumerate}[\rm{(}i\rm{)}]
\item Assume that $Hd<1$. Then process $W^H$ has a local time $L^W$ which is jointly continuous in $(t,x)$.  Furthermore, for any $\eps>0$, $\gamma\in [0,(\frac1{2H}-\frac{d}2)\wedge1)$ one has
a.s. 
\begin{equation}\label{loc1}
\sup_{\substack{x\in\R^d\\0\le s\le t\le T} }\frac{L^W([s,t],x)}{|t-s|^{1-Hd-\eps}}<\infty,\qquad
\sup_{\substack{x,y\in\R^d\\0\le s\le t\le T} }\frac{|L^W([s,t],x)-L^W([s,t],y)|}{|t-s|^{1-Hd-H\gamma-\eps}|x-y|^\gamma}<\infty.
\end{equation}
\item Assume that $H(d+1)<1$. Let $\psi\colon[0,T]\times\Omega\to\R^d$ be a continuous process adapted to the filtration $(\F_t)$ and $\psi\in\C^{1-\var}([0,T],\R^d)$ a.s. Then the process $W^H+\psi$ has a local time $L^{W+\psi}$ which is jointly continuous in $(t,x)$. Furthermore,
for any $\eps>0$, $\gamma\in [0,(\frac1{2H}-\frac{d}2-\frac12)\wedge1)$ one has
a.s. 
\begin{equation}\label{loc2}
	\sup_{\substack{x\in\R^d\\0\le s\le t\le T} }\frac{L^{W+\psi}([s,t],x)}{|t-s|^{1-H(d+1)-\eps}}<\infty,\qquad
	\sup_{\substack{x,y\in\R^d\\0\le s\le t\le T} }\frac{|L^{W+\psi}([s,t],x)-L^{W+\psi}([s,t],y)|}{|t-s|^{1-H(d+1)-H\gamma-\eps}|x-y|^\gamma}<\infty.
\end{equation}
\end{enumerate}
\end{theorem}
In the above theorem, $L([s,t])=L(t)-L(s)$ is the local time of the process accumulated over the interval $[s,t]$.

\begin{remark} The condition  $H(d+1)<1$ in part (ii) of the above theorem is optimal in the following sense: if $H(d+1)\ge1$, then the process $W^H+\psi$ might have local time which is not jointly continuous. Indeed, take $H=1/2$, $d=1$; then $H(d+1)=1$. Consider a reflected Brownian motion $|B|$, which by Tanaka's formula can be represented as $|B_t|=W_t+L^{|B|}(t,0)$, where $W$ is the standard Brownian motion, $L^{|B|}(t,0)$ is local time of $|B|$ at $0$; clearly $L^{|B|}(\cdot,0)$ is a process of finite variation. However, local time of $|B|$ is discontinuous at $x=0$ for any $t>0$. Condition $Hd<1$ in part (i) of the theorem is of course also optimal, but this is classical: \cite[Theorem~1.1]{Talagrand}, \cite[Theorem~2.24]{BMbook}.
\end{remark}

Finally, let us briefly describe our proof strategy. We sketch here informally the main steps. First, using stochastic sewing and a quantitative version of John--Nirenberg inequality, we are able to bound all moments of $\int f(W^H_s)\,ds$ for generic function $f$ in terms of $\|f\|_{\B^\alpha_p}$ for  $\alpha<0$, $p\ge1$, see \cref{C:ourbounds}. A careful application of a newly established  Rosenthal-type stochastic sewing lemma (\cref{T:RoSSL}) allows to squeeze an additional finite variation random drift $\psi$ in the integral; this is done in \cref{L:driftb2}; see also \cref{R:highmom} explaining why the application of standard stochastic sewing lemma would have led to a non-optimal results.
By taking in these bounds $f=\delta_x$, for $x\in\R^d$, we establish existence of local time for $W^H$ and $W^H+\psi$ and its regularity (\cref{T:loc}). By letting $f=b_n$ where $b_n$ is an approximation of $b$, we obtain the stability result and get weak existence (\cref{T:func,T:ident}). Finally, strong uniqueness (\cref{T:uniq}) follows from the same integral bound of  \cref{L:driftb2}
applied for $f=\nabla b$ and a result from the theory of deterministic Young equations (\cref{T:YoungODE}).

\section{Tools}
\label{sec:tools}

In this section we present the toolkit which is used to obtain our main results. Sewing, John--Nirenberg inequality and related techniques which allow to bound moments of certain integrals are presented in \cref{S:st}, the result related to Young differential equations  which is needed for uniqueness is presented in \cref{S:YDE}.

Let us introduce further necessary notation which will be used in the article.
For $0\le S< T$ denote by $\Delta_{[S,T]}$ the simplex
\begin{equation*}%\label{Deltanotation}
	\Delta_{[S,T]}:=\{(s,t)\in[S,T]^2\colon s\le t\}.
\end{equation*}
The mesh size of a partition $\Pi$ of an interval will be denoted by $|\Pi|$. For a filtered probability space $(\Omega,\F,(\F_t)_{t\in[0,T]},\P)$, $T>0$, we will denote by $\E^t$ the conditional expectation with respect to $\F_t$.

We will use the following elementary bound which follows from Jensen's inequality. If $\mathcal{G}\subset\mathcal{H}$ are sub-$\sigma$-algebras, $p\ge1$, and $\xi$ is an integrable random vector, then
\begin{equation}\label{sigmaalg}
\E\bigl[|\E [\xi|\mathcal{H}]|^p \,\bigl|\mathcal{G}\bigr]\le 
\E\bigl[|\xi|^p \,\bigl|\mathcal{G}\bigr].
\end{equation}

%For $\alpha\in(0,1)$ we define the standard H\"older norm and seminorm of $f$:
%\begin{equation*}
%[f]_{\C^\alpha([s,t])}:=\sup_{s\le s',t'\le t}\frac{|f(t')-f(s')|}{|t'-s'|^\alpha},\quad
%\|f\|_{\C^\alpha([s,t])}:=\sup_{s\le r\le t}|f(r)|+ [f]_{\C^\alpha([s,t])}.
%\end{equation*}

\subsection{Sewing and related techniques}\label{S:st}

In this section we present a variety of our sewing techniques needed for the proofs.
Whilst some of the statements below (\cref{P:sing,P:BMO,T:SSLst}) are known or can be viewed as minor modifications of the existing results, the Rosenthal-type stochastic sewing lemma, \cref{T:RoSSL}, is essentially new. The key difference between  \cite[Theorem~2.1]{LeSSL} and \cref{T:RoSSL} is that the former uses the Burkholder-Davis-Gundy inequality, while the latter uses the Rosenthal--Burkholder inequality (\cite[Theorem~21.1]{Bur73}), which leads to a different set of conditions. One application of this new stochastic sewing lemma is \cref{L:driftb2}, which does not follow from the original stochastic sewing lemma. We strongly believe that this new sewing lemma will find further interesting applications.

Let $(A_{s,t})_{(s,t)\in\Delta_{[S,T]}}$ be a collection of random vectors in $\R^d$ such that $A_{s,t}$ is $\F_t$-measurable for every $(s, t)\in\Delta_{[S,T]}$. For every triplet of times $(s,u,t)$ such that 
$S\le s\le u\le t$ denote  as usual
$$ 
\delta A_{s,u,t}:= A_{s,t}-A_{s,u}-A_{u,t}.
$$

\begin{proposition}[Taming singularities, {\cite[Lemma~3.4]{le2021taming}}, {\cite[Lemma~2.3]{BFG}}]
Let $(\mathcal E,d)$ be a metric space. Suppose that there exist constants 
$\tau_i,\eta_i\ge0$,  $\tau_i>\eta_i$, $\Gamma_i>0$, $i=1,\hdots, h$, such that a function $Y\colon (0,T]\to \mathcal E$ satisfies  \label{P:sing}
\begin{equation*}%\label{con.Ysing}
d(Y_s,Y_t) \le \sum_{i=1}^h \Gamma_is^{-\eta_i}(t-s)^{\tau_i}\quad\text{for any $0< s\le t\le T$.}
\end{equation*}
Then
\begin{equation*}
d(Y_s,Y_t)\le\sum_{i=1}^h(1-2^{\eta_i- \tau_i})^{-1} \Gamma_i(t-s)^{\tau_i- \eta_i}\quad\text{for any $0< s\le t\le T$.}
\end{equation*}	
\end{proposition}
\begin{proof}
Compared to \cite[Lemma~3.4]{le2021taming}, the restrictions $\tau_i, \eta_i \le 1$ have been removed. Nonetheless, the proof in the aforementioned reference remains valid without these restrictions
\end{proof}

\begin{proposition}[John--Nirenberg inequality, {{\cite[Exercise~A.3.2]{SVbook}, \cite[Theorem~2.3]{le2022}}}]
Let $S\in[0,T]$ and let $\A=\{\A_t:t\in[S,T]\}$ be a continuous process adapted to the filtration $(\F_t)_{t\in[0,T]}$. Assume that there exist a constant $\Gamma>0$ such that
 for any $(s,t)\in\Delta_{[S,T]}$ one has 
\begin{equation}\label{BMOcon}
\E^s|\A_t-\A_s|\le \Gamma,\quad \text{a.s.}
\end{equation}
Then for any $m\in[1,\infty)$, there exists a constant $C=C(m)$ independent of $S,T, \Gamma$ such that for any 
$(s,t)\in\Delta_{[S,T]}$ one has \label{P:BMO} 
\begin{equation*}
\|\A_t-\A_s\|_{L_m(\Omega)}\le C \Gamma.
\end{equation*}	
\end{proposition}

\begin{remark}
We would like to stress that continuity of $\A$ plays a crucial role in \cref{P:BMO}: indeed, if $\A$ is a standard Poisson process, then clearly $\E^s|\A_t-\A_s|\le t-s$,  but of course it is not true that $\E (\A_t-\A_s)^m\le C(t-s)^m$ for $m>1$. If $\A$ is an $\alpha$-stable process, $\alpha\in(1,2)$, then $\E^s|\A_t-\A_s|\le C(t-s)^{1/\alpha}$, but  $\E (\A_t-\A_s)^m=\infty$ for $m\ge\alpha$. We refer the reader to \cite{le2022} for a general version of this statement without continuity assumption.
\end{remark}

\begin{proposition}[{Stochastic sewing lemma, \cite[Theorem~2.1]{LeSSL}}]\label{T:SSLst}
	Let $m\in[2,\infty)$, $S\in[0,T]$.  	Assume that there exist  constants $\Gamma_1, \Gamma_2, \Gamma_3\ge0$, $\eps_1, \eps_2, \eps_3 >0$ such that the following conditions hold for every $(s,t)\in\Delta_{S,T}$ and $u:=(s+t)/2$
	\begin{align}	
		&\|\delta A_{s,u,t}\|_{L_m(\Omega)}\le \Gamma_1|t-s|^{\frac12+\eps_1}+
		\Gamma_2|t-s|^{\frac12+\eps_2},\label{con:s1}	\\
			&\|\E^s[ \delta A_{s,u,t}]\|_{L_m(\Omega)}\le \Gamma_3|t-s|^{1+\eps_3},\label{con:s2}
	\end{align}
	
	Further, suppose that there exists a process $\A=\{\A_t:t\in[S,T]\}$ such that for any $t\in[S,T]$ and any sequence of partitions $\Pi_N:=\{S=t^N_0,...,t^N_{k(N)}=t\}$ of $[S,t]$ with $\lim_{N\to\infty}\Di{\Pi_N}\to0$  one has
	\begin{equation}\label{con:s3}
		\sum_{i=0}^{k(N)-1} A_{t^N_i,t_{i+1}^N}\to  \A_{t}-\A_{S}\quad\text{in probability as $N\to\infty$}.
	\end{equation}
	
	Then there exists a constant $C=C(\eps_1,\eps_2,\eps_3,m)$ independent of $S,T$, $\Gamma_i$ such that for every $(s,t)\in\Delta_{S,T}$  we have 
	\begin{align}
		\|\A_{t}-\A_{s}-A_{s,t}\|_{L_m(\Omega)} &\le C \Gamma_1 |t-s|^{\frac12+\eps_1}+
		C \Gamma_2 |t-s|^{\frac12+\eps_2}+C \Gamma_3 |t-s|^{1+\eps_3}.
		\label{est:ssl1}
	\end{align}
\end{proposition}
\cref{T:SSLst} is a minor modification of \cite[Theorem~2.1]{LeSSL}, and for the proof we refer the reader either to  the original proof or to \cite[proof of Theorem~4.1]{ABLM}, which is similar. 

The John--Nirenberg inequality (\cref{P:BMO}) is very powerful; it states that a  good bound just on the conditional first moment of the increment of $\A$ is sufficient to bound all the moments of the increments of $\A$. However, it requires  precise knowledge of the conditional distribution $\Law(\A_t|\F_s)$. This is usually not a problem if $\A$ is a function of a fractional Brownian motion (see \cref{C:ourbounds}) or another process with a known law. However if $\A$ is a function of $W^H+\psi$, where $\psi$ is a generic drift of certain regularity, then a direct application of \cref{P:BMO} seems nontrivial.

This obstacle can be overcome using the stochastic sewing lemma, \cref{T:SSLst}. The lemma imposes no conditions on $\Law(\A_t|\F_s)$; instead it is assumed that we are in full control of law of $A$, which is an approximation of $\A$. A drawback is that in order to bound high moments of $\A$, one has now to bound all high moments of $A$, not just the first one as it was the case with the John--Nirenberg inequality. 

Our next result  takes the best of both worlds. It says that  to bound all the moments of $\A$, it suffices  to control only the first two conditional moments of the increments of its approximation $A$ and to have a very mild bound on a high moment of the increments of $A$.

As in \cite[Definition~4.6]{ABLM}, we  need the notion of \textit{random} control, which extends the notion of (deterministic) control, see, e.g., \cite[Section 0.1]{FZ18}. 
\begin{definition} Let $0\le S\le T$. We say that a measurable function $\lambda\colon\Delta_{[S,T]}\times\Omega\to\R_+$ is a \textit{random control}, if it is subadditive, that is for any $S\le s \le u\le t\le T$ one has
\begin{equation}\label{rcont}
\lambda(s,u,\omega)+\lambda(u,t,\omega)\le \lambda(s,t,\omega)\quad\text{a.s}.
\end{equation}
\end{definition}

\begin{theorem}[Rosenthal-type stochastic sewing lemma with random times and random controls]\label{T:RoSSL}
Let 
$n\in[2,\infty)$ and $S\in[0,T]$.  Let $\tau$ be a measurable random variable taking values in $[S,T]$. Let $\lambda_1$, $\lambda_2$ be random controls.	Assume that there exist  constants  $\alpha_1, \alpha_2, \alpha_3>0$, $\beta_1,\beta_2,\Gamma_1, \Gamma_2, \Gamma_3\ge0$ such that  
\begin{equation}\label{alphabeta}
\alpha_1+\beta_1>\frac12,\qquad \alpha_2+\beta_2>1,\qquad \alpha_3>\frac1n,
\end{equation}
and
the following conditions hold for every $(s,t)\in\Delta_{S,T}$ and $u:=(s+t)/2$
\begin{align}	
&(\E^u|\delta A_{s,u,t}|^2)^{1/2}\le \Gamma_1|t-s|^{\alpha_1}\lambda_1(s,t)^{\beta_1},\label{Rcon:s1}	\\
&|\E^u[ \delta A_{s,u,t}]|\le \Gamma_2|t-s|^{\alpha_2}\lambda_2(s,t)^{\beta_2},\label{Rcon:s2}\\
&\| A_{s,t}\|_{L_n(\Omega)}\le \Gamma_3|t-s|^{\alpha_3}.\label{Rcon:s3}	
\end{align}
	
	Further, suppose that there exists a process $\A=\{\A_t:t\in[S,T]\}$ such that for any $t\in[S,T]$ and any sequence of partitions $\Pi_N:=\{S=t^N_0,...,t^N_{k(N)}=t\}$ of $[S,t]$ with $\lim_{N\to\infty}\Di{\Pi_N}\to0$  one has
\begin{equation}\label{Rcon:3}
	\sum_{i=0}^{k(N)-1} A_{t^N_i,t_{i+1}^N}\I(t_{i}^N\le\tau)\to  \A_{t\wedge\tau}-\A_{S\wedge\tau}\quad\text{in probability as $N\to\infty$}.
\end{equation}
	
Then for every $m\in[2,n]$, there exists a constant $C=C(\alpha_1, \alpha_2,\alpha_3, \beta_1, \beta_2,m,n)>0$ independent of $S,T$, $\Gamma_i$ such that for every $(s,t)\in\Delta_{S,T}$  we have 
\begin{align}
\|\A_{t\wedge\tau}-\A_{s\wedge\tau}\|_{L_m(\Omega)}&\le C \Gamma_1 |t-s|^{\alpha_1}\|\lambda_1(s,t)^{\beta_1}\|_{L_m(\Omega)}+C \Gamma_2 |t-s|^{\alpha_2}\|\lambda_2(s,t)^{\beta_2}\|_{L_m(\Omega)}\nn\\
&\quad+C \Gamma_3 |t-s|^{\alpha_3}.
\label{Rres}
\end{align}
\end{theorem}

The main difference between the original stochastic sewing lemma  \cite[Theorem~2.1]{LeSSL} and \cref{T:RoSSL} is that the latter does not impose bounds on high moments of $\delta A_{s,u,t}$ apart from relatively weak condition \eqref{Rcon:s3}. Another difference is the use of random controls $\lambda$ and bounding differences of $\A$ up to a stopping time $\tau$. All these will be crucial in \cref{L:driftb2}, where we will apply \cref{T:RoSSL} to get a key bound on moments of  $\int f(W^H_r+\psi_r)dr$ for a generic drift $\psi$, see \cref{R:highmom}.

In the proof of \cref{T:RoSSL}, we  combine the technique from the proofs of stochastic sewing lemma  \cite[Theorem~2.1]{LeSSL}, stochastic sewing lemma with random controls \cite[Theorem~4.7]{ABLM} with the novel idea that an application of the Rosenthal-Burkholder inequality allows to obtain better results than an application of the  Burkholder--Davis--Gundy inequality. For the convenience of the reader we recall here the Rosenthal-Burkholder inequality, which is is \cite[Theorem~21.1]{Bur73} applied with $\Phi(\lambda):=|\lambda|^m$.

\begin{proposition}[{Rosenthal-Burkholder inequality, \cite[Theorem~21.1]{Bur73}}]\label{p:rbiq}
Let $n\in\N$ and $(f_k,\mathcal{G}_k)_{k=0,1,\hdots, n}$ be a martingale. Then for any $m\ge2$ there exists a constant $C=C(m)$ such that
\begin{equation*}
\|\max_{i=0,1,\hdots, n} f_i\|_{L_m(\Omega)}\le C \Bigl\|\sum_{i=0}^{n-1} \E[(f_{i+1}-f_i)^2|\mathcal{G}_i]	\Bigr\|_{L_{m/2}}^{1/2}+C \Bigl\|\max_{i=0,\hdots, n-1} |f_{i+1}-f_i|\Bigr\|_{L_m(\Omega)}.
\end{equation*}	
\end{proposition}

\begin{proof}[Proof of \cref{T:RoSSL}]

We will denote by $\pi^k_{[s,t]}=\{s=t_0^k<t_1^k<\hdots<t^k_{2^k}=t\}$ the dyadic partition of $[s,t]$; here $k\in\Z_+$ and $(s,t)\in\Delta_{[S,T]}$. That is, $t^k_i=s+i2^{-k}(t-s)$, $i=0,\hdots,2^k$. We denote by $u_i^k$ the midpoint of the interval $[t^k_i,t^k_{i+1}]$, that is, $u_i^k:=(t^k_i+t^k_{i+1})/2$. For notational convenience, we put $u_{2^k}^k:=t^k_{2^k}=t$.

Fix $s,t\in\Delta_{[S,T]}$. For $k\in\Z_+$ put
\begin{equation*}
A^{k}_{s,t}:=\sum_{i=0}^{2^k-1} A_{t_i^k,t_{i+1}^k}\I(t_{i}^k\le\tau).
\end{equation*}
By assumption \eqref{Rcon:3}, $A^k_{s,t}$ converges to  $\A_{t\wedge\tau}-\A_{s\wedge\tau}$ in probability. Note that for $k\in\Z_+$
\begin{align}\label{step1st}
|A^{k}_{s,t}-A^{k+1}_{s,t}|&\le\Bigl|\sum_{i=0}^{2^k-1} \delta A_{t_i^k,u_i^k,t_{i+1}^k}\I(t_{i}^k\le\tau)\bigr|+\Bigl|\sum_{i=0}^{2^k-1}A_{u_i^k,t_{i+1}^k}\I(\tau\in[t_i^k, u_i^k])\Bigr|\nn\\
&\le \Bigl|\sum_{i=0}^{2^k-1} \delta A_{t_i^k,u_i^k,t_{i+1}^k}\I(t_{i}^k\le\tau)\Bigr|+\max_{i=0,\hdots,2^k-1}|A_{u_i^k,t_{i+1}^k}|\nn\\
&\le \Bigl|\sum_{i=0}^{2^k-1} (\delta A_{t_i^k,u_i^k,t_{i+1}^k}-\E^{u_i^k}\delta A_{t_i^k,u_i^k,t_{i+1}^k})\I(t_{i}^k\le\tau)\Bigr|+\sum_{i=0}^{2^k-1} |\E^{u_i^k} \delta A_{t_i^k,u_i^k,t_{i+1}^k}|\nn\\
&\quad+\max_{i=0,\hdots,2^k-1}|A_{u_i^k,t_{i+1}^k}|\nn\\
&\le I_1+I_2+I_3.
\end{align}
We begin with the treatment of $I_1$. Denoting  $f_0^k:=0$, 
$f_j^k:=\sum_{i=0}^{j-1}(\delta A_{t_i^k,u_i^k,t_{i+1}^k}-\E^{u_i^k}\delta A_{t_i^k,u_i^k,t_{i+1}^k})$, $1\le j\le {2^k}$, we note that $(f_j^k)_{j=0,..., 2^k}$ is a martingale with respect to the filtration $(\G_j^k)_{j=0,...,2^k}$, $\G_j^k:=\F_{u_j^k}$. 
To bound the first sum we apply the Rosenthal--Burkholder inequality (\cref{p:rbiq}). We get
\begin{align}
\|I_1\|_{L_m(\Omega)} &\le \bigl\|\max_{j=0,...,2^k} |f_j^k| \bigr\|_{L_m(\Omega)}\nn\\
&\le C \Bigl\|\sum_{i=0}^{2^k-1}\E [|f_{i+1}^k-f_{i}^k|^2|\F_{u_i^k}]\Bigr\|_{L_{m/2}(\Omega)}^{\frac12}
+C\bigl\|\max_{i=0,...,2^k-1} |f_{i+1}^k-f_{i}^k| \bigr\|_{L_m(\Omega)}\label{step2st1}\\
&\le C \Bigl\|\sum_{i=0}^{2^k-1}\E [|f_{i+1}^k-f_{i}^k|^2|\F_{u_i^k}]\Bigr\|_{L_{m/2}(\Omega)}^{\frac12}
+C\Bigl(\sum_{i=0}^{2^k-1} \E |f_{i+1}^k-f_{i}^k|^n \Bigr)^{1/n}\label{step2st2}\\
&\le C \Bigl\|\sum_{i=0}^{2^k-1}
\E^{u_i^k}|\delta A_{t_i^k,u_i^k,t_{i+1}^k}-\E^{u_i^k}\delta A_{t_i^k,u_i^k,t_{i+1}^k}|^2\Bigr\|_{L_{m/2}(\Omega)}^{1/2}\nn\\
&\quad+C \Bigl(\sum_{i=0}^{2^k-1}
\E|\delta A_{t_i^k,u_i^k,t_{i+1}^k}-\E^{u_i^k}\delta A_{t_i^k,u_i^k,t_{i+1}^k}|^n\Bigr)^{1/n}\nn\\
&\le C \Bigl\|\sum_{i=0}^{2^k-1}
\E^{u_i^k}|\delta A_{t_i^k,u_i^k,t_{i+1}^k}|^2\Bigr\|_{L_{m/2}(\Omega)}^{1/2}+C \Bigl(\sum_{i=0}^{2^k-1}
\E|\delta A_{t_i^k,u_i^k,t_{i+1}^k}|^n\Bigr)^{1/n}\label{step2st3}\\
&\le C\Gamma_1|t-s|^{\alpha_1}2^{-k\alpha_1} \Bigl\|\sum_{i=0}^{2^k-1}
\lambda_1(t_i^k, t_{i+1}^k)^{2\beta_1}\Bigr\|_{L_{m/2}(\Omega)}^{1/2}+\!C\Gamma_3|t-s|^{\alpha_3}2^{-k (\alpha_3-\frac1n)}\label{step2st4}\\
&\le C\Gamma_1|t-s|^{\alpha_1}\|\lambda_1(s, t)^{\beta_1}\|_{L_m(\Omega)}2^{-k(\alpha_1+(\beta_1-\frac12)\wedge0)} \!
+\!C\Gamma_3|t-s|^{\alpha_3}2^{-k (\alpha_3-\frac1n)},\label{step2st}
\end{align}
for $C=C(m,n)>0$, where \eqref{step2st1} is the Rosenthal--Burkholder inequality, \eqref{step2st2} follows from the fact that $n\ge m$, \eqref{step2st3} uses the elementary inequality \eqref{sigmaalg} together with the bound $|a+b|^n\le C(n)(|a|^n+|b|^n)$ valid for all $a,b\in\R^d$, \eqref{step2st4} utilizes  assumptions \eqref{Rcon:s1} and \eqref{Rcon:s3} of the theorem, and finally \eqref{step2st} follows from the subadditivity of the control \eqref{rcont} and the Jensen inequality: 
$$
\sum_{i=0}^{2^k-1} \lambda_1(t^k_i, t^k_{i+1})^{2\beta_1} \le  \lambda_1(s, t)^{2\beta_1} 2^{k(1-2\beta_1  )\vee0}.
$$

Bounding $I_2$ in \eqref{step1st} is easy. Indeed, using \eqref{Rcon:s2}, subadditivity of the control and the H\"older inequality, we see that
\begin{align}\label{ivan}
\|I_2\|_{L_m(\Omega)}&\le \Gamma_2|t-s|^{\alpha_2}2^{-k\alpha_2}\Bigl\|\sum_{i=0}^{2^k-1}\lambda(t_i^k,t_i^{k+1})^{\beta_2}\Bigr\|_{L_m(\Omega)}\nn\\
&\le  \Gamma_2|t-s|^{\alpha_2}2^{-k(\alpha_2+(\beta_2-1)\wedge0)} \|\lambda_2(s, t)^{\beta_2}\|_{L_m(\Omega)}.
\end{align}
It is also not difficult to bound $I_3$. Using \eqref{Rcon:s2} and the fact than $n\ge m$, we see that
\begin{align}\label{itree}
\|I_3\|_{L_m(\Omega)}& \le  \|I_3\|_{L_n(\Omega)} =\|\max_{i=0,\hdots,2^k-1}|A_{u_i^k,t_{i+1}^k}|\|_{L_n(\Omega)}\le\Bigl( \sum_{i=0}^{2^k-1} \E| A_{u_i^k,t_{i+1}^k}|^n\Bigr)^{\frac1n}\nn\\
&\le \Gamma_3|t-s|^{\alpha_3}2^{-k (\alpha_3-\frac1n)}.
\end{align}
Substituting  \eqref{step2st}, \eqref{ivan}, and \eqref{itree} into \eqref{step1st}, we finally get 
\begin{align*}
\|A^{k}_{s,t}-A^{k+1}_{s,t}\|_{L_m(\Omega)}&\le 
C\Gamma_1|t-s|^{\alpha_1}\|\lambda_1(s, t)^{\beta_1}\|_{L_m(\Omega)}2^{-k(\alpha_1+(\beta_1-\frac12)\wedge0)}\\
&\quad+ \Gamma_2|t-s|^{\alpha_2}2^{-k(\alpha_2+(\beta_2-1)\wedge0)} \|\lambda_2(s, t)^{\beta_2}\|_{L_m(\Omega)}
+C\Gamma_3|t-s|^{\alpha_3}2^{-k (\alpha_3-\frac1n)}.
\end{align*}
Summing this inequality over $k$, and using the fact that, thanks to \eqref{alphabeta}, $\alpha_1+\beta_1>1/2$, $\alpha_2+\beta_2>1$, $\alpha_3>\frac1n$ we get for any $k\in\N$ (recall that $A^0_{s,t}=A_{s,t}\I_{s\le\tau}$)
\begin{align*}
\|A^{k}_{s,t}-A_{s,t}\I_{s\le \tau}\|_{L_m(\Omega)}&\le \sum_{i=0}^{k-1}
\|A^{i+1}_{s,t}-A^{i}_{s,t}\|_{L_m(\Omega)}\\
&\le C\Gamma_1|t-s|^{\alpha_1}\|\lambda_1(s, t)^{\beta_1}\|_{L_m(\Omega)}+ C\Gamma_2|t-s|^{\alpha_2} \|\lambda_2(s, t)^{\beta_2}\|_{L_m(\Omega)}\\
&\quad+C\Gamma_3|t-s|^{\alpha_3},
\end{align*}
for $C=C(\alpha_1, \alpha_2,\alpha_3, \beta_1, \beta_2,m,n)>0$.
An application of Fatou's lemma, assumption \eqref{Rcon:3}, and \eqref{Rcon:s3}  yields the desired result \eqref{Rres}.
\end{proof}

\subsection{Uniqueness of solutions to the Young differential equation}\label{S:YDE}

An important tool in establishing uniqueness of SDE \eqref{mainSDE} is the following result related to (deterministic) Young differential equations. 
\begin{proposition}\label{T:YoungODE}
Let $d=1$. Suppose that $X\in \C^{p-\var}([0,1],\R)$, $Y\in \C^{q-\var}([0,1],\R)$, $p,q\ge1$, $\frac1{p}+\frac1{q}>1$ and $X$, $Y$ are continuous. Then the following ODE
\begin{equation}\label{YoungODE}
Y_t=\int_0^t Y_s \,dX_s,\quad t\in[0,1],\quad Y_0=0,
\end{equation}
has only trivial zero solution $Y\equiv0$.
\end{proposition}
Before we begin the proof let us emphasize that we do not assume here $p<2$; otherwise the result is standard, see, e.g., \cite[Chapter 3, Theorem~3.1]{Bau}. Actually, we are interested in the opposite situation: $Y$ is a ``nice'' relatively smooth process with  $q\in[1,2)$ and $X$ is a very irregular process. Without the additional assumption on regularity of  $Y$, \eqref{YoungODE} would become a rough differential equation and one would have to lift $X$ to rough path in order to make sense of the integral. We consider  \eqref{YoungODE} as a Young differential equation and are interested in uniqueness among all smooth enough solutions.     
\begin{proof}
The main idea of the proof was communicated to us by Fedja Nazarov.

It is sufficient to consider the case $p>2>q$. Indeed, otherwise, if $p\le q$, then $p<2$ and uniqueness of solutions to \eqref{YoungODE} follows from \cite[Chapter 3, Theorem~3.1]{Bau}.
Assume that  \eqref{YoungODE} has a non-zero solution. Without loss of generality, we assume that $Y_r>0$ for some $r\in(0,1]$ (the case when $Y$ is nonpositive is done in exactly the same way). Then there exists $\eps>0$, $0\le s \le t\le 1$ such that 
\begin{equation*}
Y_s=\frac\eps{2K},\quad Y_t=\eps,\quad \inf_{r\in[s,t]} Y_r= \frac\eps{2K},
\end{equation*}
where we defined $K:=\exp(\|X\|_{L_\infty([0,1])})$. We claim now that
$X\in \C^{q-\var}([s,t],\R)$.

Indeed, by  Young-L\'oeve estimate \cite[Theorem~6.8]{FV2010} we have for any $(s',t')\in\Delta_{[s,t]}$
\begin{equation*}
|Y_{t'}-Y_{s'}-Y_{s'}(X_{t'}-X_{s'})|=\Bigl|\int_{s'}^{t'} Y_rdX_r-Y_{s'}(X_{t'}-X_{s'})\Bigr|\le C \|Y\|_{q-\var;[s',t']},
\end{equation*}
for $C=C(p,q)>0$. 
Therefore, taking into account that $Y_{s'}\ge \frac\eps{2K}$ for $s'\in[s,t]$ we get
\begin{equation*}
|X_{t'}-X_{s'}|\le \frac{2K}{\eps}(|Y_{t'}-Y_{s'}|+C \|Y\|_{q-\var;[s',t']})\le  \frac{2K}{\eps}(C+1) \|Y\|_{q-\var;[s',t']}.
\end{equation*}
This implies $ \|X\|_{q-\var;[s,t]}\le \frac{2K}{\eps}(C+1) \|Y\|_{q-\var;[s,t]}$, and thus $X\in \C^{q-\var}([s,t],\R)$. 

However, since $q<2$ and $X,Y\in \C^{q-\var}([s,t],\R)$,  equation \eqref{YoungODE} has a unique solution on $[s,t]$ by \cite[Chapter 3, Theorem~3.1]{Bau}. T his solution is given by $Y_r=Y_{s}e^{X_r}$, $r\in[s,t]$. Thus,
\begin{equation*}
\eps=Y_t=Y_s e^{X_t}\le \frac\eps{2K}K=\frac\eps2,
\end{equation*}
which is a contradiction. Therefore \eqref{YoungODE} has only zero solution.
\end{proof}

\section{Key integral bounds}\label{S:key}

In this section we establish key integral bounds which allow to establish existence and uniqueness of solutions to equation \eqref{mainSDE}. Without loss of generality, we will be working on the time interval $[0,1]$. Fix $d\in\N$, $H\in(0,1)$, and the filtration $(\F_t)_{t\in[0,1]}$ such that $W^H$ is an $(\F_t)$--fractional Brownian motion.  Recall representation \eqref{WB}. For $(s,t)\in\Delta_{[0,1]}$ consider a process
\begin{equation}\label{processV}
V_{s,t}:=W^H_t-\E^s W^H_t=\int_s^t K_H(t,r) dB_r.
\end{equation}

We begin with the following basic bound. Its proof uses  just the original stochastic sewing lemma from \cite{LeSSL}.
\begin{lemma}\label{L:firstb}
Let $f\colon\R^d\to\R$ be a bounded measurable function. Suppose that $f\in\B^\alpha_{q}$ where $q\in[2,\infty]$, $\alpha<0$.  Assume  that
\begin{equation}\label{ahcond}
 % \alpha H>-\frac12
 \alpha>-\frac1{2H}.
 %-\frac{Hd}q,-\frac12-\frac{d}{2q}).
 \end{equation}
Let $m\in[2,\infty)$ such that $m\le q$.
Then there exists a constant $C=C(H,d, \alpha,m,q)>0$ such that for any $0\le S_0< S\le T\le 1$, any \textbf{deterministic} measurable function $x\colon[0,1]\to\R^d$  one has
\begin{align}\label{mainbound1}
&\Bigl\|\int_{S}^{T} f(V_{S_0,r}+x_r)dr\Bigr\|_{L_m(\Omega)}\nn\\
&\quad\le  C \|f\|_{\B^{\alpha}_q}(T-S)^{1+\alpha H}\bigl((S-S_0)^{-\frac{Hd}q} +\I(H<1/2)
(S-S_0)^{-\frac{d}{2q}} (T-S)^{ \frac{d}{2q}-\frac{Hd}{q}}\bigr).
\end{align}
%Further, in the case $p=\infty$, \eqref{mainbound1} is valid also for $S=S_0$.
\end{lemma}
\begin{proof}
We apply stochastic sewing lemma from \cite{LeSSL}, that is, \cref{T:SSLst}, to the following processes:
\begin{align*}
&\A_t:=\int_S^{t}f(V_{S_0,r}+x_r)\,dr,\quad 	t\in[S,T]; &A_{s,t}:=\E^s\int_s^t f(V_{S_0,r}+x_r)dr,\quad s,t\in\Delta_{[S,T]}.
\end{align*}
We stress once again that the process $x$ is deterministic. Let us check that all the conditions of \cref{T:SSLst} are satisfied.
	
Using heat kernel bound \eqref{efp}, we  derive for  $s,t\in\Delta_{[S,T]}$
\begin{align}\label{Astl41}
\|  A_{s,t}\|_{L_q(\Omega)}&\le \int_s^t \|\E^s f(V_{S_0,r}+x_r)\|_{L_q(\Omega)}\, dr\nn\\
&\le C \|f\|_{\B^{\alpha}_q}(t-s)^{1+\alpha H}\bigl((S-S_0)^{-\frac{Hd}q} +\I(H<1/2)
 (S-S_0)^{-\frac{d}{2q}} (t-s)^{ \frac{d}{2q}-\frac{Hd}{q}}\bigr),
\end{align}
for $C=C(H,d,\alpha,q)>0$. 
We see that thanks to \eqref{ahcond}, $1+\alpha H>1/2$; further if $H<1/2$, then    
$1+ H(\alpha-\frac{d}q)+\frac{d}{2q}>1/2$. Thus,  condition \eqref{con:s1} holds  thanks to an obvious bound 
$$\|\delta A_{s,u,t}\|_{L_m(\Omega)}\le \| A_{s,t}\|_{L_q(\Omega)}+\| A_{s,u}\|_{L_q(\Omega)}+\| A_{u,t}\|_{L_q(\Omega)},
$$
where we also used that $m\le q$.

Next, note that for any $S\le s\le u\le t\le T$ we have
\begin{equation*}%\label{Part02}
		\delta A_{s,u,t}=\E^s \int_u^t f(V_{S_0,r}+x_r)\, dr-\E^u \int_u^t f(V_{S_0,r}+x_r)\, dr,
\end{equation*}
which implies that $\E^s \delta A_{s,u,t}=0$. Thus condition \eqref{con:s2} holds.

Thus it remains to verify \eqref{con:s3}.  Let $\Pi:=\{S=t_0,t_1,...,t_k=t\}$ be an arbitrary partition of $[S,t]$. Denote by $\Di{\Pi}$ its mesh size. Note that thanks to our choice of the processes $\A$ and $A$,  for any $i\in[0,k-1]$ we have $\E^{t_i}(\A_{t_{i+1}}-\A_{t_{i}}-A_{t_i,t_{i+1}})=0$ and $\A_{t_{i+1}}-\A_{t_{i}}-A_{t_i,t_{i+1}}$ is $\F_{t_{i+1}}$--measurable.   Therefore,  $\A_t-\sum_{i=0}^{k-1} A_{t_i,t_{i+1}}$ is a sum of martingale differences. Then, by the Burkholder--Davis--Gundy inequality,  %and applying bound
%\eqref{bounddiffcond}  to the function $f\colon z\mapsto %h(z,r,y,\omega)$ we have for any $x\in\0$
\begin{align*}%\label{vercond4}
\Bigl\|\sum_{i=0}^{k-1} (\A_{t_{i+1}}-\A_{t_i}-A_{t_i,t_{i+1}})\Bigr\|_{L_2(\Omega)}^2
&\le 
C \sum_{i=0}^{k-1}\|\A_{t_{i+1}}-\A_{t_{i}}- A_{t_i,t_{i+1}}\|_{L_2(\Omega)}^2\nn\\
&\le C\|f\|_{L_\infty(\R^d)}\sum_{i=0}^{k-1} (t_{i+1}-t_i)^2 \nn\\
&\le C\|f\|_{L_\infty(\R^d)}(t-S)|\Pi|
\end{align*}
for $C>0$. 
Therefore, $\sum_{i=0}^{k-1} A_{t_i,t_{i+1}}$ converges to $\A_{t}$ in probability as $|\Pi|\to0$ and hence condition~\eqref{con:s3} holds.

Thus, all the conditions of \cref{T:SSLst}  are satisfied and \eqref{est:ssl1} implies
\begin{align*}
&\Bigl\|\int_{S}^{T}  f(V_{S_0,r}+x_r)\,dr\Bigr\|_{L_m(\Omega)}\\
&\quad\le 
C \|f\|_{\B^{\alpha}_q}(T-S)^{1+\alpha H}\bigl((S-S_0)^{-\frac{Hd}q} +\I(H<1/2)
(S-S_0)^{-\frac{d}{2q}} (T-S)^{ \frac{d}{2q}-\frac{Hd}{q}}\bigr),
 \end{align*}
 for $C=C(H,d,\alpha,m,q)>0$,
where again we used \eqref{Astl41}. This implies \eqref{mainbound1}.
\end{proof}

The next step is to remove the singularity at $S_0$ in \eqref{mainbound1} as well as the unnatural restriction $m \le q$. This is done using the taming singularities lemma and then applying the John–Nirenberg inequality, \cref{P:sing,P:BMO}. The former proposition allows us to take $S = S_0$ in \eqref{mainbound1}, while the latter allows us to take $m$ arbitrarily large.

\begin{lemma}\label{C:ourbounds}
Let $f\colon\R^d\to\R$ be a bounded measurable function. Suppose that $f\in\B^\alpha_{q}$ where $q\in[2,\infty]$, $\alpha<0$ such that 
\begin{equation}\label{extracondTS}
% (\alpha-\frac{d}{q}) H>\max(-1,-\frac12-\frac{Hd}q).
\alpha>-\frac1{2H}\tand \alpha-\frac{d}{q}>-\frac1H.
\end{equation}
Let $m\in[1,\infty)$.
Then there exists a constant $C=C(H,d,\alpha,m,q)>0$ such that for any $0\le S_0\le S\le T\le 1$, any \textbf{deterministic} measurable function $x\colon[0,1]\to\R^d$  one has
\begin{equation}\label{mainbound2}
\Bigl\|\int_S^T f(V_{S_0,r}+x_r)dr\Bigr\|_{L_m(\Omega)}\le 
C \|f\|_{\B^{\alpha}_q}(T-S)^{1+\alpha H-\frac{Hd}q}.
\end{equation}
%
%
%\item If
%$$
%\alpha H>\max(-1+\beta,-1+\beta+\frac{Hd}{p}- \frac{d}{2p}),
%$$
%then there exists a constant $C=C(H,d,\alpha,p,\beta)$ such that for any $0\le S_0<S< T\le 1$ one has
%\begin{equation}\label{mainbound3}
%\Bigl\|\int_S^T (T-r)^{-\beta }f(V_{S_0,r})dr\Bigr\|_{L_p(\Omega)}\le 
%C \|f\|_{\B^{\alpha,p}}(T-S)^{1+\alpha H-\beta}(S-S_0)^{-\frac{Hd}p}.
%\end{equation}
%
%\item Under the conditions of part (ii) of the corollary, there exists a constant $C=C(H,d,\alpha,p,\beta)$  such that for any 
%$0\le S\le T\le T_0\le 1$ one has
%\begin{equation}\label{mainbound4}
%\Bigl\|\int_S^T (T_0-r)^{-\beta }f(V_{S,r})dr\Bigr\|_{L_p(\Omega)}\le 
%C \|f\|_{\B^{\alpha,p}}(T_0-S)^{-\beta}(T-S)^{1+\alpha H-\frac{Hd}p}.
%\end{equation}
%\end{enumerate}
%
\end{lemma}
\begin{proof} 
\textbf{Step 1}. We prove \eqref{mainbound2} for $m=1$. We tame the singularities with \cref{P:sing}. Fix $(S,T)\in\Delta_{[0,1]}$ and put
\begin{equation*}
Y_t:=\int_{S+t}^T  f(V_{S_0,r}+x_r)\, dr, \quad t\in(0,T-S].
\end{equation*}
We see that \eqref{extracondTS} implies $\alpha H >-\frac12$, and thus by \cref{L:firstb}, for $(s,t)\in\Delta_{[0,T-S]}$ we have
\begin{equation*}
\|Y_t-Y_s\|_{L_q(\Omega)}\le C \|f\|_{\B^{\alpha}_q}(t-s)^{1+\alpha H}\bigl(s^{-\frac{Hd}q} +\I(H<1/2)s^{-\frac{d}{2q}} (t-s)^{ \frac{d}{2q}-\frac{Hd}{q}}\bigr),
\end{equation*}
for $C=C(H,d,\alpha,q)>0$.
We also see from \eqref{extracondTS} that  $1+\alpha H>\frac{Hd}q$ and $1+\alpha H+ \frac{d}{2q}-\frac{Hd}{q}>\frac{d}{2q}$. Hence all the conditions of \cref{P:sing} are satisfied and we get that for any $\eps\in(0,T-S]$,
\begin{align*}
\Bigl\|\int_{S+\eps}^T  f(V_{S_0,r}+x_r)\, dr\Bigr\|_{L_1(\Omega)}&\le \Bigl\|\int_{S+\eps}^T  f(V_{S_0,r}+x_r)\, dr\Bigr\|_{L_q(\Omega)}=\|Y_{T-S}-Y_\eps\|_{L_q(\Omega)}\\
&\le C \|f\|_{\B^{\alpha}_q}(T-S-\eps)^{1+\alpha H-\frac{Hd}q},
\end{align*}
where $C=C(H,d,\alpha,q)>0$.
By passing to the limit as $\eps\to0$, we get \eqref{mainbound2} for $m=1$; here we used Fatou's lemma and boundedness of the function $f$. 

\textbf{Step 2}. Now let us prove \eqref{mainbound2} for all $m\in[1,\infty)$;  actually, the proof of Step~1 works  for $m\le q$, but not for large $m$. We apply \cref{P:BMO}.  Put 
$$
\A_t:=\int_S^t  f(V_{S_0,r}+x_r)\, dr, \quad t\in[S,T].
$$
Then for any $(s,t)\in\Delta_{[S,T]}$ we have
\begin{equation*}
\E^s |\A_t-\A_s|=\E^s  \Bigl|\int_s^t  f(\E^s V_{S_0,r}+V_{s,r}+x_r)\, dr\Bigr|=
\E \Bigl|\int_s^t f(V_{s,r}+y_r)dr\Bigr| \bigg\rvert_{y_r= \E^s V_{S_0,r}+x_r},
\end{equation*}
where we used independence of $(V_{s,r})_{r\ge s}$ and $\F_s$. Applying \eqref{mainbound2} with $m=1$ obtained from Step~1, we get 
\begin{equation*}
\E^s |\A_t-\A_s|\le C \|f\|_{\B^{\alpha}_q}(t-s)^{1+\alpha H-\frac{Hd}q}\le C \|f\|_{\B^{\alpha}_q}(T-S)^{1+\alpha H-\frac{Hd}q},\quad (s,t)\in\Delta_{[S,T]},
\end{equation*}
where $C=C(H,d,\alpha,q)>0$.
% where we used that \eqref{extracondTS} implies that $1+\alpha H-\frac{Hd}q>0$. 
Thus, \eqref{BMOcon} is satisfied with $\Gamma:=C \|f\|_{\B^{\alpha}_q}(T-S)^{1+\alpha H-\frac{Hd}q}$. Clearly, the process $\A$ is continuous since $f$ is bounded. Thus, all the conditions of \cref{P:BMO} are satisfied and hence we get  for $m\in[1,\infty)$
\begin{align*}
\Bigl\|\int_S^T f(V_{S_0,r}+x_r)dr\Bigr\|_{L_m(\Omega)}&=\|\A_T-\A_S\|_{L_m(\Omega)}\le C \|f\|_{\B^{\alpha}_q}(T-S)^{1+\alpha H-\frac{Hd}q},
\end{align*}
where $C=C(H,d,\alpha,m,q)>0$, 
which is \eqref{mainbound2}.
\end{proof}
\begin{remark}\label{R:newthings}
Note that if $f\in\B^\alpha_q$, then $f\in\B^{\alpha-\frac{d}q}_\infty$. Therefore, if we apply \cref{L:firstb} alone, we would obtain exactly the same bound \eqref{mainbound2}, but under the much more restrictive condition $\alpha - \frac{d}{q} > -\frac{1}{2H}$. The above proof shows that an application of the taming singularities lemma (\cref{P:sing}) leads to a better condition \eqref{extracondTS}. Thus, a combination of the taming singularities technique and stochastic sewing gives much better results than stochastic sewing alone. \cref{P:BMO} allows us to further refine the obtained results. Namely, when combined with stochastic sewing, it allows us to bound all the moments of the integral (rather than just the first few moments).
\end{remark}

Next, we need an analogue of \eqref{mainbound2}, but for the difference of functions.

\begin{corollary}\label{C:twopoint}
Let $f\colon\R^d\to\R$ be a bounded measurable function. Suppose that $f\in\B^\alpha_{q}$ where $q\in[2,\infty]$, $\alpha<0$. Let $\lambda\in[0,1]$, $m\in[1,\infty)$.  Assume that
% \begin{equation}\label{alphalambda}
% (\alpha-\frac{d}{q}-\lambda) H>\max(-1,-\frac12-\frac{Hd}q),
% \end{equation}
\begin{align}\label{alphalambda}
	\alpha- \lambda>-\frac1{2H}\tand
	\alpha-\frac dq- \lambda>-\frac1H
\end{align}

Then there exists a constant $C=C(H,d,\alpha,\lambda,m,q)>0$  such that the following bound holds for  any \textbf{deterministic} measurable function $x\colon[0,1]\to\R^d$, $y\in\R^d$, $(S,T)\in\Delta_{[0,1]}$:
\begin{equation}\label{mainbound2dif}
\Bigl\|\int_S^T (f(V_{S,r}+x_r+y)-f(V_{S,r}+x_r))\,dr\Bigr\|_{L_m(\Omega)}\le C \|f\|_{\B^{\alpha}_q}|y|^\lambda (T-S)^{1+H(\alpha-\lambda-\frac{d}q)}.
\end{equation}
\end{corollary}
\begin{proof}
Introduce the shift operator $\tau_yf(z):=f(y+z)$, $y,z\in\R^d$. Let us apply \eqref{mainbound2} to the function $\tau_y f -f$ with $\alpha-\lambda$ in place of $\alpha$. We see that condition \eqref{alphalambda} ensures that  \eqref{extracondTS} holds. Then we get 
\begin{align*}
\Bigl\|\int_S^T (f(V_{S,r}+x_r+y)-f(V_{S,r}+x_r))\,dr\Bigr\|_{L_m(\Omega)}&\le
		C\|\tau_y f -f\|_{\B^{\alpha-\lambda}_q}(T-S)^{1+H(\alpha-\lambda-\frac{d}q)}\\
&\le C \|f\|_{\B^{\alpha}_q}|y|^\lambda(T-S)^{1+H(\alpha-\lambda-\frac{d}q)},
\end{align*}
where $C=C(H,d,\alpha,\lambda,m,q)>0$ and  the last inequality follows from \eqref{fxydif}. This yields \eqref{mainbound2dif}.
\end{proof}

\begin{remark} By taking in \eqref{mainbound2dif} $S=0$, $x\equiv0$, $f=\delta_u$, $u\in\R^d$,
 $\alpha=-\frac{d}{2}$, $q=2$, $\lambda\in[0,(\frac{1}{2H}-\frac{d}{2})\wedge1)$ we get that if $Hd<1$, then 
	\begin{equation*}
		\|L^W(u-y,t)-L^W(u,t)\|_{L_m(\Omega)}\le C |y|^{\lambda} t^{1-H(d+\lambda)},\quad t\in[0,1],
	\end{equation*}	
	where $L^W$ is the local time of fractional Brownian motion $W^H$  and we also used \eqref{fxydif}. This provides a short and direct proof of Xiao's result about moments of local time \cite[Lemma~2.5]{Xiao}. 
	%Further, by taking $f$ to a partial derivative of order $k$ of  $\delta_u-\delta_v$ for $k=\lfloor \frac{1}{2H}-\frac{d}{2}\rfloor$, one gets a similar bound for the derivatives of local times, extending Xiao's result. 
	It just remains to justify that \cref{C:ourbounds} is applicable also for distributions and not only for bounded functions, which is done in the proof of \cref{T:loc}(i) by approximating a distribution by smooth functions (note that the right-hand side of \eqref{mainbound2} depends only on $\|f\|_{\B^\alpha_q}$ and does not depend on higher Besov norms of $f$). 
\end{remark}

Now we are ready to move on to the main lemma of this section, which allows us to replace a deterministic drift in \eqref{mainbound2} with a stochastic drift. It is clear that such a replacement should come with certain restrictions: one must impose that the stochastic drift $\psi$ is more regular (in a certain sense) than the driving noise $W^H$. Indeed, otherwise, one could simply take $\psi = -W$, thereby completely canceling the noise. An important observation is that it is much better to measure regularity on the $p$-variation scale rather than in the Hölder scale. Indeed, if $b \in L_p(\R^d)$, then the drift $\int b(X_r) , dr$ in \Gref{x;b} is only $(1 - \frac{Hd}{p})$-Hölder (and this exponent can be arbitrarily small), but it always has finite $1$-variation.

The proof of the lemma strongly relies on the Rosenthal-type stochastic sewing, \cref{T:RoSSL}, and the integral bounds already obtained in this section.

\begin{lemma}\label{L:driftb2}
Let $m\in[2,\infty)$. 
Let $f\colon\R^d\to\R$ be a measurable function, $f\in \B^\alpha_{q}$, where $q\in [2,\infty]$, $\alpha<0$. Suppose further that $f$ is bounded   or $f\ge0$. Let $z\colon[0,1]\times\Omega\to\R^d$ be a (possibly \textbf{non-deterministic}) continuous process adapted to the filtration $(\F_t)_{t\ge0}$
such that 
\begin{equation}\label{finvar}
\|\,\|z\|_{1-\var;[0,1]}\|_{L_m(\Omega)}<\infty.	
\end{equation}
Let $\tau$ be a measurable random variable taking values in $[0,1]$. Assume  
 that the following conditions hold:
\begin{align}\label{maincond1}
&\Bigl(\alpha-\frac{d}q\Bigr)H>-1+H,\\
&\alpha>\bigl(\frac{Hd}q\wedge\frac12\bigr)-\frac1{2H}.
\label{maincond2}
\end{align}

Then there exists a constant $C=C(H,d,\alpha,m,q)>0$  such that for any $(s,t)\in\Delta_{[0,1]}$:
\begin{align}\label{finbound}
\Bigl\|\int_{s\wedge\tau}^{t\wedge\tau} f(W_r^H+z_r)dr\Bigr\|_{L_m(\Omega)}\le& C \|f\|_{\B^{\alpha}_q}(t-s)^{1+H(\alpha-\frac{d}q-1)}\|\,\|z\|_{1-\var;[s,t]}\|_{L_m(\Omega)}\nn\\
&+C \|f\|_{\B^{\alpha}_q}(t-s)^{1+ H(\alpha-\frac{d}{q})}.
\end{align}
\end{lemma}

\begin{proof}[Proof of \cref{L:driftb2}]
\textbf{Step~1}. Assume further that $f\in\C^1(\R^d)$.
First of all, we note that conditions \eqref{maincond1}, \eqref{maincond2} implies that for $\mu:=\frac{Hd}{q}\wedge\frac12$ the following condition holds:
\begin{equation}\label{equationmu}
(\alpha-\mu)H>-\frac1{2}\quad\text{and}\quad 
(\alpha-\frac{d}q-\mu) H>\max(-\frac12-\mu,-1).
\end{equation}
Indeed, the first inequality of \eqref{equationmu} is just \eqref{maincond2}. For the second inequality of \eqref{equationmu}, we note that if $\frac{Hd}q\le\frac12$, then $\mu=\frac{Hd}q$ and by \eqref{maincond2}
\begin{equation*}
(\alpha-\frac{d}q-\mu)H>(\mu-\frac1{2H}-\frac{d}q-\mu)H=-\frac12-\frac{Hd}q=-\frac12-\mu.
\end{equation*}
If, alternatively, $\frac{Hd}q>\frac12$, then $\mu=\frac12$ and \eqref{maincond1} implies
\begin{equation*}
(\alpha-\frac{d}q-\mu)H>-1+H-\frac12H>-1=-\frac12-\mu.
\end{equation*}
Thus, in both cases the second inequality of \eqref{equationmu} also holds.

We apply the  Rosenthal-type stochastic sewing  with random controls, \cref{T:RoSSL}, to the following processes
\begin{align}%\label{processA}
&\A_t:=\int_0^t f(W_r^H+z_r)dr,\quad t\in[0,1]; \nn\\
&A_{s,t}:=\int_s^t f(W_r^H+z_s)\,dr,\quad (s,t)\in\Delta_{[0,1]}.\label{incr}
\end{align}
Let us verify that all the conditions of \cref{T:RoSSL} are satisfied. Clearly, for any $0\le s<u<t\le 1$ we have
\begin{equation*}%\label{deltaAsut2}
	\delta A_{s,u,t}=\int_u^t (f(W_r^H+z_s)-f(W_r^H+z_u))dr.
\end{equation*}

Recall the definition of process $V$ in \eqref{processV}. We have
\begin{align*}
\E^u|\delta A_{s,u,t}|^2&=\E^u \Bigl|\int_u^t (f(\E^u W_r^H+V_{u,r}+z_s)-f(\E^u W_r^H+V_{u,r}+z_u))dr\Bigr|^2\nn\\
&=\E \Bigl|\int_u^t (f(V_{u,r}+x_r+y)-f(V_{u,r}+x_r))dr\Bigr|^2 \bigg\rvert_{x_r= \E^u W_r^H+z_u,\,y=z_s-z_u},
\end{align*}
where we used the independence of $\F_u$ and the process $(V_{u,r})_{r\ge u}$.

Now we apply \cref{C:twopoint} with $\lambda=\mu$. Note that  \eqref{alphalambda} is satisfied thanks to assumption \eqref{equationmu}. We get 
\begin{equation}\label{onecondcheck}
\E^u|\delta A_{s,u,t}|^2\le C \|f\|_{\B^{\alpha}_q}^2 |z_u-z_s|^{2\mu} (t-s)^{2+2H(\alpha-\mu-\frac{d}q)}\le C \|f\|_{\B^{\alpha}_q}^2 \|z\|_{1-\var;[s,t]}^{2\mu} (t-s)^{2+2H(\alpha-\mu-\frac{d}q)}
,\end{equation}
where $C=C(H,d,\alpha,\mu,q)>0$.
We note that $\|z\|_{1-\var;[s,t]}$ is a random control in the sense of \eqref{rcont}. By \eqref{equationmu}, we have  $\mu+1+(\alpha-\mu-\frac{d}q)H>\frac12$ and  condition \eqref{Rcon:s1} holds.

Further, using consequently \eqref{meanformula}, \eqref{Cbound}, and \eqref{KHT}, we derive for $0\le s\le u \le t \le 1$
\begin{align}\label{twocondcheck}
	|\E^u[\delta A_{s,u,t}]|&=\Bigl|\int_u^t \bigl(P_{\sigma^2(u,r)}f(\E^u W_r^H+z_s)-P_{\sigma^2(u,r)}f (\E^u W_r^H+z_u)\bigr)\,dr\Bigr|\nn\\
	&\le \|z\|_{1-\var;[s,u]}\int_u^t \|P_{\sigma^2(u,r)}f\|_{\C^1(\R^d)}\,dr\nn\\
	&\le C \|z\|_{1-\var;[s,t]} \|f\|_{\B^{\alpha}_q}\int_u^t (r-u)^{H(\alpha-\frac{d}q-1)}\,dr\nn\\
	&\le C \|z\|_{1-\var;[s,t]} \|f\|_{\B^{\alpha}_q}(t-u)^{1+H(\alpha-\frac{d}q-1)},
\end{align}
where $\sigma^2(u,r):=\int_u^t (K_H(t,r'))^2\,dr'$ and $C=C(H,d,\alpha,q)>0$.
Note  that  $\|z\|_{1-\var;[s,t]}$ is a random control and by \eqref{maincond1}, ${1+H(\alpha-\frac{d}q-1)>0}$. Therefore \eqref{Rcon:s2} is satisfied. 

Next, we see that for any $n\ge2$, $(s,t)\in\Delta_{[0,1]}$
\begin{align*}
\| A_{s,t}\|_{L_n(\Omega)}^n&=\E \E^s \Bigl|\int_s^t  f(W_r^H+z_s)\,dr\Bigr|^n=\E \E^s\Bigl|\int_s^t  f(\E^s W_r^H+V_{s,r}+z_s)\,dr\Bigr|^n\\
&=\E \Bigl[\E \Bigl|\int_s^t f(V_{s,r}+x_r)dr\Bigr|^n \bigg\rvert_{x_r= \E^s W_r^H+z_s}\Bigr] ,
\end{align*}
where we again used the independence of $\F_s$ and $(V_{s,r})_{r\ge s}$. We see that by  \eqref{maincond2},  $\alpha>-\frac1{2H}$. Further,   \eqref{maincond1} guarantees that $\alpha-\frac{d}q>-\frac1H$. Thus, \eqref{extracondTS} holds. Therefore, all the conditions of \cref{C:ourbounds} are satisfied and we get
\begin{equation}\label{threecondcheck}
\| A_{s,t}\|_{L_n(\Omega)}\le C \|f\|_{\B^{\alpha}_q}(t-s)^{1+\alpha H-\frac{Hd}{q}}
\end{equation}
for $C=C(H,d,\alpha,n,
q)>0$
By taking $n\ge m$ large enough, we see that \eqref{Rcon:s3} and the last inequality of \eqref{alphabeta} are also satisfied.

Finally, let us verify \eqref{Rcon:3}.  Let $\Pi:=\{S=t_0,t_1,...,t_k=t\}$ be an arbitrary partition of $[S,t]$.
Recalling definition \eqref{incr}, we note that for any $i=0,\hdots,k-1$
\begin{align*}
&|\A_{t_{i+1}\wedge \tau}-\A_{t_i\wedge \tau}-A_{t_i,t_{i+1}}\I(t_i\le \tau))|\\
&\quad =\Bigl|\int_{t_i\wedge \tau}^{t_{i+1}\wedge \tau} (f(W_r^H+z_r)- f(W_r^H+z_s))\,dr -\I(t_i\le \tau < t_{i+1})\int_{\tau}^{t_{i+1}}f(W_r^H+z_s)\,dr\Bigr|\\
&\quad \le\int_{t_i}^{t_{i+1}} |f(W_r^H+z_r)- f(W_r^H+z_s)|\,dr +|\Pi|\|f\|_{L_\infty(\R^d)}\I(t_i\le \tau< t_{i+1})\\
&\quad \le \|f\|_{\C^1(\R^d)}|\Pi|(\|z\|_{1-\var;[t_i,t_{i+1}]} +\I(t_i\le \tau< t_{i+1})).
\end{align*}
Then 
\begin{align*}
&\Bigl\|\A_{t\wedge \tau}-\A_{S\wedge \tau}-\sum_{i=0}^{k-1} A_{t_i,t_{i+1}}\I(t_i\le \tau))\Bigr\|_{L_1(\Omega)}\\
&\quad\le \Bigl\| \sum_{i=0}^{k-1} |\A_{t_{i+1}\wedge \tau}-\A_{t_i\wedge \tau}-A_{t_i,t_{i+1}}\I(t_i\le \tau)|\Bigr\|_{L_1(\Omega)}\\
&\quad\le \|f\|_{\C^1(\R^d)}|\Pi|(\|\,\|z\|_{1-\var;[0,1]}\|_{L_1(\Omega)} +1).
\end{align*}
Since $\|\,\|z\|_{1-\var;[0,1]}\|_{L_1(\Omega)}<\infty$ by \eqref{finvar}, we see that \eqref{Rcon:3} holds.

Therefore, all the conditions of \cref{T:RoSSL} are satisfied and we get 
for $(s,t)\in\Delta_{[0,1]}$ (recall \eqref{onecondcheck}, \eqref{twocondcheck}, \eqref{threecondcheck})
\begin{align*}
&\|\A_{t\wedge\tau}-\A_{s\wedge\tau}\|_{L_m(\Omega)}\\
&\quad= \Bigl\|\int_{s\wedge\tau}^{t\wedge\tau} f(W_r^H+z_r)dr\Bigr\|_{L_m(\Omega)}\\
&\quad \le C \|f\|_{\B^{\alpha}_q}(t-s)^{1+H(\alpha-\frac{d}q-\mu)}\|\,\|z\|_{1-\var;[s,t]}\|^\mu_{L_m(\Omega)}\\
&\qquad+C \|f\|_{\B^{\alpha}_q}(t-s)^{1+H(\alpha-\frac{d}q-1)}\|\,\|z\|_{1-\var;[s,t]}\|_{L_m(\Omega)}+C \|f\|_{\B^{\alpha}_q}(t-s)^{1+ H(\alpha-\frac{d}{q})}\\
&\quad = C \|f\|_{\B^{\alpha}_q}(t-s)^{1+ H(\alpha-\frac{d}{q})}\bigl((t-s)^{-1}\|\,\|z\|_{1-\var;[s,t]}\|_{L_m(\Omega)}+(t-s)^{-\mu}\|\,\|z\|_{1-\var;[s,t]}\|_{L_m(\Omega)}^\mu+1\bigr),
\end{align*}
where $C=C(H,d,\alpha,\mu,m,q)>0$. 
Note that for any $a,b>0$ one  has $a^{-\mu}b^{\mu}\le a^{-1}b+1$, since $\mu\in[0,1]$. Together with the above inequality this implies \eqref{finbound}.

\textbf{Step 2}. Assume now that $f$ is a bounded continuous function. For $k\in\Z_+$ put $f_k:=P_{1/k}f$. Then it is clear that $f_k$ converges to $f$ pointwise and $ \|f_k\|_{\B^{\alpha}_q}\le \|f\|_{\B^{\alpha}_q}$, see, e.g.,  \cite[Lemma~A.3(iii)]{ABLM}. Since $f_k\in\C^1(\R^d)$ for any $k\in\Z_+$, the desired bound \eqref{finbound} follows now directly from Step~1 and Fatou's lemma.

\textbf{Step 3}. Now let us treat discontinuous $f$. We begin with $f=\I_U$, where  $U\subset \R^d$ is  an open set of Lebesgue measure $\delta>0$.  By Urysohn's lemma, there exists a sequence of bounded continuos functions $f_n\colon\R^d\to[0,1]$, $n\in\Z_+$, such that $0\le f_n\le \1_U$ for any $n\in\Z_+$ and $f_n(x)\to \1_U(x)$ for any $x\in\R^d$ as $n\to\infty$. Then $f_n(W_r^H+z_r)\to \1_U(W_r^H+z_r)$ as $n\to\infty$ for any $r\in[0,1]$, $\omega\in\Omega$. 

Set 
\begin{equation*}
\qh:=q,\,\,\, \text{if $q<\infty;$}\qquad \qh:=\frac{d}{|\alpha|}\vee1,\,\,\,\text{if $q=\infty.$}
\end{equation*}
Note that $L_{q}(\R^d)\subset \B^\alpha_{q}$ and  $L_{\frac{d}{|\alpha|}\vee1}(\R^d)\subset \B^0_{\frac{d}{|\alpha|}\vee1}\subset \B^{\alpha\vee(-d)}_{\infty}\subset \B^\alpha_\infty$. Thus, we always have the embedding
\begin{equation}\label{embedqq}
L_{\widehat q}(\R^d)\subset \B^\alpha_{q}.
\end{equation}
Let us apply  bound  \eqref{finbound} proved in Step~2 to the bounded continuous function $f_n$. 
Using \eqref{embedqq}, we see that
$\|f_n\|_{\B^{\alpha}_{\qh}}\le \|f_n\|_{L_{\qh}(\R^d)}\le  \|\I_U\|_{L_{\qh}(\R^d)}\le \delta^{1/{\qh}}$. Hence  an application of Fatou's lemma implies
\begin{align}\label{epsres}
\Bigl\|\int_{s\wedge\tau}^{t\wedge\tau} \I_U(W_r^H+z_r)dr\Bigr\|_{L_m(\Omega)}&\le \liminf_{n\to\infty}\Bigl\|\int_{s\wedge\tau}^{t\wedge\tau} f_n(W_r^H+z_r)dr\Bigr\|_{L_m(\Omega)}\nn\\
&\le  C \delta^{1/\qh}(1+\|\,\|z\|_{1-\var;[0,1]}\|_{L_m(\Omega)}),
\end{align}
where $C=C(H,d,\alpha,\mu,m,q)>0$

\textbf{Step 4}. Let $f$ be  a bounded measurable function $\R^d\to\R$. Denote $M:=\sup_{x\in\R^d}|f(x)|$. 
By Lusin’s theorem, there exists a bounded  continuous function $f_\delta\colon\R^d\to\R_+$ such that $\Leb(\{f_\delta\neq f\})\le \delta$ and 
$\|f_\delta\|_{L_\infty(\R^d)}\le  M$. Let $U_\delta$ be an open set of measure $2\delta$ containing the set $\{ f_\delta\neq f\}$; such set exists since the Lebesgue measure is regular. Clearly, for any $x\in\R^d$
\begin{equation*}%\label{ourdiff}
	|f_\delta(x)-f(x)|\le 2M\I_{U_\delta}(x).
\end{equation*}
Using again \eqref{embedqq}, we derive
\begin{equation*}
	\|f_\delta\|_{\B^{\alpha}_q}\le  \|f\|_{\B^{\alpha}_q}+\|f_\delta - f\|_{\B^{\alpha}_q}\le  \|f\|_{\B^{\alpha}_q}+\|f_\delta - f\|_{L_{\qh}(\R^d)}\\
	\le  \|f\|_{\B^{\alpha}_q}+2M \delta^{1/\qh}.
\end{equation*}
Then, applying again Step~2 of the proof to the bounded continuous function $f_\delta$ and recalling \eqref{epsres}, we  deduce
\begin{align*}
	&\Bigl\|\int_{s\wedge\tau}^{t\wedge\tau} f(W_r^H+z_r)\,dr\Bigr\|_{L_m(\Omega)}\\
	&\quad\le 
	\Bigl\|\int_{s\wedge\tau}^{t\wedge\tau} f_\delta(W_r^H+z_r)\,dr\Bigr\|_{L_m(\Omega)}+2M
	\Bigl\|\int_{s\wedge\tau}^{t\wedge\tau} \I_{U_\delta}(W_r^H+z_r)\,dr\Bigr\|_{L_m(\Omega)}\\
	&\quad\le C\|f\|_{\B^{\alpha}_q}(t-s)^{1+H(\alpha-\frac{d}q-1)}\|\,\|z\|_{1-\var;[s,t]}\|_{L_m(\Omega)}+C\|f\|_{\B^{\alpha}_q}(t-s)^{1+ H(\alpha-\frac{d}{q})}\\
	&\qquad +C M \delta^{1/\qh}(1+\|\,\|z\|_{1-\var;[s,t]}\|_{L_m(\Omega)}),
	\end{align*}
where $C=C(H,d,\alpha,\mu,m,q)>0$.
By passing to the limit in the above bound as $\delta\to0$  (note that $\qh<\infty$), we get 
\eqref{finbound}.

\textbf{Step 5}. Suppose that  $f$ is a nonnegative measurable function $\R^d\to\R$. It is clear that a sequence of bounded functions $f\wedge N$, $N\in\N$, converges to $f$ pointwise. Note also that $\|f\wedge N\|_{\B^{\alpha}_{q}}\le \|f\|_{\B^{\alpha}_{q}}$ for any $N\in\N$. Therefore,  
applying  the results of Step~4 to a bounded function $f\wedge N$ and then Fatou's lemma, we derive \eqref{finbound}.
\end{proof}

The following important corollary says that \cref{L:driftb2} is applicable in the full range of parameters $H,d,p$ satisfying \eqref{maincond}.

\begin{corollary}\label{c:driftb2}
Let $f\colon\R^d\to\R$ be a measurable function which is additionally bounded or non-negative. 
Suppose that $f\in \B^{-\delta}_{p}$ and that the parameters $p\in[1,\infty]$, $\delta>0$ satisfy
\begin{equation}\label{maincondupd}
	H+\frac{Hd}{p}+2\delta H<1.
\end{equation}
Then the assumptions \eqref{maincond1}-\eqref{maincond2} are satisfied with $\alpha=-\delta+\I_{p\in[1,2]}(\frac{d}2-\frac{d}p)$, $q=p\vee2$.
\end{corollary}

\begin{proof}
We have $\alpha-\frac{d}q= -\delta-\frac{d}p$. Therefore, we immediately see that \eqref{maincond1} is satisfied thanks to \eqref{maincondupd}.

If $p\in(2,\infty]$, then using again \eqref{maincondupd} we get
\begin{equation*}
\alpha=-\delta>\frac12-\frac1{2H}\ge \bigl(\frac{Hd}q\wedge\frac12\bigr)-\frac1{2H},
\end{equation*}
If $p\in[1,2]$, then using the inequality $\frac1p\le \frac1{2p}+\frac{1}{2}$, we get
\begin{equation*}
\alpha=-\delta+\frac{d}2-\frac{d}p\ge -\delta-\frac{d}{2p}>\frac12-\frac1{2H}\ge \bigl(\frac{Hd}q\wedge\frac12\bigr)-\frac1{2H}.
\end{equation*}
Thus, in both cases \eqref{maincond2} holds. 
\end{proof}

\begin{remark}\label{R:highmom}
Now we see why it was important to use in the proof of \cref{L:driftb2} Rosenthal-type stochastic sewing (\cref{T:RoSSL}) rather than the usual stochastic sewing \cite{LeSSL}. Indeed, we see from \eqref{onecondcheck} that for $m\ge2$, $0\le s\le u \le t$,
$$
\|\delta A_{s,u,t}\|_{L_m(\Omega)}\le C (t-s)^{1+H(\alpha-\mu-\frac{d}q)} \bigl\|\|z\|_{1-\var;[s,t]}^{\mu}\bigr\|_{L_m}.
$$
Recall that $\bigl(\E \|z\|_{1-var;[s,t]}^{1+\rho}\bigr)_{(s,t)\in\Delta_{[0,T]}}$ is a control for any $\rho\ge0$. Thus, clearly, if $m\mu\le1$, then
$$
\|\|z\|_{1-\var;[s,t]}^{\mu}\bigr\|_{L_m(\Omega)}\le \bigl(\E \|z\|_{1-var;[s,t]}\bigr)^\mu.
$$
If, alternatively, $m\mu>1$, then we note that 
$$
\|\|z\|_{1-\var;[s,t]}^{\mu}\bigr\|_{L_m(\Omega)}=\bigl(\E \|z\|^{m\mu }_{1-var;[s,t]}\bigr)^{\frac1m}.
$$
Thus, we obtain that $\bigl\|\|z\|_{1-\var;[s,t]}^{\mu}\bigr\|_{L_m(\Omega)}$ is bounded by a deterministic control raised to the power $\mu\wedge\frac 1m$. Recalling \eqref{con:s1} and \eqref{onecondcheck}, we end up with the condition
\begin{equation}\label{worsemu}
(\alpha-\mu)H>-\frac1{2}\quad\text{and}\quad 
(\alpha-\frac{d}q-\mu) H>\max(-\frac12-(\mu\wedge\frac 1m),-1)
\end{equation}
which is much worse than \eqref{equationmu}. This leads to a non-optimal condition for weak existence: $b\in L_p(\R^d)$, where $\frac{Hd}p<\frac12\wedge(1-H)$ instead of $\frac{Hd}p<1-H$. Indeed,  the optimal choice for the parameter $\mu$ in this worse condition \eqref{worsemu} is $\mu=0$, which combined with \eqref{maincond1} gives for $\alpha = 0$ the condition $\frac{Hd}{p} < \frac12 \wedge (1 - H)$.

We introduced a stopping time $\tau$ in  \cref{L:driftb2} because in applications we assume only that $\|z\|_{1-\var;[0,1]}<\infty$ a.s. rather than $\|\,\|z\|_{1-\var;[0,1]}\|_{L_m(\Omega)}<\infty$. Therefore, we would need to stop the process $z$ once its variation becomes sufficiently large, see, e.g., proof of \cref{L:51}.	
\end{remark}

\section{Proofs of the main results}
\label{sec:proof}

Now, we proceed with the proof of our main results from \cref{S:MR}. Without loss of generality, we assume the time interval is $[0,1]$. We begin by proving the existence of local time for fractional Brownian motion and related processes (\cref{T:loc}) by applying the key integral bound \eqref{finbound} with $f = \delta_x$, $x \in \R^d$. This is done in \cref{S:51}. In \cref{S:52}, we establish the regularity of the solutions to \Gref{x;b}, using again the bound \eqref{finbound} with $f = |b|$. These two results allow us to show the equivalence of different notions of solutions to \eqref{mainSDE}. Once we have a good bound on the regularity of the solutions to \Gref{x;b}, we obtain weak existence by a standard tightness argument, see, e.g., the proofs of \cite[Proposition~2.8]{ABM2020}, \cite[Proposition 3.3 and Corollary 3.5]{ABLM}, \cite[Theorem~8.2]{GG22}, \cite[Theorem~2.8]{ART21}, and others; this is the subject of  \cref{s:WU}. Strong uniqueness for $d = 1$ follows from the bound \eqref{finbound} with $f = \nabla b$ and the uniqueness of solutions to the Young differential equation (\cref{T:YoungODE}), see \cref{s:SU}. Finally, strong existence follows from the Gy\"ongy–Krylov lemma \cite[Lemma~1.1]{MR1392450}, which provides a flexible (non black-box) alternative to the classical Yamada–Watanabe result, see \cref{s:SE}. We conclude with \cref{S:NE}, which provides  a counterexample showing that condition~\eqref{maincond} of \cref{T:func} is optimal, thereby proving \cref{T:opt}.

\subsection{Existence of local times and their regularity}\label{S:51}
Now we are ready to show that the processes $W^H$ and $W^H+\psi$, where $\psi$ is of finite $1$-variation, have local times. Heuristically, the proof is simply an application of \cref{c:driftb2} for the function $f := \delta_x$, $x \in \R^d$. Indeed, formally, the local time of the process $W^H + \psi$ at $x \in \R^d$ over the interval $[0,t]$, $t \in [0,1]$, is $\int_0^t \delta_x(W_r^H + \psi_r)\,dr$. \cref{c:driftb2} guarantees the existence this integral under the condition \eqref{maincondupd}. Clearly, $\delta_x \in \B^0_1$, and thus \eqref{maincondupd} reduces to $H + Hd < 1$, which is the desired condition.

Note, however, that \cref{c:driftb2} applies  only to measurable functions, not distributions. Therefore, to make the above argument rigorous, one must appropriately approximate the delta function and justify passing to the limit. Another technical challenge is that, to make the moments of the quadratic variation of $\psi$ finite, one needs to introduce stopping times and stop $\psi$ once its quadratic variation becomes large enough. This program is implemented in the proof below.

\begin{proof}[Proof of \cref{T:loc}]
(i). Fix $d\in\N$, $H\in(0,1)$ such that $Hd<1$, $\gamma\in (0,(\frac1{2H}-\frac{d}2)\wedge1)$. Recall the definition of $\nu_d$ in \eqref{nud}. For  $x\in\R^d$, $R>0$ define
\begin{equation}\label{lrx}
l^{R,x}(z):=\frac1{v_dR^{d}}\I(z\in\Ba(x,R)),\quad z\in\R^d.
\end{equation}
We see that by \cref{p:deltacon} $l^{R,x}$ converges to $\delta_x$ in $\B^{-\eps}_1$ as $R\to0$, for any $\eps>0$. Denote
$$
L^{R}(t,x):=\int_{0}^{t} l^{R,x}(W_r^H)\,dr,\quad t\in[0,1],\,x\in\R^d.
$$
Take $\eps>0$ small enough so that $H(d+2\eps)<1$. 
For $x, y \in \R^d$ and $n, k \in \N$, we apply \cref{C:ourbounds} three times with $q = 2$, $m \ge 2$, $S_0 = 0$, $x(\cdot) \equiv 0$, and the following remaining parameters. First, we take $f=l^{\frac1n,x}$, $\alpha=-\frac{d}2$. Then we take $f=l^{\frac1n,x}-l^{\frac1n,y}$, $\alpha=-\frac{d}2-\gamma$. Finally, we take $f=l^{\frac1n,x}-l^{\frac1k,x}$, $\alpha=-\frac{d}2-\eps$. In all three cases, we see that \eqref{extracondTS} holds. Recall that, by definition, $V_{0,r} = W^H_r$, $r\in[0,1]$. Therefore, using the embedding $L_1(\R^d) \subset \B^{0}_1 \subset \B^{-d/2}_2$, we obtain for any $s,t\in[0,1]$, $x,y\in\R^d$
\begin{align}
&\| L^{\frac1n}(t,x)-L^{\frac1n}(s,x)\|_{L_m(\Omega)}\le C \|l^{\frac1n,x}\|_{\B^{-d/2}_2}|t-s|^{1-Hd}\le C|t-s|^{1-Hd};\label{step1lt}\\
&\| L^{\frac1n}(t,x)-L^{\frac1n}(s,x)-(L^{\frac1n}(t,y)-L^{\frac1n}(s,y))\|_{L_m(\Omega)}\le 
C |x-y|^\gamma |t-s|^{1-Hd-H\gamma};\label{step2lt}\\
&\| L^{\frac1n}(t,x)-L^{\frac1k}(t,x)\|_{L_m(\Omega)}\le C \|l^{\frac1n,x}-l^{\frac1k,x}\|_{\B^{-\eps}_1}\le C(n\wedge k)^{-\eps} \label{step3lt};
\end{align}	
for $C=C(H,d,\eps,\gamma,m)>0$. Here  in \eqref{step2lt}, we used \eqref{fxydif}, which implies that
$$\|l^{\frac1n,x}-l^{\frac1n,y}\|_{\B^{-d/2-\gamma}_2}\le C |x-y|^\gamma \|l^{\frac1n,x}\|_{\B^{-d/2}_2}\le C |x-y|^\gamma\|l^{\frac1n,x}\|_{L_1(\R^d)}\le C |x-y|^\gamma
$$
for $C=C(d,\gamma)>0$; in \eqref{step3lt} we applied \eqref{apa4mr}. We stress that the constant $C$ does not depend on $n,k$. Collecting \eqref{step1lt}, \eqref{step2lt}, \eqref{step3lt}, and noting that the process $L^R$ is continuous in $t$ for any fixed $R>0$, we see that all the conditions of the Kolmogorov continuity theorem in the form \cref{p:KCT} with  $X^n:=L^{\frac1n}$ are satisfied. Therefore, there exists a jointly continuous  process $L\colon
\Omega\times[0,1]\times \R^d\to\R$ and a set of full measure $\Omega'\subset\Omega$ such that on $\Omega'$ for any $M>0$ 
\begin{align}
&\sup_{\substack{x\in\R^d\\|x|\le M}}\sup_{s,t\in[0,1]}\frac{|L(t,x)-L(s,x)|}{|t-s|^{1-Hd-\eps}}<\infty;\label{limpartL2}\\
&
\sup_{\substack{x,y\in\R^d\\|x|,|y|\le M}}\sup_{s,t\in[0,1]}\frac{|L(t,x)-L(s,x)-(L(t,y)-L(s,y))|}{|t-s|^{1-Hd-H\gamma-\eps}|x-y|^{\gamma-\eps}}<\infty.\label{limpartL3}
\end{align}
Further, \cref{p:KCT}(iii) implies that for any $\omega\in\Omega'$ there exists a set $A(\omega)\subset \R^d$ of zero Lebesgue measure such that 
for any $M>0$, $t\in[0,1]$
\begin{equation}\label{limpartL}
\sup_{\substack{x\in\R^d\setminus A(\omega)\\|x|\le M}}|L(\omega, t,x)- L^{\frac1{n}}(\omega,t,x)|\to0 \quad\text{ as $n\to\infty$}.
\end{equation}
Since for a fixed $\omega\in \Omega'$ the trajectory $W^H(\omega)$ is bounded on $[0,1]$, we see that $L^R(t,x)=0$ for any $t\in[0,1]$, $|x|>M(\omega)$, $R<1$ for some $M(\omega)$. Therefore, $L(t,x)=0$ for  any $t\in[0,1]$, $|x|>M(\omega)$. Hence, restriction $|x|\le M$ in \eqref{limpartL2}, \eqref{limpartL3}, \eqref{limpartL} can be removed. Therefore, $L$ satisfies \eqref{loc1}. 

It remains to show that $L$ is a local time of $W^H$. Define $\mu_t(B)=\int_0^t \I(W_r^H\in B)\,dr$, $B\in\mathcal{B}(\R^d)$, $t\in[0,1]$. Then for any $R>0$, $t\in[0,1]$, $x\in\R^d$
\begin{equation*}
L^{R}(t,x)=\frac1{v_dR^{d}}\mu_t (\Ba(x,R)),\quad 
\end{equation*}
Substituting this into \eqref{limpartL}, we see that all the conditions of \cref{p:localtime} are satisfied. Thus, on $\Omega'$ for any $t\in[0,1]$  we have $d \mu_t/d \Leb=L(t)$, which shows that $L$ is a local time of $W^H$.

(ii). The proof is similar to the proof of part (i), however we have to modify certain steps since we have to stop the drift once its variation becomes is too large. Fix $d\in\N$, $H\in(0,1)$ such that $H(d+1)<1$, $\gamma\in (0,(\frac1{2H}-\frac{d}{2}-\frac12)\wedge1)$, and a random process $\psi$ satisfying the assumptions of the theorem. For $N\in\Z_+$ introduce a stopping time
$$
\tau_N:=\inf\{t\in[0,1]: \|\psi\|_{1-\var;[0,t]}\ge N\}\wedge1.
$$
Recall that $\psi$ is continuous, hence $\|\psi\|_{1-\var;[0,\cdot]}$ is continuous,
and $\|\psi\|_{1-\var;[0,\tau_N]}\le N$. Let $\psi^N_t:=\psi_{t\wedge \tau_N}$, $t\in[0,1]$. Consider a set where we do not stop the process $\psi$, that is 
\begin{equation}\label{omegan}
	\Omega_N:=\{\tau_N=1\}.
\end{equation}	
For a function $l^{R,x}$ defined in \eqref{lrx}, $N\in\Z_+$ put 
$$
L^{R,N}(t,x):=\int_{0}^{t } l^{R,x}(W_r^H+\psi^N_r)\,dr,\quad t\in[0,1],\,x\in\R^d.
$$

Take $\eps>0$ small enough so that $H(d+1+2\eps)<1$. For $x,y\in\R^d$, $n,k\in\N$ we apply  \cref{L:driftb2} three times with $z=\psi^N$, $q=2$, $m\ge2$, $\tau=1$ and the following remaining parameters. First, we take $f=l^{\frac1n,x}$, $\alpha=-\frac{d}2$. Then, we take $f=l^{\frac1n,x}-l^{\frac1n,y}$, $\alpha=-\frac{d}2-\gamma$. Finally, we take $f=l^{\frac1n,x}-l^{\frac1k,x}$, $\alpha=-\frac{d}2-\eps$. It is easy to check that in all three cases conditions \eqref{maincond1}-\eqref{maincond2} are satisfied. Thus, using \eqref{fxydif} and \eqref{apa4mr},  we deduce for any $s,t\in[0,1]$, $x,y\in\R^d$
\begin{align*}%\label{step1psi}
&\| L^{\frac1n,N}(t,x)-L^{\frac1n,N}(s,x)\|_{L_m(\Omega)}\le C (N+1) |t-s|^{1-H(d+1)};\\
&\| L^{\frac1n,N}(t,x)-L^{\frac1n,N}(s,x)-(L^{\frac1n,N}(t,y)-L^{\frac1n,N}(s,y))\|_{L_m(\Omega)};\nn\\
&\qquad\le C(N+1) |x-y|^\gamma |t-s|^{1-H(d+1+\gamma)}\\%\label{step2psi}\\
&\| L^{\frac1n}(t,x)-L^{\frac1k}(t,x)\|_{L_m(\Omega)}\le C (N+1) (n\wedge k)^{-\eps};%\label{step3psi}
\end{align*}
for $C=C(H,d,\eps,\gamma,m)>0$. Applying a version of the Kolmogorov continuity theorem (\cref{p:KCT}), we conclude that there exists a process $L^N$ and a set $\wt \Omega_N\subset\Omega$ of full probability measure such that on $\wt \Omega_N$ \eqref{limpartL2}--\eqref{limpartL3} hold with $L^N$ in place of $L$ and $H(d+1)$ in place of $Hd$. Furthermore, for any $\omega\in\wt \Omega_N$ there exists
a  set $A_N(\omega)\subset\R^d$ of Lebesgue measure zero such that for any $M>0$
\begin{equation}\label{limpartLpsi}
\sup_{\substack{x\in\R^d\setminus A_N(\omega)\\|x|\le M}}|L^N(\omega,t,x)- L^{\frac1{n},N}(\omega,t,x)|\to0 \quad\text{ as $n\to\infty$}.
\end{equation}
 Now on the set of full probability measure
$
\Omega^*:=\bigcup_{N=1}^\infty (\Omega_N\setminus \Omega_{N-1})\cap \wt\Omega_N
$
we are ready to define a candidate for the local time as follows:
$$
\mathbf{L}(\omega,t,x):=\sum_{N=1}^\infty  L^{N}(t,x) (\omega)\I_{(\Omega_N\setminus \Omega_{N-1})\cap\wt \Omega_N}(\omega),\quad t\in[0,1], x\in\R^d.
$$
By above, on $\Omega^*$, the process $\mathbf{L}$ is jointly continuous in $(t,x)$ and satisfies \eqref{loc2}. Here, as in part (i), we also use the fact that for any fixed  $\omega\in \Omega^*$ the trajectory $W^H(\omega)+\psi(\omega)$ is bounded. 

Denote by $\mu_t$ the occupation measure of $W^H+\psi$ over the time interval $[0,t]$. Note that by the definition of  the set $\Omega_N$ in \eqref{omegan}, for any $N\in\Z_+$, $R>0$, $t\in[0,1]$ $x\in\R^d$, $\omega\in\Omega$ we have
$$
L^{R,N}(\omega,t,x) \I_{\Omega_N}(\omega)=\frac1{v_dR^{d}}\mu_t ( \Ba(x,R))
\I_{\Omega_N}(\omega),\quad  
$$
Therefore \eqref{limpartLpsi} and the definition of the process $\mathbf{L}$ imply for any $\omega\in\Omega^*$, $t\in[0,1]$
\begin{equation*}
	\sup_{x\in\R^d\setminus (\cup_{N=1}^\infty A_N(\omega))}\Bigl|\mathbf{L}(\omega,t,x)- \frac{1}{v_dn^{-d}}\mu_t \bigl(\Ba(x,\frac1{n})\bigr)\Bigr|\to0 \quad\text{ as $n\to\infty$}.
\end{equation*}
This, together with \cref{p:localtime}, shows that $\mathbf{L}$ is indeed the local time of $W^H+\psi$.
\end{proof}

\subsection{A priori bound on regularity of solutions to SDE and equivalence of different notions of solutions}\label{S:52}
Next, we establish a very useful bound for the regularity of solutions to \eqref{mainSDE}. This result, combined with \cref{T:loc} proved earlier, is crucial for showing the equivalence of different notions of solutions to \eqref{mainSDE}.
\begin{lemma}\label{L:51}
Let $H\in(0,1)$, $d\in\N$, $p\in[1,\infty]$, $m\in[2,\infty)$, $M>0$. Assume that \eqref{maincond} holds. Let $f\colon\R^d\to\R^d$ be a measurable function such that 
\begin{equation}\label{Mcond}
 \|f\|_{L_p(\R^d)} \le M.
\end{equation}
Let $x\in\R^d$ and assume that $(X,W^H)$ is a weak solution to \Gref{x;f}.  
Then there exists a constant $C=C(H,d, m,M,p)>0$ such that for any $(s,t)\in\Delta_{[0,1]}$ we have
\begin{equation}\label{regresult}
\|\,\|X-W^H\|_{1-\var;[s,t]}\,\|_{L_m(\Omega)}=\Bigl\|\int_s^t |f(X_r)|\,dr\Bigr\|_{L_m(\Omega)}\le C (t-s)^{1-\frac{Hd}p}.
\end{equation}
\end{lemma}
\begin{proof}
First, we treat the case $p<\infty$. 
Define
$$
\psi(t):=x+\int_0^t f(X(s))\,ds=X_t-W_t^H,\quad t\in[0,1].
$$
Clearly,  $\psi$ is a process of finite $1$-variation and 
\begin{equation*}%\label{varpsi}
\|\psi\|_{1-\var;[s,t]}=\int_s^t |f(X(r))|\,dr,\quad (s,t)\in\Delta_{[0,1]}.
\end{equation*}
Note that we did not assume that $f$ is bounded; therefore, we do not know a priori whether $\|\psi\|_{1-\var;[s,t]}$ has a finite moment of order $m$. We establish this using the stopping time technique and by bounding the expression in terms of itself via \cref{c:driftb2}.

Similar to the proof  of \cref{T:loc}, we fix arbitrary $N>0$ and introduce a stopping time 
$
\tau_N:=\inf\{t\in[0,1]: \|\psi\|_{1-\var;[0,t]}\ge N\};
$
as usual we set $\tau_N=1$ if $\|\psi\|_{1-\var;[0,1]}< N$. Recall, that $\psi$ is continuous by the definition of a solution to \Gref{x;f}; hence $\|\psi\|_{1-\var;[0,\cdot]}$ is continuous,
 and $\|\psi\|_{1-\var;[0,\tau_N]}\le N$. Consider a stopped process
 $$
 \psi^N_t:=\psi_{t\wedge\tau_N},\quad t\in[0,1].
 $$
By construction for $(s,t)\in\Delta_{[0,1]}$
\begin{equation}\label{psivar}
\|\psi^N\|_{1-\var;[s,t]}=\int_{s\wedge\tau_N}^{t\wedge\tau_N}|f(X_r)|\, dr\le N.
\end{equation}

Fix small $\ell\in[0,1]$. We apply \cref{L:driftb2} with $z:=\psi^N$,  $|f|$ in place of $f$, $\tau=\tau_N$. If $p\in[1,2)$, we take $\alpha:=-\frac{d}p+\frac{d}2<0$, $q=2$. Since $\alpha>-\frac{d}{2p}>\frac12-\frac1{2H}$, we see that conditions \eqref{maincond1}-\eqref{maincond2} are satisfied for this choice of $\alpha$ and $q$. If $p\in[2,\infty)$, then there exists $q>p$ such that $-\frac{d}p+\frac{d}q>\frac12-\frac1{2H}$. We take $\alpha=-\frac{d}p+\frac{d}q<0$. We see that conditions \eqref{maincond1}-\eqref{maincond2} holds as well. In both cases we have $L_p(\R^d)\subset \B^{\alpha}_q$ and hence $\|\,|f|\,\|_{\B^\alpha_q}\le \|f\|_{L_p(\R^d)}\le M$. Thus, all the conditions of \cref{L:driftb2} are satisfied and we get from \eqref{finbound}
for any $s,t\in\Delta_{[0,1]}$ with $|t-s|\le \ell$, $m\ge2$
\begin{align}\label{buckling}
\Bigl\|\int_{s\wedge\tau_N}^{t\wedge\tau_N} |f(X_r)|\,dr\Bigr\|_{L_m(\Omega)}&=
\Bigl\|\int_{s\wedge\tau_N}^{t\wedge\tau_N} |f(W_r+\psi^N_r)|\,dr\Bigr\|_{L_m(\Omega)}\nn\\
&\le C_0 M \ell^{1-\frac{Hd}p-H}
\Bigl\|\int_{s\wedge\tau_N}^{t\wedge\tau_N}|f(X_r)|\,dr\Bigr\|_{L_m(\Omega)}+C_0 M (t-s)^{1-\frac{Hd}p},
\end{align}
where $C_0=C_0(H,d,m,p)>0$ and we used the total variation identity \eqref{psivar} and assumption \eqref{Mcond}. Note that $\|\int_{s\wedge\tau_N}^{t\wedge\tau_N} |f(X_r)|\,dr\|\le N<\infty$, see \eqref{psivar}. Further, we see that by \eqref{maincond}, $1-\frac{Hd}p-H>0$. Thus, by taking in \eqref{buckling} $\ell=\ell(H,d,m, M,p)$ small enough so that 
$$
C_0 M \ell^{1-\frac{Hd}p-H}<\frac12,
$$
we get
\begin{equation*}
\Bigl\|\int_{s\wedge\tau_N}^{t\wedge\tau_N} |f(X_r)|\,dr\Bigr\|_{L_m(\Omega)}\le 2C_0 M (t-s)^{1-\frac{Hd}p},\quad s,t\in[0,1],\,|t-s|\le \ell.
\end{equation*}
Since $\tau_N\to 1$ a.s. as $N\to\infty$, we have $\int_{s\wedge\tau_N}^{t\wedge\tau_N} |f(X_r)|\,dr\to\int_{s}^{t} |f(X_r)|\,dr$ a.s. as $N\to\infty$. Therefore, Fatou's lemma implies
\begin{equation*}
\Bigl\|\int_{s}^{t} |f(X_r)|\,dr\Bigr\|_{L_m(\Omega)}\le 2C_0 M (t-s)^{1-\frac{Hd}p},\quad s,t\in[0,1],\,|t-s|\le \ell.
\end{equation*}
Applying this bound $\lceil 1/\ell\rceil$ times, we get \eqref{regresult}.

Now let us treat the case $p=\infty$. Denote $f_M:=|f|\I_{|f|>M}$. Choose $\eps>0$ such that $\eps<\frac1{2H}-\frac12$. Then using the stopping times  $\tau_N$ defined as above, we derive for any $(s,t)\in\Delta_{[0,1]}$
\begin{align*}
\Bigl\|\int_s^t |f(X_r)|\,dr\Bigr\|_{L_m(\Omega)}&\le\Bigl\|\int_s^t |f(X_r)|\I_{|f(X_r)|\le M}\,dr\Bigr\|_{L_m(\Omega)}+\Bigl\|\int_0^1 f_M(X_r)\,dr\Bigr\|_{L_m(\Omega)}\\
&\le  M (t-s)+\lim_{N\to\infty} \Bigl\|\int_{0}^{\tau_N} f_M(X_r)\,dr\Bigr\|_{L_m(\Omega)}\\
&\le  M (t-s)+\lim_{N\to\infty} C N \|f_M\|_{\B^{-\eps}_\infty}=M(t-s),
\end{align*}
where in the last line we applied \cref{L:driftb2} with $z:=\psi^N$, $f=f_M$, $\tau=\tau_N$, $\alpha=\eps$, $q=\infty$ and used that $\|f_M\|_{\B^{-\eps}_\infty}\le \|f_M\|_{L_\infty(\R^d)}=0$.  
\end{proof}

Next, we establish a similar result for solutions to equation \eqref{measureeq}. For a measure $\mu \in \mathcal{M}(\mathbb{R}^d, \mathbb{R}^d)$, we denote its total variation measure by $|\mu| := \mu_+ + \mu_-$, where the measures $\mu_+$ and $\mu_-$ are from the Jordan decomposition of the measure $\mu$: $\mu = \mu_+ - \mu_-$.

\begin{lemma}\label{L:51m}
	Let $H\in(0,1)$, $d\in\N$, $m\in[2,\infty)$, $x\in\R^d$, $b\in\M(\R^d,\R^d)$. Assume that $H(d+1)<1$. Let 
	 $(X,W^H)$ is a weak solution to \eqref{measureeq}.   
	Then there exists a constant $C=C(H,d, m,|b|(\R^d))>0$ such that for any $(s,t)\in\Delta_{[0,1]}$ we have
	\begin{equation}\label{regresultm}
		\|\,\|X-W^H\|_{1-\var;[s,t]}\,\|_{L_m(\Omega)}\le C (t-s)^{1-Hd}.
	\end{equation}
\end{lemma}
\begin{proof}
We use a similar argument as in the proof of \cref{L:51} with appropriate modifications. The main difference is that, instead of the drift term $\int_0^t f(W^H_r+\psi_r)\,dr$, for which we had already established a good bound in \cref{L:driftb2}, we now have to work with the drift $\int_{\R^d} L_t^{W^H+\psi}(y)b(dy)$, where $\psi$ is a “nice” perturbation. The main idea is to approximate this drift by the drifts of the form $\int_0^t f^n(W^H_r+\psi_r)\,dr$ for an appropriate sequence $(f^n)$ and then use the bound \eqref{finbound}.

Define
$$
\psi(t):=x+\int_{\R^d} L_t^X(y) b(dy)=X_t-W_t^H,\quad t\in[0,1].
$$
It is easy to see that $\psi$ is a process of finite $1$-variation and 
	\begin{equation*}%\label{varpsi}
		\|\psi\|_{1-\var;[s,t]}=\int_{\R^d} (L_t^X(y)-L_s^X(y)) |b|(dy),\quad (s,t)\in\Delta_{[0,1]}.
	\end{equation*}
By  \cref{T:loc}(ii), there exists a set of full measure $\Omega'\subset\Omega$ and $\gamma>0$ such that for any $t\in[0,1]$ the process $L_t^X$ belongs to $\C^\gamma(\R^d,\R)$ on $\Omega'$. Define $f^n:=P_{\frac1n}|b|$ and note that $f^n\to |b|$ in $\B^{-\frac\gamma2}_1$ as $n\to\infty$ by \cref{p:veryboring}. Therefore, by \cref{p:ltc} on $\Omega'$ for any $(s,t)\in\Delta_{[0,1]}$
\begin{equation}\label{limboundfnst}
\int_s^t f^n(X_r)\,dr=\int_{\R^d} (L_t^X(y)-L_s^X(y)) f^n(y)\,dy\to\int_{\R^d} (L_t^X(y)-L_s^X(y)) |b|(dy).
\end{equation}
as $n\to\infty$. 

As in the proof of \cref{L:51}, we fix arbitrary $N>0$ and introduce a stopping time and a stopped variation process
$$
\tau_N:=\inf\{t\in[0,1]: \|\psi\|_{1-\var;[0,t]}\ge N\};\qquad \psi^N_t:=\psi_{t\wedge\tau_N},\quad t\in[0,1].
$$
We have 
\begin{equation}\label{varbound}
	\|\psi^N\|_{1-\var;[s,t]}=\int_{\R^d} (L_{t\wedge\tau_N}^X(y)-L_{s\wedge\tau_N}^X(y)) |b|(dy),\quad (s,t)\in\Delta_{[0,1]}.
\end{equation}
Fix small $\ell\in[0,1]$. We apply \cref{L:driftb2} with $z:=\psi^N$,  $f^n$ in place of $f$, $\tau=\tau_N$, $\alpha:=-\frac{d}2$, $q=2$. We see  all the conditions of \cref{L:driftb2} are satisfied and we get from \eqref{finbound}
for any $s,t\in\Delta_{[0,1]}$ with $|t-s|\le \ell$, $m\ge2$
\begin{align*}
	\Bigl\|\int_{s\wedge\tau_N}^{t\wedge\tau_N} f^n(X_r)\,dr\Bigr\|_{L_m(\Omega)}&=
	\Bigl\|\int_{s\wedge\tau_N}^{t\wedge\tau_N} f^n(W_r+\psi^N_r)\,dr\Bigr\|_{L_m(\Omega)}\nn\\
	&\le C_0 M \ell^{1-Hd-H}
	\Bigl\|\int_{\R^d} (L_{t\wedge\tau_N}^X(y)-L_{s\wedge\tau_N}^X(y)) |b|(dy)\Bigr\|_{L_m(\Omega)}\nn\\
	&\phantom{\le}+C_0 M (t-s)^{1-{Hd}},
\end{align*}
where we denoted $M:=|b|(\R^d)$, $C_0=C_0(H,d,m)>0$ and we used the embedding $L_1(\R^d)\subset \B^{-\frac{d}2}_q$ and the  identity for total variation \eqref{varbound}. We note that $C_0$ does not depend on $n$. Thus, by passing to the limit as $n\to\infty$ and using \eqref{limboundfnst} and Fatou's lemma we get for  $s,t\in\Delta_{[0,1]}$ with $|t-s|\le \ell$, $m\ge2$
\begin{align*}
&\Bigl\|\int_{\R^d} (L_{t\wedge\tau_N}^X(y)-L_{s\wedge\tau_N}^X(y)) |b|(dy)\Bigr\|_{L_m(\Omega)}\\
&\qquad \le C_0 M \ell^{1-Hd-H}
\Bigl\|\int_{\R^d} (L_{t\wedge\tau_N}^X(y)-L_{s\wedge\tau_N}^X(y)) |b|(dy)\Bigr\|_{L_m(\Omega)}+C_0 M (t-s)^{1-{Hd}}.
\end{align*}
By choosing now $\ell= \ell(H,d,m,M)$ small enough, we get from the above inequality 
	\begin{equation*}
\Bigl\|\int_{\R^d} (L_{t\wedge\tau_N}^X(y)-L_{s\wedge\tau_N}^X(y)) |b|(dy)\Bigr\|_{L_m(\Omega)}		\le C (t-s)^{1-Hd},\quad s,t\in[0,1],\,|t-s|\le \ell,
	\end{equation*}
where $C=C(H,d,m,M)>0$. By passing to the limit in this inequality as $N\to\infty$ and applying the resulting bound $\lceil 1/\ell\rceil$ times, we get \eqref{regresultm}.
\end{proof}

Now using \cref{L:51} we conclude that for $b \in L_p(\R^d)$ the standard notion of a solution to \eqref{mainSDE} coincides with the notion of a regularized solution introduced in \cref{D:sol}, provided that the solution has finite variation and \eqref{maincond} holds.

\begin{proof}[Proof of \cref{T:func}(ii)--(iv)]
(ii). This follows immediately from  \cref{L:51}.

(iii). Let $(X,W^H)$ be a solution to  \Gref{x;b}. Let $(b^n)_{n\in\Z_+}$ be a sequence of $\C^\infty_b(\R^d,\R^d)$ functions converging to $b$ in $\B^{0-}_p$. Define
\begin{equation*}
\psi_t:=\int_0^t b(X_r)\,dr,\quad \psi^n_t:=\int_0^t b^n(X_r)\,dr,\qquad t\in[0,1],\,n\in\Z_+.
\end{equation*}
We see that $X_t=x+\psi_t+W_t^H$, $t\in[0,1]$, and thus part~(1) of \cref{D:sol} holds. 

To verify part~(2), we fix $N>0$ and take $\delta>0$ small enough such that \eqref{maincondupd} holds. We apply \cref{L:driftb2} twice: first to the function $f := b^n - b\I_{|b| \le N}$, and then to the function $f := |b|\I_{|b| \ge N}$. The remaining parameters are the same in both cases and are given by $z =x+ \psi$, $\tau = 1$, $q = p \vee 2$, and $\alpha = -\delta + \I_{p \in [1,2]}\big(\frac{d}{2} - \frac{d}{p}\big)$. By \cref{c:driftb2}, conditions \eqref{maincond1}--\eqref{maincond2} are satisfied. Additionally, the function $b^n - b\I_{|b| \le N}$ is bounded, and the function $|b|\I_{|b| \le N}$ is nonnegative. Thus, in both cases, all the conditions of \cref{L:driftb2} are satisfied. Recalling the bound on $1$-variation of $\psi$ obtained in \eqref{regresult}, we derive from \eqref{finbound}  for any $(s,t)\in\Delta_{[0,1]}$, $m\ge2$
\begin{align*}
\|(\psi_t^n-\psi_t)-(\psi_s^n-\psi_s)\|_{L_m(\Omega)}&=
\Bigl\|\int_{s}^{t} (b^n-b)(W_r^H+x+\psi_r)\,dr\Bigr\|_{L_m(\Omega)}\\
&\le \Bigl\|\int_{s}^{t} (b^n-b\I_{|b|\le N})(W_r^H+x+\psi_r)\,dr\Bigr\|_{L_m(\Omega)}\\
&\phantom{\le}+\Bigl\|\int_{s}^{t} (|b|\I_{|b|> N})(W_r^H+x+\psi_r)\,dr\Bigr\|_{L_m(\Omega)}\\
&\le C \|b^n-b\I_{|b|\le N}\|_{\B^{-\delta}_p}(t-s)^{1-\frac{Hd}p-\delta H}+
C\|b\I_{|b|> N}\|_{L_p(\R^d)}\\
&\le C \|b^n-b\|_{\B^{-\delta}_p}(t-s)^{1-\frac{Hd}p-\delta H}+
2C\|b\I_{|b|> N}\|_{L_p(\R^d)},
\end{align*}
where $C=C(H,d,\delta,m,p,\|b\|_{L_p(\R^d)})>0$. By passing to the limit as $N\to\infty$ we get 
\begin{equation*}
	\|(\psi_t^n-\psi_t)-(\psi_s^n-\psi_s)\|_{L_m(\Omega)}\le 
	 C \|b^n-b\|_{\B^{-\delta}_p}(t-s)^{1-\frac{Hd}p-\delta H}.
\end{equation*}
By taking $m$ sufficiently large such that $(1 - \frac{Hd}{p} - \delta H)m > 1$, we conclude, by the Kolmogorov continuity theorem, that
\begin{equation*}
\|\sup_{t\in[0,1]}|\psi_t^n-\psi_t|\|_{L_m(\Omega)}\le  C \|b-b^n\|_{\B^{-\delta}_p},
\end{equation*}
where the constant $C=C(H,d,\delta,m,p,\|b\|_{L_p(\R^d)})>0$ is independent of $n$. Recalling that by assumption $\|b-b^n\|_{\B^{-\delta}_p}\to0$ as $n\to\infty$, we deduce that $\sup_{t\in[0,1]}|\psi_t^n-\psi_t|\to0$ as $n\to\infty$ in $L_m(\Omega)$ and hence in probability. Therefore, part (2) of  \cref{D:sol} is also satisfied. Thus, $(X,W^H)$ solves \eqref{mainSDE} in the sense of \cref{D:sol}.

(iv). Assume that $(X,W^H)$ solves \eqref{mainSDE} in the sense of \cref{D:sol}. Put  $\psi_t:=X_t-W^H_t-x$, $t\in[0,1]$. By the definition of the class \textbf{BV}, $\|\psi\|_{1-\var;[0,1]}<\infty$. Let $(b^n)_{n\in\Z_+}$ be a sequence of $\C^\infty_b(\R^d,\R^d)$ functions converging to $b$ in $\B^{0-}_p$. Fix $N>0$ and  introduce a stopping time 
$$
\tau_N:=\inf\{s\in[0,1]: \|\psi\|_{1-\var;[0,s]}\ge N\}\wedge1.
$$ 
Define $\psi^N_t:=\psi_{t\wedge\tau_N}$, $t\in[0,1]$, and note that for any $(s,t)\in[0,1]$ we have ${\|\psi^N\|_{1-\var;[s,t]}\le N}$. 
Fix $t \in [0,1]$, $M>0$, and take $\delta>0$ small enough such that \eqref{maincondupd} holds. As before, we apply \cref{L:driftb2} twice: first to the function $f := b^n - b\I_{|b| \le M}$, and then to the function $f := |b|\I_{|b| \ge M}$. The remaining parameters are the same in both cases and are given by $z = x + \psi^N$, $\tau =\tau_N$, $q = p \vee 2$, and $\alpha = -\delta + \I_{p \in [1,2]}\big(\frac{d}{2} - \frac{d}{p}\big)$. By \cref{c:driftb2}, conditions \eqref{maincond1}--\eqref{maincond2} are satisfied. Moreover, the function $b^n - b\I_{|b| \le M}$ is bounded, and the function $|b|\I_{|b| \ge M}$ is nonnegative. Therefore, in both cases, all the conditions of \cref{L:driftb2} are satisfied, and we obtain from \eqref{finbound}
\begin{align*}
\Bigl\|\int_{0}^{t\wedge\tau^N} (b^n(X_r)-b(X_r))\,dr\Bigr\|_{L_2(\Omega)}&=
\Bigl\|\int_{0}^{t\wedge\tau^N} (b^n-b)(W_r^H+x+\psi^N_t)\,dr\Bigr\|_{L_2(\Omega)}\\
&\le \Bigl\|\int_{0}^{t\wedge\tau^N} (b^n-b\I_{|b|\le M})(W_r^H+x+\psi^N_r)\,dr\Bigr\|_{L_2(\Omega)}\\
&\phantom{\le}+\Bigl\|\int_{0}^{t\wedge\tau^N} |b|\I_{|b|> M})(W_r^H+x+\psi^N_r)\,dr\Bigr\|_{L_2(\Omega)}\\
&\le C (1+N) (\|b-b^n\I_{|b|\le M}\|_{\B^{-\delta}_p}+\|b\I_{|b|> M}\|_{L_p(\R^d)})\\
&\le C (1+N) (\|b-b^n\|_{\B^{-\delta}_p}+\|b\I_{|b|> M}\|_{L_p(\R^d)})
\end{align*}
for $C=C(H,d,\delta,p)>0$. 
By passing to the limit as $M\to\infty$ and then as $n\to\infty$ we get that for any fixed $N>0$
\begin{equation*}
\int_{0}^{t\wedge\tau^N} b^n(X_r)\,dr\to
\int_{0}^{t\wedge\tau^N} b(X_r)\,dr,\quad\text{in probability as $n\to\infty$}.
\end{equation*}
On the other hand, since $(X,W^H)$ is a regularized solution to \eqref{mainSDE}, 
\begin{equation*}
	\int_{0}^{t\wedge\tau^N} b^n(X_r)\,dr\to \psi_{t\wedge\tau^N},\quad\text{in probability as $n\to\infty$}.
\end{equation*}
Thus, $\psi_{t\wedge\tau^N}=\int_{0}^{t\wedge\tau^N} b(X_r)\,dr$ a.s. for any $N>0$.   By passing to the limit as $N\to\infty$, we get
\begin{equation*} 
X_{t}-x-W_t^H=\psi_t=\int_0^{t} b(X_r)\,dr\quad \text{a.s.},
\end{equation*}
where we used that $\lim_{N\to\infty}\tau^N=1$ a.s. because  $\|\psi\|_{1-\var;[0,1]}<\infty$.
This implies that $(X,W^H)$ is a solution to \Gref{x;b}.
\end{proof}

Next, we prove a related statement concerning equivalence of different notions of solutions in case $b\in\M(\R^d,\R^d)$. 

\begin{proof}[Proof of \cref{T:measure}]
(i). 
Assume that $(X,W^H)$ is a regularized solution to \eqref{mainSDE} in the sense of \cref{D:sol}. Define $\psi_t:=X_t-W_t^H-x$ and note that $\|\psi\|_{1-\var;[0,1]}<\infty$ by the definition of class \textbf{BV}. Then $X=W^H+\psi+x$, and therefore, \cref{T:loc}(ii) implies that $X$ has local time $L_t^X$, which on a set of full measure $\Omega'\subset\Omega$ belongs to $\C^\gamma(\R^d,\R)$ for any $t\in[0,1]$, and $\gamma=\frac1{4H}-\frac{d}4-\frac14>0$, thanks to the assumption $H(d+1)<1$. By \cref{p:limitlt}, \eqref{lebloc} holds for $L^X$.

Let $(b^n)_{n\in\Z_+}$ be a sequence of $\C^\infty_b(\R^d,\R^d)$ functions converging to $b$ in $\B^{0-}_1$. By part (2) of \cref{D:sol} and by passing to a subsequence if necessary, we have on a set $\Omega''\subset\Omega'$ of full measure for any $t\in[0,1]$
$$
\int_0^t b^n(X_r)dr\to \psi(t)\quad \text{as $n\to\infty$}.
$$
On the other hand,  by the definition of local time and \cref{p:ltc}, for any $t\in[0,1]$ on $\Omega''$
$$
\int_0^t b^n(X_r)dr=\int_{\R^d}b^n(y)L_t^X(y)\,dy\to \int_{\R^d}L_t^X(y)b(dy)\quad \text{as $n\to\infty$}.
$$
Here we used that $L_t^X\in\C^\gamma(\R^d,\R)$ for some $\gamma>0$, and $b^n\to b$ in $\B^{-\gamma/2}_1$ as $n\to\infty$. Therefore,  on $\Omega''$
$$
X_t=x+\psi(t)+W^H(t)=x+\int_{\R^d}L_t^X(y)b(dy)+W^H(t),
$$
which is \eqref{measureeq}.

(ii). 
Let $(X,W^H)$ be a solution to \eqref{measureeq}. Define $\psi(t):=\int_{\R^d}L_t^X(y)b(dy)$. Clearly, $\psi$ is of bounded variation, and thus $X$ belongs to \textbf{BV}. By \cref{T:loc}(ii), there exists a set of full measure $\Omega'\subset\Omega$ and $\gamma>0$ such that on $\Omega'$
\begin{equation}\label{lambdabound}
\sup_{t\in[0,1]}\|L_t^X\|_{\C^\gamma(\R^d)}<\infty.
\end{equation}

Now let us verify that $X$ is regularized solution to \eqref{mainSDE}. 
We see that $X=x+\psi+W^H$ and therefore part (i) of \cref{D:sol} is satisfied. Next, we fix a sequence  $(b^n)_{n\in\Z_+}$ of $\C^\infty_b(\R^d,\R^d)$ functions converging to $b$ in $\B^{0-}_1$.
\begin{align*}
\sup_{t\in[0,1]}\Bigl|\int_0^t b^n(X_r)\,dr-\psi_t\Bigr|&=
\sup_{t\in[0,1]}\Bigl|\int_{\R^d} L_t^X(y) b^n(y)\,dr-\int_{\R^d} L_t^X(y) b(dy)\Bigr|\\
&\le \|b^n-b\|_{\B^{-\frac\gamma2}_1}\sup_{t\in[0,1]}\|L_t^X\|_{\C^\gamma(\R^d)},
\end{align*}
where the last inequality follows from \cref{p:ltc}. 
Recalling \eqref{lambdabound}, we see that $\sup_{t\in[0,1]}\Bigl|\int_0^t b^n(X_r)\,dr-\psi_t\Bigr|\to0$ as $n\to\infty$ on $\Omega'$ and thus  part (ii) of \cref{D:sol} also holds.  Hence $(X,W^H)$ solves equation \eqref{mainSDE} in the sense of \cref{D:sol}.

Bound \eqref{reztmeas} follows from \cref{L:51m}.
\end{proof}

\subsection{Weak existence of solutions to SDE}\label{s:WU}

Now, let us show the existence of weak solutions to the SDE \eqref{mainSDE}. Keeping in mind that we will later need to establish the strong existence of solutions and apply \cite[Lemma~1.1]{MR1392450}, we prepare by also providing more general results related to weak existence.

We will prove the weak existence of regularized solutions to \eqref{mainSDE} for $b \in \B^0_p$, where $p$ satisfies \eqref{maincond}. Thanks to the embedding $L_p(\R^d,\R^d) \subset \B^0_p$ and the already established equivalence of regularized solutions and solutions to \Gref{x;b} (\cref{T:func}\ref{part3tfunc}), this will imply the weak existence of solutions to \Gref{x;b}, that is, \cref{T:func}\ref{part1tfunc}. \cref{T:ident}\ref{part1tm} will  follow similarly from the embedding $\M(\R^d,\R^d) \subset \B^0_1$.

Until the end of this section, we fix the parameters $H,d,p$ satisfying \eqref{maincond} and assume $b \in \B^0_p$.

\begin{lemma}[Tightness]\label{l:tght}
Let $(x_n')_{n\in\Z_+}$, $(x_n'')_{n\in\Z_+}$ be two sequences  of vectors in  $\R^d$. Let $(b_n')_{n\in\Z_+}$, $(b_n'')_{n\in\Z_+}$ be two sequences of $\C^1(\R^d,\R^d)$ functions such that
\begin{equation}\label{condabs}
\sup_{n\in\Z_+}(\|b_n'\|_{L_p(\R^d)}+ \|b_n''\|_{L_p(\R^d)}+|x_n'|+|x_n''|)<\infty,
\end{equation}
Let $X_n'$, $X_n''$ be strong solutions to \Gref{x_n';b_n'},
\Gref{x_n'';b_n''}, respectively.

Then there exists a subsequence $(n_k)_{k\in\Z_+}$ such that $(X_{n_k}',X_{n_k}'',W^H,B)_{k\in\Z_+}$ converges weakly in the space $\C([0,1],\R^{3d})\times\C^\theta([0,1],\R)$ for any $\theta\in[0,\frac12)$.
\end{lemma}
\begin{proof}
Denote 
$$
M:=\sup_{n\in\Z_+}\|b_n'\|_{L_p(\R^d)}.
$$
Introduce the process 
\begin{equation*}
\psi_n'(t):=X_n'(t)-W^H(t)=x_n'+\int_0^t b'_n(X_n'(r))\,dr,\quad t\in[0,1].
\end{equation*}
We  apply \cref{L:51} with $f=b_n'$. We get that for any $m\ge2$ there exists a constant $C=C(H,d,m,M,p)>0$,  such that for any $(s,t)\in\Delta_{[0,1]}$
\begin{equation*}
\|\psi_n'(t)-\psi_n'(s)\|_{L_m(\Omega)}\le C(t-s)^{1-\frac{Hd}p}.
\end{equation*}
Therefore, by the Kolmogorov continuity theorem, for any $\eps>0$ there exists a constant  $C=C(H,\eps,d,M,p)>0$ such that
\begin{equation}\label{Holdfin}
\E [\psi'_n(\omega)]_{\C^{1-\frac{Hd}p-\eps}([0,1])}\le C,
\end{equation}
and $C$ does not depend on $n$.
Let us prove now that the sequence $(\psi'_n)_{n\in\Z_+}$ is tight in $\C([0,1])=\C([0,1],\R^d)$. For $N>0$ put
$$
H_N:=\{f\in \C([0,1]): |f(0)|\le R,\,[f]_{\C^{1-\frac{Hd}p-\eps}([0,1])}\le N\},
$$
where $R:=\sup_{n\in\Z_+}|x_n'|<\infty$. 
By the Arzel\`a--Ascoli theorem, for each $N>0$ the set $H_N$ is a compact set. Furthermore, by \eqref{Holdfin} for any  $n\in\Z_+$ we have
$$
\P(\psi'_n\notin H_N)=\P([\psi'_n(\omega)]_{\C^{1-\frac{Hd}p-\eps}} >N)\le C  N^{-1}.
$$
Thus, the sequence $(\psi'_n)_{n\in\Z_+}$ is tight in $\C([0,1])$. Recalling that $X_n'=\psi_n'+W^H$ and the addition is continuous in $\C([0,1])$, we see that $(X'_n)_{n\in\Z_+}$
is tight in $\C([0,1])$. Similarly,  the sequence $(X''_n)_{n\in\Z_+}$ is tight. Thus, 
$(X_{n}',X_{n}'',W^H,B)_{k\in\Z_+}$ is tight. By the Prokhorov theorem, this implies the desired statement. 
\end{proof}

\begin{lemma}[Identification of the limit]\label{L:ident}
Let $(f_n)_{n\in\Z_+}$  be a sequence of $\C^1(\R^d,\R^d)$ functions converging to $b$ in $\B^{0-}_p$. Suppose that 
\begin{equation}\label{condbabs}
\sup_{n\in\Z_+}\|f_n\|_{L_p(\R^d)}<\infty.
\end{equation}
Let $(W^H_n)_{n\in\Z_+}$ be a sequence of fractional Brownian motions with the same Hurst parameter $H\in(0,1)$.
Let $(x_n)_{n\in\Z_+}$ be a sequence  of vectors in  $\R^d$ converging to $x\in\R^d$. Let $X_n$  be a strong solution to \Gref{x_n;f_n} with $W^H_n$ in place of $W^H$.

Suppose that there exists measurable functions $X,W^H\colon\Omega\times [0,1]\to\R^d$ such that $(X_n,W_n^H)$ converges to $(X,W^H)$ in space $\C([0,1],\R^{2d})$ in probability as $n\to\infty$.

Then $X$ is a regularized solution to \eqref{mainSDE}. Further, for any $m\ge1$
	\begin{equation}\label{TVbound}
		\|\,\| X-W^H\|_{1-\var;[0,1]}\,\|_{L_m(\Omega)}<\infty.
	\end{equation}
\end{lemma} 
\begin{proof}
First, we see that $W^H$ has the law of fractional Brownian motion with the Hurst parameter $H$. Therefore, defining 
\begin{equation}\label{psiprdef}
\psi(t):= X(t)-x- W^H(t),\quad t\in[0,1],
\end{equation}
we see that part (1) of  \cref{D:sol} holds.

To check part (2) of  \cref{D:sol}, we fix a sequence $(b_n)_{n\in\Z_+}$ of $\C^\infty_b(\R^d,\R^d)$ functions converging to $b$ in $\B^{0-}_p$. We put for $n\in\Z_+$, $t\in[0,1]$
\begin{equation*}%\label{Psin}
\psi_n(t):=\int_0^t b_n( X_r)\,dr.
\end{equation*}
	
Our goal is to prove that $\|\sup_{t\in[0,1]} |\psi_n(t)-\psi(t)|\,\|_{L_2(\Omega)}\to0$. Recalling the definition of $\psi$ in \eqref{psiprdef} and that $X_k$ is a solution to  \Gref{x_k;f_k}, we deduce for any $n,k\in\Z_+$
	\begin{align}\label{baza0}
		&\sup_{t\in[0,1]} |\psi_n(t)-\psi(t)|\nn\\
		&\qquad\le
		\sup_{t\in[0,1]} \Bigl|\int_0^t b_n(X(r))\,dr -\int_0^t b_n(X_k(r))\,dr|+
		\sup_{t\in[0,1]} \Bigl|\int_0^t (b_n-f_k)(X_k(r))\,dr\Bigr|\nn\\
		&\qquad\phantom{\le} +\sup_{t\in[0,1]} \Bigl|\int_0^t f_k(X_k(r))\,dr-\psi(t)\Bigr|\nn\\
		&\qquad\le (\|b_n\|_{\C^1(\R^d)}+1) \bigl(\sup_{t\in[0,1]} |X_k(t)-X(t)|+
		\sup_{t\in[0,1]} | W^H_k(t)- W^H(t)|+|x_k-x|\bigr)\nn\\
		&\qquad\phantom{\le}+ 	\sup_{t\in[0,1]} \Bigl|\int_0^t (b_n-f_k)(X_k(r))\,dr\Bigr|\nn\\
		&\qquad=:I_1(n,k)+I_2(n,k).
	\end{align}
Recall that by assumption, $(X_k,W_k^H)$ converges to $(X,W^H)$ in $\C([0,1],\R^{2d})$ in probability and $x_k$ converges to $x$. 
Therefore, for any fixed $n\in\Z_+$
\begin{equation}\label{baza}
\lim_{k\to\infty} I_1(n,k)=0\quad\text{in probability}.
\end{equation}

Further, using again that $(X_k, W^H_k)$ is a solution to \Gref{x_k,f_k} and that condition \eqref{condbabs} holds, we see that all the assumptions of \cref{L:51} with $f=b_k$ are satisfied. Hence for any $m\ge1$ we have
\begin{equation}\label{var}
\|\,\| X_k-W^H_k\|_{1-\var;[0,1]}\,\|_{L_m(\Omega)}\le C,
\end{equation}
for some constant $C=C(H,d,m,p)>0$ independent of $k$. We take $\delta>0$ small enough such that \eqref{maincondupd} holds and apply \cref{L:driftb2} with $f=b_n-f_k$, $z=X_k-W_k^H$, $\tau=1$, $q=p\vee2$, 
and $\alpha = -\delta + \I_{p \in [1,2]}\big(\frac{d}{2} - \frac{d}{p}\big)$. By \cref{c:driftb2}, conditions \eqref{maincond1}--\eqref{maincond2} are satisfied. Moreover, the function $b^n - f_k$ is bounded. Therefore, all the conditions of \cref{L:driftb2} are satisfied, and we obtain from \eqref{finbound} for any $m\ge2$, $(s,t)\in\Delta_{[0,1]}$
\begin{align*}
\Bigl\|\int_s^t (b_n-f_k)( X_k(r))\,dr\Bigr\|_{L_m(\Omega)}\le C\|b_n-f_k\|_{\B^{-\delta}_p}|t-s|^{1-\frac{Hd}{p}-H-\delta}
\end{align*}
where $C=C(H,d,\delta,m,p)>0$ does not depend on $n,k$. This implies, by the Kolmogorov continuity theorem that
\begin{equation*}
\|I_2(n,k)\|_{L_2(\Omega)}	=	\Bigl\|\sup_{t\in[0,1]} \Bigl|\int_0^t (b_n-f_k)( X_k(r))\,dr\Bigr|\Bigr\|_{L_2(\Omega)}\le C\|b_n-f_k\|_{\B^{-\delta}_p}\to0\,\,\text{as $k,n\to\infty$}.
\end{equation*}
Combining this with \eqref{baza} and passing to the limit in \eqref{baza0} first as $k\to\infty$ and then as $n\to\infty$, we finally get
$$
\sup_{t\in[0,1]} |\psi_n(t)-\psi(t)|\to0,\quad \text{in probability as $n\to\infty$},
$$
and thus $X$ is indeed a solution to \eqref{mainSDE}.

Finally, \eqref{var}, \cref{p:TV}, and Fatou's lemma imply \eqref{TVbound}.
\end{proof}

\begin{corollary}\label{c:ews}
Assume that all the conditions of \cref{l:tght} are satisfied. Suppose additionally that 
$b_n'\to b$, $b_n''\to b$ in $\B^{0-}_p$ as $n\to\infty$ and $x_n'\to x$, $x_n''\to x$ as $n\to\infty$.
Then the following holds. There exists a
filtered probability space $(\wh\Omega, \wh \F,(\wh \F_t)_{t\in[0,1]},\wh \P)$,
an $(\wh \F_t)$-fractional Brownian motion $\wh W$ defined on this space, measurable functions 
$\wh X',\wh X''\colon[0,1]\times\wh\Omega\to\R^d$ such that
\begin{enumerate}[(i)]
\item both $\wh X'$ and $\wh X''$ are adapted to the filtration $(\wh \F_t)$ and are weak regularized solutions to \eqref{mainSDE} with the initial condition $x$;
\item for any $m\ge1$ we have 
$$
\|\,\|\wh X'-\wh W^H\|_{1-\var;[0,1]}\,\|_{L_m(\Omega)}<\infty,\quad  \|\,\|\wh X''-\wh W^H\|_{1-\var;[0,1]}\,\|_{L_m(\Omega)}<\infty.
$$
\item there exists a subsequence $(n_k)_{k\in\Z_+}$ such that $(X_{n_k}',X_{n_k}'')_{k\in\Z_+}$ converge weakly in $\C([0,1],\R^{2d})$ to $(\wh X',\wh X'')$.
\end{enumerate}
\end{corollary}
\begin{proof}
(i). Take $\theta\in[0,\frac12)$, $\theta>H-\frac12$. By \cref{l:tght},  the sequence $(X_{n_k}',X_{n_k}'',W^H,B)_{k\in\Z_+}$ weakly converges in $\C([0,1],\R^{3d})\times \C^\theta([0,1],\R^{d})\ $. By passing to an appropriate subsequence and applying the Skorokhod representation theorem there exists a random element $(\wh X',\wh X'',\wh W^H, \wh B)$ and a sequence of random elements $(\wh X'_n,\wh X''_n,\wh W^H_n,\wh B_n)$ defined on a common probability space $(\wh\Omega, \wh\F, \wh P)$
such that $\Law(\wh X'_n,\wh X''_n,\wh W^H_n,\wh B_n)=\Law( X'_n, X''_n, W^H_n, B_n)$ and
\begin{equation}\label{prishli}
	\|(\wh X'_n,\wh X''_n,\wh W^H_n)-(\wh X',\wh X'',\wh W^H)\|_{\C([0,1])}+\|\wh B_n-\wh B\|_{\C^\theta([0,1])}\to0 \,\,\,\text{as $n\to\infty$ a.s.}
\end{equation}
for any $m\ge1$. 
Obviously, we have  $\Law(\wh B_n)=\Law(B)$, and thus, $\wh B_n$, $\wh B$ are standard Brownian motions. For $t\in[0,1]$, put $\wh \F_t:=\sigma(\wh B_s,\wh X'_s,\wh X''_s; s\le t)$. We claim that $\wh B$ is an $(\wh \F_t)$-Brownian motion.

Indeed, for any $(s,t)\in\Delta_{[0,1]}$, $n,N\in\Z_+$, any bounded continuous functions $f\colon\R^d\to\R$, 
$g\colon\R^{3Nd}\to\R$, and any time points $0\le t_1\le\hdots\le t_N\le s$ one has
\begin{align*}
&\E f(\wh B_n(t)-\wh B_n(s)) g\bigl(\wh B_n(t_1),\hdots, \wh B_n(t_N), \wh X'_n(t_1),\hdots, \wh X'_n(t_N), \wh X''_n(t_1),\hdots, \wh X''_n(t_N)\bigr)\\
&\quad=\E  f(\wh B_n(t)-\wh B_n(s))\E  g\bigl(\wh B_n(t_1),\hdots, \wh B_n(t_N), \wh X'_n(t_1),\hdots, \wh X'_n(t_N), \wh X''_n(t_1),\hdots, \wh X''_n(t_N)\bigr),
\end{align*}
since $\wh B_n(t)-\wh B_n(s)$ is independent of $\wh \F^n_s:=\sigma(\wh B^n_r, \wh X'_n(r), \wh X''_n(r); s\le t)$. By passing to the limit in the above expression as $n\to\infty$ using \eqref{prishli}, one derives that the same identity holds for $(\wh B,\wh X',\wh X'')$ in place of $(\wh B_n,\wh X'_n,\wh X''_n)$, which implies that $\wh B(t)-\wh B(s)$ is independent of $\wh \F_s$. Thus, $\wh B$ is an $(\wh \F_t)$-Brownian motion.

Recall that we have $W^H=\Psi(B)$ for a certain functional $\Psi$, see relationship \eqref{WB}.

If $H\le \frac12$, then it follows from \cref{p:cont} that $B=\Phi(W^H)$ for a continuous functional $\Phi\colon\C([0,1])\to
\C([0,1])$. Therefore, for any $n\in\Z_+$ we have $\wh B_n=\Phi(\wh W^H_n)$. Passing to the limit as $n\to\infty$ in this identity and using continuity of $\Phi$ and \eqref{prishli}, we get $\wh B=\Phi(\wh W^H)$. Therefore, $\Psi(\wh B)=\Psi\circ\Phi(\wh W^H)= \wh W^H$ thanks to \cref{p:cont}. Since $\wh B$ is
an $(\wh \F_t)$-Brownian motion,  this means that $\wh W^H$ is an $(\wh \F_t)$-fractional Brownian motion with the Hurst index $H$.

If $H>\frac12$, then by  \cref{p:cont} the functional $\Psi\colon\C^\theta_0([0,1])\to\C([0,1])$ is continuous (recall that we chose $\theta>H-\frac12$). We also have $\wh W^H_n=\Psi(\wh B_n)$ and $\wh B_n$ converge to $\wh B$ in $\C^\theta([0,1])$. Hence, $\wh W^H=\Psi(\wh B)$. Thus $\wh W^H$ is an $(\wh \F_t)$-fractional Brownian motion.

By definition, $\wh X'$ and $\wh X''$ are adapted to $(\wh \F_t)$. By \cref{L:ident}
$\wh X'$ and $\wh X''$ are  solutions to \eqref{mainSDE} in the sense of \cref{D:sol}. Hence, $\wh X'$ and $\wh X''$ are weak solutions to this equation.

(ii). Follows from \eqref{TVbound}.

(iii). Follows from \eqref{prishli}.
\end{proof}

Summarizing the results presented in this section, we can complete the proofs of  \cref{T:func,T:ident}. 

\begin{proof}[Proof of \cref{T:func}(i)] Let $b\in L_p(\R^d,\R^d)$. Then $b\in\B^0_p$.
We apply \cref{l:tght} with  $b_n'=b_n''=P_{\frac1n} b$, $x_n'=x_n''=x$, $n\in\Z_+$. \cref{p:veryboring}(ii) guarantees that  $b_n'\to b$ in $\B^{0-}_p$ and $\sup_{n\in\Z_+}\|b_n'\|_{L_p(\R^d)}<\infty$.  Therefore, all the conditions of \cref{l:tght} are satisfied and by \cref{c:ews}, equation \eqref{mainSDE} has a regularized weak solution $(X,W^H)$ and this solution is in \textbf{BV}.  \cref{T:func}\ref{part3tfunc} obtained above implies that $(X,W^H)$ solves \Gref{x;b}.
\end{proof}

\begin{proof}[Proof of \cref{T:ident}](i). 
Let $b\in\mathcal{M}(\R^d,\R^d)$. Then $b\in\B^0_1$. Similarly, we apply \cref{l:tght} with  
$b_n'=b_n''=P_{\frac1n} b$, $x_n'=x_n''=x$, $n\in\Z_+$.   \cref{p:veryboring}(i) implies that $b_n'\to b$ in $\B^{0-}_1$ and $\sup_{n\in\Z_+}\|b_n'\|_{L_1(\R^d)}<\infty$. Further, condition $H<1/(d+1)$ is equivalent to condition \eqref{maincond} for $p=1$. Therefore, all the conditions of \cref{l:tght} with $p=1$ are satisfied, and \cref{c:ews} implies the existence of weak solution  to \eqref{mainSDE} that is in \textbf{BV}.

(ii). The tightness of the sequence $(\Law(X_n,W^H_n))_{n\in\Z_+}$ follows from  \cref{l:tght} with $p=1$. \cref{L:ident} and the Skorokhod representation theorem imply that any of the partial limits of $(\Law(X_n,W^H))_{n\in\Z_+}$ is a weak regularized solution to \eqref{mainSDE} (see also the proof of \cref{c:ews}(i)).
\end{proof}

\subsection{Strong uniqueness of solutions of SDE}\label{s:SU}

Now, we proceed to the strong well-posedness of \eqref{mainSDE}. As mentioned earlier, we will rely on \cref{L:driftb2} with $f = \nabla b$ and \cref{T:YoungODE}. The latter imposes the restriction $d = 1$, which is additionally assumed in this subsection.
Similar to the weak existence proof above, we will establish strong existence for $b \in \B^0_p$. \cref{T:uniq} will then follow from the embeddings $L_p(\R^d,\R^d) \subset \B^0_p$ and $\mathcal{M}(\R^d,\R^d) \subset \B^0_1$.

\begin{lemma}\label{L:uniq} Let $d=1$, $p\in[1,\infty]$, $H\in(0,\frac{p}{2p+1})$, $b\in\B^0_p$. If $p\in[1,2]$, then suppose additionally that \eqref{uniqcond} holds.
Let $(X,W^H)$, $(Y,W^H)$ be two weak regularized solutions of \eqref{mainSDE} with the same initial condition $x\in\R^d$ defined on the same probability space and adapted to the same filtration $(\F_t)$. Suppose that 
for any $m\ge1$
$$
\|\,\|X-W^H\|_{1-\var;[0,1]}\|_{L_m(\Omega)}<\infty,\qquad \|\,\|Y-W^H\|_{1-\var;[0,1]}\|_{L_m(\Omega)}<\infty.
$$
Then $\P(X_t=Y_t \text{ for all $t\in[0,1]$})=1$.
\end{lemma}
\begin{proof}
We denote for $t\in[0,1]$
$$
\psi(t):=X(t)-W^H(t),\quad \phi(t):=Y(t)-W^H(t),\quad v(t)=\psi(t)-\phi(t).
$$
Our goal is to prove that $v=0$ a.s. As discussed above, we aim to apply \cref{L:driftb2} for $f = \nabla b$ and then conclude with \cref{T:YoungODE}. Note, however, that $\nabla b$ is not defined as a function (but only as a distribution), and therefore, to apply \cref{L:driftb2}, one must consider appropriate approximations.

For $n\in \N$ we set $b_n := P_{\frac1n}b$. By \cite[Lemma~A.3]{ABLM}, $b_n\to b$ in $\B_p^{0-}$. 
Define
\begin{equation*}
\psi_n(t):=x+\int_0^t b_n(W_r^H+\psi_r)\,dr,\quad \phi_n(t):=x+\int_0^t b_n(W_r^H+\phi_r)\,dr,\qquad t\in[0,1].
\end{equation*}
By definition of the regularized solution and by passing to an appropriate subsequence we can assume that there is a set $\Omega'\subset\Omega$ of full measure such that on $\Omega'$  we have $\|\psi_n-\psi\|_{\C([0,1])}\to0$, $\|\phi_n-\phi\|_{\C([0,1])}\to0$  as $n\to\infty$. Hence for $t\in[0,1]$  on $\Omega'$
\begin{align}\label{st1swp}
v(t)&=\psi(t)-\phi(t)= \lim_{n\to\infty}(\psi_n(t)-\phi_n(t))\nn\\
&=\lim_{n\to\infty}\int_0^{t} (b_n(W_r^H+\psi_r)-b_n(W_r^H+\phi_r))\,dr
=\lim_{n\to\infty}\int_0^{t} v_r\,dR_n(r),
\end{align}
where we denoted
\begin{equation*}
R_n(t):=\int_0^{t}\int_0^1 \nabla b_n(W_r^H+\theta\phi_r+(1-\theta)\psi_r)\,d\theta\,dr,\quad n\in\N.
\end{equation*}
Let us now prove that the sequence of random elements $(R_n)_{n\in\N}$ has a subsequence which converges a.s. in $\C^{\gamma}([0,1])$ for some $\gamma>0$. It is easy to check that the assumption $H<p/(2p+1)$ implies $\frac{H^2}p+H-\frac12<0$. Therefore, there exists $\eps>0$ such that
\begin{equation}\label{epsassum1}
2H+\frac{H}p+\eps H<1\quad\text{and}\quad \frac{H^2}p+H(1+\eps)-\frac12<0.
\end{equation}
If $p\in[1,2]$, then we assume additionally that 
\begin{equation}\label{epsassum2}
H^2+H(1+\frac2p+2\eps)-1<0,
\end{equation}
this is possible thanks to the extra condition \eqref{uniqcond}. 

Fix now $n,k\in\Z_+$.
We apply \cref{L:driftb2} twice to the functions $f:=\nabla b^n-\nabla b^k$ and $f:=\nabla b^n$, respectively. The remaining parameters are as follows: $m\ge2$, $d=1$, $\tau=1$, $q=p\vee2$, $\alpha=-1-\eps+\I_{p\in[1,2]}(\frac12-\frac1p)$, $z=\theta\phi+(1-\theta)\psi$, where $\theta\in[0,1]$. We see that $\nabla b^n$, $\nabla b^n-\nabla b^k$ are bounded by construction. Further, it follows from conditions \eqref{epsassum1} and \eqref{epsassum2} that conditions \eqref{maincond1} and \eqref{maincond2} hold.
Therefore, all the conditions of \cref{L:driftb2} are satisfied, and there exists a constant $C=C(H,\eps,m,p)>0$ such that for any $s,t\in[0,1]$, $n\in\Z_+$ we have
\begin{align}\label{eqlimn}
&\|R_n(t)-R_n(s)\|_{L_m(\Omega)}\nn\\
&\quad\le \int_0^1\Bigl\|\int_s^t \nabla b^n(W_r^H+\theta\phi_r+(1-\theta)\psi_r)\,dr\Bigr\|_{L_m(\Omega)}\,d\theta\nn\\
&\quad\le C \| b_n\|_{\B^{-\eps}_p}(1+\|\,\|\phi\|_{1-\var;[0,1]}\|_{L_m(\Omega)}+\|\,\|\psi\|_{1-\var;[0,1]}\|_{L_m(\Omega)})(t-s)^{1-H(2+\frac1p+\eps)},
\end{align}
where we used inequality $\|\nabla b_n\|_{\B^{\alpha}_q}\le \|\nabla b_n\|_{\B^{-1-\eps}_p}\le \| b_n\|_{\B^{-\eps}_p}$. Similarly, we get from \eqref{finbound} that there exists a constant $C=C(H,\eps,m,p)>0$ such that for any  $t\in[0,1]$, $n,k\in\Z_+$ we have
\begin{equation}\label{eqlimn1}
\|R_n(t)-R_k(t)\|_{L_m(\Omega)}\le C (1+\|\|\phi\|_{1-\var;[0,1]}\|_{L_m(\Omega)}+\|\|\psi\|_{1-\var;[0,1]}\|_{L_m(\Omega)})\| b_n-b_k\|_{\B^{-\eps}_p}.
\end{equation}
Note that the constants $C$ in \eqref{eqlimn}, \eqref{eqlimn1} do not depend on $n,k$. Therefore, by passing to a subsequence, if necessary, and applying the Kolmogorov continuity theorem in the form of \cref{p:KCT}, we see that  there exists a set of full probability measure $\wt\Omega\subset \Omega$ and a measurable process $R\colon\Omega\times[0,1]\to\R$, such that on $\wt\Omega$
\begin{equation*}%\label{convRR}
\|R\|_{\C^\gamma([0,1])}<\infty \quad \text{and } \|R-R_n\|_{\C^\gamma([0,1])}\to0
\end{equation*}
for some $\gamma>0$.

 Note, that $v$ is a process of finite $1$-variation. Hence, \eqref{st1swp} and stability theorem for Young integrals \cite[Proposition~6.12]{FV2010} implies that on  $\wt\Omega \cap \Omega'$ we have 
\begin{equation*}
v(t)=\lim_{n\to\infty}\int_0^{t} v_r\,dR_n(r)=\int_0^{t} v_r\,dR(r).
\end{equation*}
By \cref{T:YoungODE}, this implies that $v=0$ on $\wt\Omega \cap \Omega'$ and thus ${\P(X=Y)\ge P(\wt\Omega \cap \Omega')=1}$.
\end{proof}

\subsection{Strong existence of solutions of SDE and completion of the proof of \cref{T:uniq}}\label{s:SE}
Finally, let us show the strong existence of solutions to the SDE \eqref{mainSDE} and complete the proof of \cref{T:uniq}. As explained above, strong existence follows from the strong uniqueness and weak existence by the Gy\"ongy–Krylov lemma, which we put here for the sake of completeness and convenience of the reader

\begin{proposition}[{\cite[Lemma~1.1]{MR1392450}}]\label{L:GK}
Let $(Z_n)_{n\in\Z_+}$ be a sequence of random elements in
a Polish space $(E, \rho)$ equipped with the Borel $\sigma$-algebra. Assume that for every pair of
subsequences $(Z_{l_k})$ and $(Z_{m_k})$ there exists a further sub-subsequence $(Z_{l_{k_r}}, 
Z_{m_{k_r}})$ which converges weakly in $E\times E$ to a random element $w = (w',w'')$ such that 
$w' = w''$ a.s. Then there exists an $E$-valued random element $Z$ such that $(Z_n)$ converges in probability to $Z$ as $n\to\infty$.
\end{proposition}	

\begin{proof}[Proof of Theorem~\ref{T:uniq}]
	We will use \cref{L:GK}.
(i). For $n\in\N$ put $b_n:=P_{\frac1n}b$. Let $X_n$ be the strong solution to \Gref{x;b_n}. Let $(b'_n,X'_n)$, $(b''_n,X''_n)$, be two arbitrary subsequences of $(b_n,X_n)$. We apply \cref{l:tght} and \cref{c:ews}. We see that \eqref{condabs} is satisfied thanks to \cref{p:veryboring} and \eqref{maincond} holds with $p=1$ thanks to the assumption $H<(\sqrt13-3)/2$. It follows that there  exist a filtered probability space $(\wh \Omega, \wh \F, (\wh\F_t)_{t\in [0,1]},\wh \P)$, 
an $(\wh \F_t)$-fractional Brownian motion $\wh W$ defined on this space and a pair of weak solutions $(\wh X',\wh X'')$  to \eqref{mainSDE}  adapted to the filtration $(\wh \F_t)$. We see also that $\wh X',\wh X''\in\text{\textbf{BV}}$ and  there exists a subsequence $(n_k)$ such that $(X'_{n_k},X''_{n_k})$ converges weakly in $\C([0,1],\R^{2})$ to $(\wh X',\wh X'')$ as $k\to\infty$. 

By \cref{T:measure}, for any $m\ge1$
\begin{equation*}
	\|\,\|\wh X'-\wh W^H\|_{1-\var;[0,1]}\,\|_{L_m(\Omega)}+\|\,\|\wh X''-\wh W^H\|_{1-\var;[0,1]}\,\|_{L_m(\Omega)}<\infty.
\end{equation*}
It is easy to see that the standing assumption $H < (\sqrt{13} - 3)/2$ implies \eqref{uniqcond} with $p = 1$. Hence, recalling the embedding $\mathcal{M} \subset \mathcal{B}^0_1$, we see that all the conditions of \cref{L:uniq} are satisfied and $\wh{X}' = \wh{X}''$. Therefore, \cref{L:GK} now implies that there exists a $\C([0, 1], \mathbb{R})$-valued random element $X$, such that $X_n$ converges to $X$ in probability as $n \to \infty$. \cref{L:ident} now implies that $X$ is a regularized solution of \eqref{mainSDE} and is in the class \textbf{BV}. Since $X_n$ is $(\mathcal{F}_t^{W^H})$-measurable, $X$ is also $(\mathcal{F}_t^{W^H})$-measurable, and thus $X$ is a strong solution to \eqref{mainSDE}.

Let $Y$ be another solution to \eqref{mainSDE} adapted to a different filtration $(\F_t^Y)\supset(\F_t^{W^H})$. Since $X$ is a strong solution to \eqref{mainSDE}, $X$ is also adapted to  $(\F_t^Y)$. It is also clear that $Y\in\text{\textbf{BV}}$, because $Y-W^H$ is increasing. Then, by \cref{T:measure}, 
$\|\,\|Y- W^H\|_{1-\var;[s,t]}\,\|_{L_m(\Omega)}<\infty$.
Therefore, \cref{L:uniq} yields $X=Y$ and pathwise uniqueness holds.

(ii). The proof is exactly the same as (i) with the only difference that now it is assumed that $Y\in\text{\textbf{BV}}$.

(iii). Arguing as in part (i), we get that there exists a strong regularized solution  $X$ to \eqref{mainSDE} and this solution is in \textbf{BV}. By \cref{T:func}\ref{part3tfunc} $X$ solves \Gref{x;b} and it is a strong solution.

Let $Y$ be another solution to \Gref{x;b} adapted to a different filtration $(\F_t^Y)\supset(\F_t^{W^H})$. Then we see again that $X$ is also adapted to  $(\F_t^Y)$. Further, \cref{T:func}\ref{partiiitfunc} implies that $Y$ is a regularized solution to \eqref{mainSDE}. By \cref{T:func}\ref{part2tfunc},  for any $m\ge1$
\begin{equation*}
	\|\,\| X- W^H\|_{1-\var;[0,1]}\,\|_{L_m(\Omega)}+\|\,\| Y- W^H\|_{1-\var;[0,1]}\,\|_{L_m(\Omega)}<\infty.
\end{equation*}
Therefore, the embedding $L_p(\R^d)\subset \B^0_p$ and  \cref{L:uniq} imply that $X=Y$ and thus pathwise uniqueness holds.
\end{proof}

\subsection{Non-existence of solutions to SDE}\label{S:NE}
The last subsection is devoted to showing optimality of condition \eqref{mainSDE}.

\begin{proof}[Proof of \cref{T:genopt}]
Assume the contrary and let $X$ be a continuous function which solves  equation \eqref{SDEce}. Note that $X_0=0$. Fix arbitrary $\eps\in(0,\frac1d)$ and suppose that for some $t\in(0,1]$ we have $|X(t)|>d\eps$. Set 
\begin{equation*}
t''_\eps:=\inf\{s>0: |X(s)|\ge d \eps\}.
\end{equation*}
By continuity of $X$, we have $|X(t''_\eps)|=d\eps$, and therefore for some 
$i\in 1,\hdots, d$ we have $|X^i(t''_\eps)|\ge \eps$. Without loss of generality, assume that $X^i(t''_\eps)>0$. Put 
\begin{equation*}
t'_\eps:=\sup\{s\in(0,t''_\eps): X^i(s)=0\}.
\end{equation*}
Then $X^i(t'_\eps)=0$ and  
\begin{equation*}
X^i(s)>0;\,\, 0<|X(s)|<d\eps\quad \text{for $s\in(t'_\eps,t''_\eps)$.}
\end{equation*}
Then we derive
\begin{align*}
\eps&\le X^i(t''_\eps)-X^i(t'_\eps)=-\int_{t'_\eps}^{t''_\eps}|X(t)|^{-\alpha} dt+f^i(t''_\eps)-f^i(t'_\eps)\\
&\le -(t''_\eps-t'_\eps)d^{-\alpha}\eps^{-\alpha}+K(t''_\eps-t'_\eps)^{\gamma}\\
&\le \sup_{t>0} (-td^{-\alpha}\eps^{-\alpha}+Kt^{\gamma})\le \gamma^{\frac{\gamma}{1-\gamma}} K^{\frac1{1-\gamma}}(\eps d)^{\frac{\alpha \gamma }{1-\gamma}},
\end{align*}
where we denoted $K:=[f]_{\C^{\gamma}([0,1])}<\infty$.
Therefore,
\begin{equation*}
K\ge C(\gamma,d) \eps^{1-\gamma-\alpha \gamma }.
\end{equation*}
Since $\eps>0$ was arbitrary, we pass to the limit in the above inequality as $\eps\searrow0$. By assumption, we have $1-\gamma-\alpha \gamma<0$. Therefore we end up with 
\begin{equation*}
[f]_{\C^{\gamma}([0,1])}=K=\infty,
\end{equation*}
which contradicts the fact that $f\in\C^\gamma([0,1])$. This implies that $X\equiv0$, but this is not a solution to equation \eqref{SDEce}. Therefore, equation  \eqref{SDEce} does not have a solution. 
\end{proof}

\begin{proof}[Proof of \cref{T:opt}]
Suppose we are given $p$ such that \eqref{maincondnot} holds. Take $H'\in(0,H)$, $p'>p$ such that $\frac{d}{p'}>\frac1{H'}-1$ and put  
$\alpha:=\frac{d}{p'}$.  Then, by \cref{T:genopt}, for such choice of $\alpha$, SDE
\begin{equation*}
dX_t^i=-\sign(X_t^i) |X_t|^{-\alpha}\I(|X_t|<1) dt+dW_t^{H,i},\quad i=1,\hdots,d
\end{equation*}
has no solutions since $W^H\in\C^{H'}([0,1])$. To conclude the proof it remains to note that the function $x\mapsto-\sign(x^i)|x|^{-\alpha}\I(|x|<1)\in L_{p}(\R^d)$ for any $i\in1,\hdots,d$.
\end{proof}

\begin{proof}[Proof of \cref{C:nr}]
Fix $0<H<H'<1$. Then there exists some $p\in[1,\infty)$, $d\in\N$, $\alpha>0$ such that
\begin{equation*}
\frac1{H'}-1<\alpha<\frac{d}p<\frac1H-1.
\end{equation*}
Consider a function $b_i(x)=-\sign(x_i)|x|^{-\alpha}\I(|x|<1)$, $x=(x_1,\hdots,x_d)\in\R^d$, $i\in1,\hdots,d$. We see that $b\in L_p(\R^d,\R^d)$. Then \cref{T:func} implies that on some probability space $\wt \Omega$ there is a set of full measure $\wt \Omega_0\subset \wt \Omega$ and continuous processes $X$, $\wt W^H$ such that $\Law (\wt W^H)=\Law  (W^H)$ and on $\wt \Omega_0$
$$
dX_t=b(X_t)dt +d \wt W_t^H.
$$
On the other hand, by \cref{T:genopt}, on a set $\{[\wt W^H]_{\C^{H'}([0,1],\R^d)}<\infty\}$ the above equation has no solutions. Hence 
$$
\P(W^H\in \C^{H'}([0,1],\R^d))=
\P(\wt W^H\in \C^{H'}([0,1],\R^d))\le \P(\wt \Omega\setminus \wt \Omega_0)=0.
$$
Since all the components of $W^H$ are independent and have the same law, we get $\P(W^{H,1}\in\C^{H'}([0,1],\R))=0$, which concludes the proof.
\end{proof}
\appendix

\section{ Useful results on Besov spaces}\label{B.bes}

We will consider Besov space $\B^\alpha_p=\B^\alpha_{p,\infty}$ of regularity $\alpha\in\R$ and with integrability parameter $p\in[1,\infty]$. Recall that for $\alpha<0$, $p\in[1,\infty]$ the  Besov norm is equivalent to the following expression \cite[Theorem~2.34]{bahouri}
\begin{equation}
\|f\|_{\B^\alpha_{p,\infty}}\sim \sup_{t\in(0,1]} t^{-\frac{\alpha}2}\|P_tf\|_{L_p(\R^d)}.\label{modinfF}
\end{equation}

\begin{proposition} 
Let  $f\in \B^{\alpha}_p(\R^d,\R)$, where $d\in\N$, $\alpha\in\R$, $p\in[1,\infty]$, $\lambda\in[0,1]$. Then there exists a constant $C=C(\alpha,\lambda,d,p)>0$ such that for any $x,y\in\R^d$,  one has 
		\begin{equation}\label{fxydif}
			\|f(x+\cdot)\|_{\B^{\alpha}_p}=\|f\|_{\B^{\alpha}_p};\quad
			\|f(x+\cdot)-f(y+\cdot)\|_{\B^{\alpha}_p}\le C |x-y|^{\lambda}\|f\|_{\B^{\alpha+\lambda}_p}.
		\end{equation}
\end{proposition}
\begin{proof}
	The proof is the same as in \cite[Lemma~A.2]{ABLM}.
\end{proof}

\begin{proposition}\label{p:veryboring}
\begin{enumerate}[\rm{(}i\rm{)}]
\item Let $f\in \M(\R^d,\R^d)$, $d\in\N$. Then $\sup_{t\in(0,1]}\| P_t f\,\|_{L_1(\R^d)}<\infty$ and $P_t f\to f$ in $\B^{0-}_1$ as $t\to0$.
\item  Let $f\in L_p(\R^d,\R^d)$, $p\in[1,\infty]$, $d\in\N$. Then $\sup_{t\in(0,1]}\|P_t f\|_{L_p(\R^d)}<\infty$ and $P_t f\to f$ in $\B^{0-}_p$ as $t\to0$.
\end{enumerate}
\end{proposition}
\begin{proof}
(i), (ii). The convergences  $P_tf\to f $ follow from the embeddings $\M(\R^d,\R^d)\subset \B^0_1$, $L_p(\R^d,\R^d)\subset \B^0_p$ and  \cite[Lemma~A.3]{ABLM}.  If $f\in \M(\R^d,\R^d)$, then for any $t>0$ we have
$
 \|P_t f\|_{L_1(\R^d)}\le |f|(\R^d)<\infty.
$
If $f\in L_p(\R^d,\R^d)$, then 
$
 \|P_t f\|_{L_p(\R^d)}\le \|f\|_{L_p(\R^d)}<\infty.
$\end{proof}

\begin{proposition}\label{p:deltacon}
Let $x\in\R^d$, $d\in\N$. Then for any $\eps>0$ there exists a constant $C=C(\eps,d)>0$ such that for any $R>0$ one has
\begin{equation}\label{apa4mr}
\Bigl\|\frac{1}{v_dR^{d}}\I_{\Ba(x,R)}(\cdot)-\delta_x\Bigr\|_{\B^{-\eps}_1}\le CR^\eps.
\end{equation}
and $\frac{1}{v_dR^{d}}\I_{\Ba(x,R)}(\cdot)\to\delta_x$ in $\B^{0-}_1$ as $R\to0$.
\end{proposition}
\begin{proof}
Fix $x\in\R^d$, $\eps>0$. 
Write $f_R(y):=\frac{1}{v_dR^{d}}\I_{\Ba(x,R)}(y)$, $y\in\R^d$. Using an elementary bound 
$$
|p_t(x_1)-p_t(x_2)|\le C |x_1-x_2|^{\eps}t^{-\frac\eps2}(p_{at}(x_1)+p_{at}(x_2)),
$$
valid for some constants $C=C(\eps,d)>0$, $a=a(d)>0$ and any $t>0$, $x,y\in\R^d$, we get 
\begin{align*}
|P_tf_R(y)-P_t \delta_x(y)|&\le \frac{1}{v_dR^{d}}\int_{\Ba(x,R)}|p_t(z-y)-p_t(x-y)|\,dz\\
&\le \frac{C}{v_dR^{d}}R^\eps t^{-\frac\eps2}\int_{\Ba(x,R)}(p_{at}(z-y)+p_{at}(x-y))\,dz.
\end{align*}
Integrating with respect to $y \in \R^d$ and applying Fubini's theorem, we immediately get for $C=C(\eps,d)>0$
\begin{equation*}
\|P_tf_R-P_t \delta_x\|_{L_1(\R^d)}\le CR^\eps t^{-\frac\eps2}, \quad t>0.
\end{equation*}
Recalling now \eqref{modinfF}, we derive
$
\|f_R- \delta_x\|_{\B^{-\eps}_1}\le CR^\eps,
$ where $C=C(\eps,d)$. This implies \eqref{apa4mr}. Further, we see that $f_R\to \delta_x$ in $\B^{-\eps}_1$ as $R\to0$. Since $\eps>0$ was arbitrary and  $\|f_R\|_{\B^0_1}\le\|f_R\|_{L_1(\R^d)}=1$ for any $R>0$, we have $f_R\to \delta_x$ in $\B^{0-}_1$
\end{proof}

\begin{proposition}\label{p:ltc}
	Let $f\in\M(\R^d,\R)$, $d\in\N$. Let $\alpha<0$, $\beta>-\alpha$. Let $f^n$ be a sequence of $\C^\infty_b(\R^d,\R)$ functions converging to $f$ in $\B^{\alpha}_1$ as $n\to\infty$. Let $g\in\C^\beta(\R^d,\R)$. Then there exists a constant $C=C(\alpha,\beta,d)>0$ such that
	\begin{equation*}
		\Bigl|\int_{\R^d}g(y)f^n(y)\,dy- \int_{\R^d}g(y)f(dy)\Bigr|\le  C
		\|g\|_{\C^{\beta}(\R^d)}\|f^n-f\|_{\B^{\alpha}_{1}}
	\end{equation*}
	and 
		\begin{equation*}
		\int_{\R^d}g(y)f^n(y)\,dy\to \int_{\R^d}g(y)f(dy)\quad \text{as $n\to\infty$}.
	\end{equation*}
\end{proposition}

\begin{proof}
It follows from \cite[Theorem~2.17.1 and Proposition 2.3]{Besovbook},
\begin{equation*}
\bigl|	\langle g,f^n-f\rangle\bigr|\le
\|g\|_{\B^{\beta}_{\infty,\infty}}\|f^n-f\|_{\B^{-\beta}_{1,1}}\le C
\|g\|_{\C^{\beta}(\R^d)}\|f^n-f\|_{\B^{\alpha}_{1}}
\end{equation*}	
for  $C=C(\alpha,\beta,d)>0$.
This implies the desired statement.
\end{proof}

\section{Heat kernel bounds}\label{S:41}

\begin{proposition} 
Let  $f\in \B^{\alpha}_p(\R^d,\R)$,  $d\in\N$, $\alpha\le1$, $p\in[1,\infty]$. 
 There exists a constant $C=C(\alpha,p,d)$ such that for any  $t\in(0,1]$
\begin{equation}
\|P_t f\|_{\C^1(\R^d)}\le C t^{\frac\alpha2-\frac{d}{2p}-\frac12}\|f\|_{\B^{\alpha}_p}.
\label{Cbound}
\end{equation}
\end{proposition}
\begin{proof}
The proof is the same as in \cite[Lemma~A.3(iv)]{ABLM}.
\end{proof}

Recall the definition of the process $V_{s,t}$ in \eqref{processV}.

\begin{proposition}\label{P:proptg}
\begin{enumerate}[\rm{(}i\rm{)}]
\item For any measurable function $f\colon\R^d\to\R$ we have
\begin{equation}\label{meanformula}
\E^u f(V_{s,t})= P_{\sigma^2(u,t)}f(\E^u V_{s,t}),\quad  0\le s\le u\le t\le1,
\end{equation}
where $\sigma^2(u,t):=\int_u^t (K_H(t,r))^2\, dr$.
\item There exists a constant $C=C(H,d)$ such that for any $(s,t)\in\Delta_{[0,1]}$ one has 
\begin{equation}\label{KHT}
K_H(t,s)\ge C(t-s)^{H-\frac12},\qquad \sigma^2(s,t)\ge C (t-s)^{2H}.
\end{equation}
\item $\E^u V_{s,t}$, where $0\le s\le u\le t$, is a Gaussian random vector with $d$ independent components, each of mean zero and variance
\begin{equation*}%\label{vareuvst}
\Var[\E^u V^i_{s,t}]\ge C((t-s)^{2H}-(t-u)^{2H}),\quad i=1,\hdots, d,
\end{equation*}
where $C=C(H,d)$.
\item Let $\alpha<0$, $p\in[1,\infty]$. Then there exists a constant $C=C(H,d,\alpha,p)$ such that for any bounded measurable function $f\colon\R^d\to\R$, $0\le s\le u \le t\le1$ one has
\begin{equation}
\|\E^u f(V_{s,t})\|_{L_p(\Omega)}\le C \|f\|_{\B^{\alpha}_p}(t-u)^{\alpha H}\bigl((u-s)^{-\frac{Hd}{p}}+\I(H<1/2)(u-s)^{-\frac{d}{2p}}(t-u)^{\frac{d}{2p}-\frac{Hd}p}\bigr),\label{efp}
\end{equation}
\end{enumerate}
\end{proposition}

\begin{proof}
(i). Follows from the decomposition 
\begin{equation*}
V_{s,t}=\E^u V_{s,t}+\int_u^t K_H(t,r)\,dB_r, 
\end{equation*}
and independence of $\E^u V_{s,t}$ and $\int_u^t K_H(t,r)\,dB_r$.

(ii). If $H<\frac12$, then $(r-s)^{H-\frac12}\ge(t-s)^{H-\frac12}$ for $s\le r \le t$. Therefore, \eqref{Hg12} implies for $0\le s \le t$
\begin{align*}
K_H(t,s)&=C\Bigl(t^{H-\frac12}s^{\frac12-H}(t-s)^{H-\frac12}+(\frac12-H)s^{\frac12-H}\int_s^t (r-s)^{H-\frac12}r^{H-\frac32}\,dr\Bigr)\\
&\ge C (t-s)^{H-\frac12}\bigl(t^{H-\frac12}s^{\frac12-H}+s^{\frac12-H}(s^{H-\frac12}-t^{H-\frac12})\bigr)\\
&=C(t-s)^{H-\frac12},	
\end{align*}	
where $C=C(H,d)>0$. 

If $H>\frac12$, then $r^{H-1/2}\ge s^{H-1/2}$ for $r\ge s$ and \eqref{Hl12} yields
\begin{equation*}
	K_H(t,s)=C s^{\frac12-H}\int_s^t (r-s)^{H-\frac32}r^{H-\frac12}\,dr\ge 
	C \int_s^t (r-s)^{H-\frac32}\,dr=C (t-s)^{H-\frac12},
\end{equation*}	
where $C=C(H,d)>0$.
Thus, in both cases, $K_H(t,s)\ge C (t-s)^{H-\frac12}$. This directly yields  $\sigma^2(s,t)\ge C (t-s)^{2H}$.

Part (iii) follows from the identity $\E^uV_{s,t}=\int_s^u K_H(t,r)\,dB_r$ and part (ii) of the proposition.

(iv). Using parts (i), (ii) and (iii) of the proposition together with \eqref{modinfF}, we derive
\begin{align}\label{esthard}
\|\E^u f(V_{s,t})\|_{L_p(\Omega)}^p&=\|P_{\sigma^2(u,t)} f(\E^u V_{s,t})\|_{L_p(\Omega)}^p\nn\\
&=\int_{\R^d} p_{\Var[\E^u V^1_{s,t}]}(x)|P_{\sigma^2(u,t)} f(x)|^p\,dx\nn\\
&\le \|p_{\Var[\E^u V_{s,t}^1]}\|_{L_\infty(\R^d)}\|P_{\sigma^2(u,t)} f\|^p_{L_p(\R^d)}\nn\\
&\le C ((t-s)^{2H}-(t-u)^{2H})^{-\frac{d}2}(t-u)^{\alpha pH}\|f\|_{\B^\alpha_p}^p,
\end{align}
where $C=C(H,d,\alpha,p)>0$.
If $H<1/2$, then 
\begin{equation*}
(t-s)^{2H}-(t-u)^{2H}\ge C (u-s)(t-s)^{2H-1}
\end{equation*}
for $C=C(H)>0$. 
We continue \eqref{esthard} using this bound together with the inequality $(a+b)^\rho\le C a^\rho+C b^\rho$ valid for all $a,b,\rho>0$ with $C=C(\rho)$ independent of $a,b$. We get
\begin{align*}
\|\E^u f(V_{s,t})\|_{L_p(\Omega)}^p&\le C\|f\|_{\B^\alpha_p}^p(u-s)^{-\frac{d}2}(t-s)^{\frac{d}2-Hd} (t-u)^{\alpha pH}\\
&\le  C\|f\|_{\B^\alpha_p}^p (u-s)^{-Hd}(t-u)^{\alpha pH} \\
&\phantom{\le}+
C\|f\|_{\B^\alpha_p}^p (u-s)^{-\frac{d}2}(t-u)^{\alpha pH+\frac{d}2-Hd} 
\end{align*}
where  $C=C(H,d,\alpha,p)>0$ and we used the fact that $d/2-Hd>0$. This implies \eqref{efp}.

If $H\ge1/2$, then $2H\ge1$ and thus
\begin{equation*}
(t-s)^{2H}-(t-u)^{2H}\ge (u-s)^{2H}.
\end{equation*}
Substituting this into \eqref{esthard}, we immediately get \eqref{efp}.
\end{proof}

For $H\in(0,1)$, following \cite[formulas (18), (23), (38)]{picard}, consider now a functional
\begin{equation*}
\Phi:=\wt \Pi^{H-\frac12}I^{\frac12-H}\wt \Pi^{\frac12-H},
\end{equation*}
where for a measurable function $f\colon[0,1]\to\R$ we put 
\begin{align*}
&\wt\Pi^\alpha f(t):=t^\alpha f(t)-\alpha \int_0^t s^{\alpha-1}f(s)\,ds,\quad t\in[0,1],\, \alpha\in\R\\
&I^{\alpha}f(t):=C(\alpha) \int_0^t (t-s)^{\alpha-1}f(s)\,ds, \quad t\in[0,1],\, \alpha>0; \quad I^{\alpha}f:=\frac{d}{dt} I^{\alpha+1}f,\quad \alpha\in(-1,0].
\end{align*}
Recall the definition of functional $\Psi$ in \eqref{WB} and recall that the space $\C_0^\theta([0,1])$ is the space of all functions in $\C^\theta([0,1])$ which are zero at zero.
\begin{proposition}\label{p:cont} Let $d\in\N$, $H\in(0,1)$. Let $W^H$ be a fractional Brownian motion with Hurst index $H$. Then the following holds:
	\begin{enumerate}[\rm{(}i\rm{)}]
		\item the Brownian motion $B$ in representation  \eqref{WB} is given by $B=\Phi(W^H)$;
		\item we have $\Psi\circ\Phi(W^H)=W^H$;
		\item if  $H\in(0,\frac12)$, then   $\Phi$ is a  continuous functional $\C([0,1])\to \C([0,1])$;		
		\item if $H\in(\frac12,1)$, $\theta>H-\frac12$ then $\Psi$ is a  continuous functional $\C^\theta_0([0,1])\to \C([0,1])$.
	\end{enumerate}
	
\end{proposition}
\begin{proof}
(i), (ii) is \cite[Theorem~11]{picard}. (iii) is \cite[Lemma~7.2]{ART21}.

(iv).  It follows from \cite[Theorem~11]{picard}, that $\Psi=\wt \Pi^{H-\frac12}I^{H-\frac12}\wt \Pi^{\frac12-H}$. Let $f\in\C^\theta_0([0,1])$. Then for $t\in[0,1]$ we have $\Psi f(t)=I_1(t)+I_2(t)+I_3(t)+I_4(t)$, where
\begin{align*}
&I_1(t)=C t^{H-\frac12}\int_0^t (t-s)^{H-\frac32}s^{\frac12-H}f(s)\,ds;\\
&I_2(t)=C t^{H-\frac12}\int_0^t (t-s)^{H-\frac32}\int_0^s s_1^{-\frac12-H}f(s_1)\,ds_1ds;\\
&I_3(t)=C \int_0^t s^{H-\frac32}\int_0^{s}(s-s_1)^{H-\frac32} s_1^{\frac12-H}f(s_1)\,ds_1ds;\\
&I_4(t)=C\int_0^t s^{H-\frac32}\int_0^{s}(s-s_1)^{H-\frac32}\int_0^{s_1} s_2^{-\frac12-H}f(s_2)\,ds_2ds_1ds
\end{align*}	
for $C=C(H)$.
It is easy to see that $\|I_1\|_{L_\infty([0,1])}\le \|f\|_{L_\infty([0,1])}$. Since $f\in\C_0^\theta$, we have $|f(s_1)|\le \|f\|_{\C^\theta([0,1])} s_1^\theta$. Using that $\theta>H-\frac12$, we see that $-\frac12-H+\theta>-1$ and we get $\|I_2\|_{L_\infty([0,1])}\le \|f\|_{\C^\theta([0,1])}$. Clearly,  $\|I_3\|_{L_\infty([0,1])}\le \|f\|_{L_\infty([0,1])}$. Finally,
\begin{align*}
\|I_4\|_{L_\infty([0,1])}&\le C\|f\|_{\C^\theta([0,1])} \int_0^1 s^{H-\frac32}\int_0^{s}(s-s_1)^{H-\frac32}s_1^{\frac12-H+\theta}\,ds_1ds\\
&\le C\|f\|_{\C^\theta([0,1])} \int_0^1 s^{H-\frac32+\theta}\,ds\le C\|f\|_{\C^\theta([0,1])}
\end{align*}
for $C=C(H,\theta)$. Here at the very first step we used that $-\frac12-H+\theta>-1$. Combining this with the previous bounds, we get $\|\Psi f\|_{L_\infty([0,1])}\le C\|f\|_{\C^\theta([0,1])}$, which completes the proof.
\end{proof}

\section{Miscellaneous}

\begin{proposition}\label{p:KCT}[Kolmogorov continuity theorem]
	Let $d\in\N$, $\gamma_1, \gamma_2, \gamma_3,\gamma_4>0$, $\gamma_1\ge\gamma_2$.
	Let $X^n\colon\Omega\times[0,1]\times\R^d\to\R^d$, $n\in\Z_+$ be a sequence of measurable processes such that $X^n(0,x)=0$ for any $x\in\R^d$, $n\in\N$. Suppose that for any fixed  $\omega\in\Omega$, $x\in\R^d$, $n\in\Z_+$ the function $t\mapsto X^n(\omega,t,x)$ is continuous. Assume further that for any $m\ge1$ there exists a constant $C=C(d,\gamma_1,\gamma_2,\gamma_3,\gamma_4,m)$ such that the following bounds holds for any $s,t\in[0,1]$, $x,y\in\R^d$, $n,k\in\N$:
	\begin{align}
		&\|X^n(t,x)-X^n(s,x)\|_{L_m(\Omega)}\le C |t-s|^{\gamma_1};\label{con1}\\
		&\|X^n(t,x)-X^n(s,x)-(X^n(t,y)-X^n(s,y))\|_{L_m(\Omega)}\le C |t-s|^{\gamma_2}|x-y|^{\gamma_3};\label{con2}\\
		&\|X^n(t,x)-X^k(t,x)\|_{L_m(\Omega)}\le C(n\wedge k)^{-\gamma_4}\label{conmul}.
	\end{align} 
	
	Then there exists  a measurable process $X\colon\Omega\times[0,1]\times\R^d\to\R^d$ and a set of full probability measure $\Omega'\subset \Omega$ such that on $\Omega'$ the following holds:
	\begin{enumerate}[\rm{(}i\rm{)}]
		\item $X$ is jointly continuous in $(t,x)$. Further, for any $\eps>0$, $M>0$ we have 
		\begin{align}\label{contModx}
			&\sup_{\substack{x\in\R^d\\|x|\le M}}\sup_{s,t\in[0,1]}\frac{|X(t,x)-X(s,x)|}{|t-s|^{\gamma_1-\eps}}<\infty,\\
			&
			\sup_{\substack{x,y\in\R^d\\|x|,|y|\le M}}\sup_{s,t\in[0,1]}\frac{|X(t,x)-X(s,x)-(X(t,y)-X(s,y))|}{|t-s|^{\gamma_2-\eps}|x-y|^{\gamma_3-\eps}}<\infty;\label{contModx2}
		\end{align}
		
		\item for any $n\in\N$ process $X^n$ has a continuous modification $\wt X^n$. Further, for any $\eps>0$, $M>0$
		\begin{equation}\label{convkolm}
			\sup_{\substack{x,y\in\R^d\\|x|,|y|\le M}}\sup_{s,t\in[0,1]}\frac{|(X(t,x)-\wt X^{n}(t,x))-(X(s,y)-\wt X^{n}(s,y))|}{|t-s|^{\gamma_1-\eps}+|x-y|^{\gamma_3-\eps}}\to0, \quad\text{ as $n\to\infty$;} 
		\end{equation}
		
		\item for any $\omega\in\Omega'$ there exists a set $A(\omega)\subset \R^d$ of Lebesgue measure $0$ such that for any $M>0$, $t\in[0,1]$ we have
		\begin{equation}\label{limpart}
			\sup_{\substack{x\in\R^d\setminus A(\omega)\\|x|\le M}}|X(\omega,t,x)- X^{n}(\omega,t,x)|\to0, \quad\text{ as $n\to\infty$.}
		\end{equation}
	\end{enumerate}
\end{proposition}
\begin{proof}
(i). Fix $\eps>0$. We see from \eqref{conmul} that for any fixed $t\in[0,1]$, $x\in\R^d$ the sequence $(X^n(t,x))_{n\in\N}$  is Cauchy in $L_m(\Omega)$. Therefore, as $n\to\infty$,  this sequence converges in $L_m(\Omega)$ to a limit which we will denote by $\wt X(t,x)$.
An application of Fatou's lemma allows to pass to the limit as $n\to\infty$ in \eqref{con1} and \eqref{con2}. We derive that for any $m\ge1$ there exists a
constant $C=C(d,\gamma_1,\gamma_2,\gamma_3,\gamma_4,m)>0$ such that for any $x,y\in\R^d$, $s,t\in[0,1]$
	\begin{align}
		&\|\wt X(t,x)-\wt X(s,x)\|_{L_m(\Omega)}\le C |t-s|^{\gamma_1};\label{limKolmbox1}\\
		&\|\wt X(t,x)-\wt X(s,x)-(\wt X(t,y)-\wt X(s,y))\|_{L_m(\Omega)}\le C |t-s|^{\gamma_2}|x-y|^{\gamma_3}.\label{limKolmbox2}\\
		&\|\wt X(t,x)-\wt X(t,y)\|_{L_m(\Omega)}\le C |x-y|^{\gamma_3},\label{limKolmbox3}
	\end{align}
where the last inequality follows from \eqref{limKolmbox2} by taking there $s=0$. Take now $m$ large enough so that $d/m<\eps/2$. Recalling that $\gamma_1\ge\gamma_2$, we see that conditions (9), (10), (11) of Kolmogorov--Chentsov continuity theorem for random fields \cite[Theorem~1.4.4]{Kunita} are satisfied. Therefore $\wt X$ has a continuous modification $X$ for which \eqref{contModx2} holds. 
	
Adding now \eqref{limKolmbox1} and \eqref{limKolmbox3}, we get
that for any $x,y\in\R^d$, $s,t\in[0,1]$
\begin{equation*}
\|\wt X(t,x)-\wt X(s,y)\|_{L_m(\Omega)}\le C |t-s|^{\gamma_1}+C|x-y|^{\gamma_3},
\end{equation*}
where $C=C(d,\gamma_1,\gamma_2,\gamma_3,\gamma_4,m)>0$.
Choose now $m$ large enough so that 
$(1+d\frac{\gamma_1}{\gamma_3})/m<\frac\eps2$. 
Then condition (1) of standard Kolmogorov continuity theorem \cite[Theorem~1.4.1]{Kunita} is met, which implies \eqref{contModx}.
	
	(ii).  First, note that, as in part (i), it follows from \eqref{con1}–\eqref{con2} and the Kolmogorov--Chentsov continuity theorem \cite[Theorem~1.4.4]{Kunita} that for any $n \in \Z_+$, the process $X^n$ has a continuous modification $\wt X^n$.
Let $n, k \in \N$ and $m \ge 1$. Applying \eqref{con1} and taking $s = 0$ in \eqref{con2}, we obtain for any $s, t \in [0, 1]$ and $x, y \in \R^d$
	\begin{align}
		&\|X^n(t,x)-X^k(t,x)-(X^n(s,y)-X^k(s,y))\|_{L_m(\Omega)}\nn\\
		&\quad\le 
		\|X^n(t,x)-X^n(s,y)\|_{L_m(\Omega)}+\|(X^k(t,x)-X^k(s,y))\|_{L_m(\Omega)}\nn\\
		&\quad\le C |t-s|^{\gamma_1}+C|x-y|^{\gamma_3}\label{proof2step1}
	\end{align}	
	where $C=C(d,\gamma_1,\gamma_2,\gamma_3,\gamma_4, m)>0$.
	On the other hand, by grouping the terms in the above inequality differently and using \eqref{conmul}, we see that
	\begin{equation*}
		\|X^n(t,x)-X^k(t,x)-(X^n(s,y)-X^k(s,y))\|_{L_m(\Omega)}\le C(n\wedge k)^{-\gamma_4},
	\end{equation*}	
	for $C=C(d,\gamma_1,\gamma_2,\gamma_3,\gamma_4,m)$.
	Hence, combining it with \eqref{proof2step1}, we get that for any $\eps\in(0,1)$, $m\ge1$ there exists a constant $C=C(d,\gamma_1,\gamma_2,\gamma_3,\gamma_4,m)>0$ such that for any $s, t \in [0,1]$ and $x, y \in \R^d$
	\begin{equation*}
		\|X^n(t,x)-X^k(t,x)-(X^n(s,y)-X^k(s,y))\|_{L_m(\Omega)}\le C(n\wedge k)^{-\eps\gamma_4}(|t-s|^{\gamma_1(1-\eps)}+|x-y|^{\gamma_3(1-\eps)}).
	\end{equation*}	
 Let us pass to the  limit in the above inequality for fixed $n$ as $k\to\infty$. By Fatou's lemma we have 
	\begin{equation*}
		\|(X^n(t,x)-X(t,x))-(X^n(s,y)-X(s,y))\|_{L_m(\Omega)}\le C  n^{-\eps\gamma_4}(|t-s|^{\gamma_1(1-\eps)}+|x-y|^{\gamma_3(1-\eps)}),
	\end{equation*}
Applying again the Kolmogorov continuity theorem  \cite[Theorem~1.4.1]{Kunita} and choosing $m>(\eps\gamma_4)^{-1}$ large enough such that, we see that for any $M>0$, $n\in\N$ we have
	\begin{equation*}
		\Bigl\|
		\sup_{\substack{x\in\R^d\\|x|\le M}}\sup_{s,t\in[0,1]}\frac{|(X(t,x)-\wt X^{n}(t,x))-(X(s,x)-\wt X^{n}(s,y))|}{|t-s|^{\gamma_1-2\eps}+|x-y|^{\gamma_3-2\eps}}\Bigr\|_{L_m(\Omega)}\le C n^{-\eps\gamma_4}
	\end{equation*}
for $C=C(d,\gamma_1,\gamma_2,\gamma_3,\gamma_4,\eps,m)$.
	By the standard Borel-Cantelli arguments (see, e.g., \cite[Chapter~2.10, Corollary~2, p.~309]{ShiryaevProbab}), this implies \eqref{convkolm}.
	
	(iii). By part (ii) of the lemma, for any $n\in\N$, $t\in\mathbb{Q}\cap[0,1]$
	$$
	\E \int_{\R^d} \I\bigl(X^n(t,x)\neq \wt X^n(t,x)\bigr)\,dx=
	\int_{\R^d} \P\bigl(X^n(t,x)\neq \wt X^n(t,x)\bigr)\,dx=0. 
	$$ 
	Hence there exists a set $\Omega''\subset\Omega'$ of full probability measure and a set $A(\omega)\subset\R^d$ of Lebesgue measure zero such that for any $\omega\in\Omega''$ we have
	\begin{equation}\label{almostend}
		X^n(\omega, t,x)=\wt X^n(\omega, t,x)\quad \text{ for any $n\in\N$, $t\in\mathbb{Q}\cap[0,1]$, $x\in\R^d\setminus A(\omega)$}.
	\end{equation}
	Note that for any fixed $x\in\R^d$, $n\in\N$, $\omega\in\Omega''$ the processes $X^n(\omega,\cdot,x)$ and  $\wt X^n(\omega,\cdot,x)$ are continuous: the first one by assumptions of the proposition, the second one by construction. Hence \eqref{almostend} implies that $X^n(\omega,t,x)=\wt X^n(\omega,t,x)$ for any $\omega\in\Omega''$, $t\in[0,1]$, $n\in\N$, $x\in\R^d\setminus A(\omega)$.  By taking in \eqref{convkolm} $s=0$,  $y=x$, we get \eqref{limpart}.
\end{proof}

\begin{proposition}\label{p:localtime}
Let $\mu$ be a finite measure on $\R^d$, $d \in \N$. Suppose there exists a set $A \subset \R^d$ of Lebesgue measure $0$ and a continuous function $\ell \colon \R^d \to \R_+$ such that
	\begin{equation}\label{localtimecond}
		\lim_{k\to\infty} \sup_{x\in\R^d\setminus A }\Bigl|\frac{\mu(\Ba(x,2^{-k}))}{v_d2^{-kd}}-\ell(x)\Bigr|=0.
	\end{equation} 
Then the measure $\mu$ is absolutely continuous with respect to the Lebesgue measure, and its Radon–Nikodym derivative is given by $d\mu / d\Leb = \ell$.
\end{proposition}
Before we begin the proof, let us note that the supremum in \eqref{localtimecond} is crucial. Indeed, if $\mu = \delta_0$, then for any $x \in \R^d \setminus {0}$, we obviously have
\begin{equation*}
	\lim_{k\to\infty} \frac{\delta_0(\Ba(x,2^{-k}))}{v_d2^{-kd}}=0,
\end{equation*}
so an analogue of \eqref{localtimecond} without the supremum holds for $\ell(x) \equiv 0$, but $\delta_0$ is not absolutely continuous with respect to the Lebesgue measure.
\begin{proof}
\textbf{Step 1}. 
We show that $\mu$ is absolutely continuous with respect to the Lebesgue measure. Fix $\eps > 0$, and take $K = K(\eps) > 0$ large enough so that
$\mu(\R^d \setminus \Ba(0, K)) < \eps$. Since $\ell$ is continuous, it is bounded on $\Ba(0, K)$. Therefore, recalling \eqref{localtimecond}, we have
\begin{equation}\label{firstlim}
\sup_{\substack{x\in \Ba(0,K)\setminus A\\k\in\Z_+}}
\frac{\mu(\Ba(x,2^{-k}))}{v_d2^{-kd}} \le M,
\end{equation}	
for some constant $M=M(K)$. Since $\Leb(A)=0$, we see that for any $x\in \Ba(0,K)$, $r\in(0,1/2)$ there exists $\wt x\in \Ba(0,K)\setminus A$, such that 
$\Ba(x,r)\subset \Ba(\wt x,2r)\subset \Ba(\wt x,2^{-k})$, where $k\in\Z_+$ is such that $2^{-k-1}< 2r\le 2^{-k}$. Therefore, applying \eqref{firstlim}, we get 
\begin{equation}\label{ballmeasure}
\sup_{x\in \Ba(0,K),\, r\in(0,1]}\frac{\mu(\Ba(x,r))}{v_dr^{d}}\le 4^dM.
\end{equation}	
Let now $B\subset\R^d$ be an arbitrary set of zero Lebesgue measure. Then a set $B\cap\Ba(0,K)$ can be covered by a countable union of balls $ \Ba(x_i,r_i)$ with $\sum_{i=1}^\infty \nu_dr_i^d<\frac\eps{4^dM}$ and $x_i\in \Ba(0,K)$. Applying \eqref{ballmeasure}, we deduce
\begin{equation*}
\mu(B)\le\mu(B\cap \Ba(0,K))+\mu(\R^d\setminus \Ba(0,K))\le \sum_{i=1}^\infty \mu (\Ba(x_i,r_i))+\eps\le 2\eps.
\end{equation*}		
Since $\eps>0$ was arbitrary, we see that $\mu(B)=0$ and thus $\mu\ll \Leb$.

\textbf{Step~2}. Since $\mu \ll \Leb$, by the Radon–Nikodym theorem, for any set $S \in \mathscr{B}(\R^d)$, we have $\mu(S) = \int_S \wt \ell(x)\,dx$ for some measurable function $\wt \ell \colon \R^d \to \R$. Then, the Lebesgue differentiation theorem implies that for $\Leb$-a.e. $x \in \R^d$,
$$
\lim_{k\to\infty} \frac{\mu(\Ba(x,2^{-k}))}{v_d2^{-kd}}=\wt \ell(x).
$$
Comparing this with  \eqref{localtimecond}, we see that $\wt \ell = \ell$ Lebesgue a.e., and thus $d\mu / d\Leb = \ell$.
\end{proof}

\begin{proposition}\label{p:limitlt}
	Let $\mu$ be a finite measure on $\R^d$. Suppose that $\mu\ll \Leb$ and that Radon-Nikodym derivative $\ell :=d\mu/d \Leb$ is continuous. Then for any $x\in\R^d$
	\begin{equation*}%\label{localtimeres}
		\ell(x)=\lim_{\eps\to0} \frac{\mu(\Ba(x,\eps))}{v_d\eps^{d}}.
	\end{equation*} 
\begin{proof}
Follows immediately from continuity of $\ell$ and the fact that	 for any $\eps>0$, $x\in\R^d$
\begin{equation*}
 \inf_{y\in\Ba(x,\eps)} \ell(y)\le \frac{\mu(\Ba(x,\eps))}{v_d\eps^{d}}\le \sup_{y\in\Ba(x,\eps)} \ell(y).\qedhere
\end{equation*}
\end{proof}
\end{proposition}

\begin{proposition}[{\cite[Lemma~5.12]{FV2010}}]\label{p:TV}
	Let $\kappa\ge1$, $X^n\colon[0,1]\to\R^d$ be a sequence of  functions of finite $\kappa$-variation converging to a function $X\colon[0,1]\to\R^d$ pointwise as $n\to\infty$. Then $X$ is also of finite $\kappa$-variation and 
	$$
	\|X\|_{\kappa-\var;[0,1]}\le \liminf_{n\to\infty} \|X^n\|_{\kappa-\var;[0,1]}.
	$$
\end{proposition}

\end{document}